\documentclass[11pt,reqno]{amsart}
\usepackage[T1]{fontenc}
\usepackage[utf8]{inputenc}
\usepackage{lmodern}
\usepackage{amsmath,amssymb,amsthm,mathtools}
\usepackage{booktabs,array,longtable}
\usepackage[margin=1.1in]{geometry}
\usepackage{microtype}
\usepackage{xcolor}
\definecolor{linkcol}{rgb}{0.10,0.20,0.45}
\usepackage[colorlinks=true,linkcolor=linkcol,citecolor=linkcol,
            urlcolor=linkcol]{hyperref}

\pdftrailerid{}

\theoremstyle{plain}
\newtheorem{thm}{Theorem}[section]
\newtheorem{lem}[thm]{Lemma}
\newtheorem{prop}[thm]{Proposition}
\newtheorem{cor}[thm]{Corollary}
\theoremstyle{definition}
\newtheorem{dfn}[thm]{Definition}
\theoremstyle{remark}
\newtheorem{rem}[thm]{Remark}
\theoremstyle{plain}
\newtheorem*{thmA}{Theorem A}
\newtheorem*{thmB}{Theorem B}
\newtheorem*{thmC}{Theorem C}
\newtheorem*{thmD}{Theorem D}
\theoremstyle{definition}
\theoremstyle{plain}
\newtheorem*{lembdyprime}{Lemma 3.1$'$}
\newtheorem*{lemlinkprime}{Lemma 8.3$'$}

\newcommand{\IS}{\mathrm{IS}}
\newcommand{\MS}{\mathrm{MS}}
\DeclareMathOperator{\Int}{Int}
\DeclareMathOperator{\Area}{Area}
\DeclareMathOperator{\dist}{dist}
\DeclareMathOperator{\lk}{lk}

\DeclareMathOperator{\Hess}{Hess}
\newcommand{\bd}{\partial}
\newcommand{\R}{\mathbb{R}}
\newcommand{\Z}{\mathbb{Z}}
\newcommand{\N}{\mathbb{N}}
\newcommand{\cF}{\mathcal{F}}

\newcommand{\cD}{\mathcal{D}}

\newcommand{\cN}{\mathcal{N}}

\newcommand{\cI}{\mathcal{I}}
\newcommand{\cM}{\mathcal{M}}
\newcommand{\cG}{\mathcal{G}}
\newcommand{\cB}{\mathcal{B}}
\newcommand{\cA}{\mathcal{A}}
\newcommand{\cK}{\mathcal{K}}

\allowdisplaybreaks
\title[Exchange complexes and contractibility]
      {Exchange complexes and contractibility\\
       of the complex of incompressible Seifert surfaces}
\author[G. Pan]{Guancheng Pan}
\author[C. You]{Chengsong You}
\author[J. Zhou]{Junwei Zhou}
\author[Y. Chen]{Yongchao Chen}

\newcommand{\afn}[1]{\textsuperscript{\normalfont #1}}
\makeatletter
\renewcommand{\@setauthors}{%
  \begin{center}
    \normalsize
    \begin{tabular}{c@{\hspace{2.2em}}c}
      \large Guancheng Pan\afn{1,2} & \large Chengsong You\afn{1,3}
    \end{tabular}\\[0.5ex]
    \begin{tabular}{c@{\hspace{2.2em}}c}
      \large Junwei Zhou\afn{4,\dag} & \large Yongchao Chen\afn{1,5,\dag}
    \end{tabular}\\[1.4ex]
    \small
    \textsuperscript{1}Apex Intelligence \quad
    \textsuperscript{2}Chu Kochen Honors College, Zhejiang University\\
    \textsuperscript{3}East China Normal University \quad
    \textsuperscript{4}Independent Researcher \quad
    \textsuperscript{5}Tsinghua University\\[0.8ex]
    \texttt{3250101086@zju.edu.cn}, \texttt{youchengsong3@gmail.com}\\
    \texttt{zjw330501@gmail.com}, \texttt{cyc@apexin.ai}\\[0.6ex]
    \textsuperscript{\dag}Corresponding authors
  \end{center}}
\makeatother

\date{12 September 2026}

\begin{document}

\begin{abstract}
Let $K\subset S^3$ be a non-trivial knot and let $\IS(K)$ be the simplicial
complex whose vertices are the ambient isotopy classes of incompressible
Seifert surfaces in the exterior $E(K)$, a finite set of distinct vertices
spanning a simplex exactly when its classes admit simultaneously pairwise
disjoint representatives.  Kakimizu proved that $\IS(K)$ is connected; whether
it is contractible was asked by Przytycki and Schultens, who also identified
the obstruction, namely that the projection used in the connectedness proof is
not known to be well defined on isotopy classes.  We prove that $\IS(K)$ is
contractible.  We also prove contractibility of every genus truncation
$\IS_\ell(K)$ with $\ell\ge g(K)$ and every non-empty lexicographic complexity
sublevel.  The argument factors through an unconditional combinatorial
theorem: every non-empty connected flag
\emph{exchange complex} is
contractible, where an exchange complex carries a complexity function subject
to two axioms which require only that an exchanging vertex \emph{exist}, never
that a coherent selection rule be supplied.  The combinatorial proof uses
hereditary descending links and transfinite induction.  We compare its
attachment step with existing Morse criteria, including a formulation
applicable to every countable exchange complex, and distinguish the given
exchange order from a dismantling order (\S\ref{ssec:F-discussion}).  The
geometric argument also applies to links satisfying a linking condition that
forces every spanning surface to be connected.  For the geometric input we
prove fixed-boundary area attainment among smooth neat embeddings, using
smooth convex replacement domains, area-controlled disc cleaning and
boundary-preserving smoothing.  All area complexities use this smooth
competing class.  No interpretation of an undefined piecewise smooth
existence class is needed.
\end{abstract}

\maketitle

\begin{center}\small
\end{center}

\section{Introduction}\label{sec:intro}

\subsection{The problem}

Let $K\subset S^3$ be a knot with exterior $E(K)=S^3\setminus\Int N(K)$.  A
\emph{Seifert surface} for $K$ is a compact connected orientable surface
properly embedded in $E(K)$ whose boundary is a longitude of $K$; it is
\emph{incompressible} if the inclusion induces an injection on fundamental
groups.  Kakimizu \cite{Kak92} introduced two simplicial complexes attached to
$K$: the complex $\IS(K)$ whose vertices are the ambient isotopy classes of
incompressible Seifert surfaces, and the full subcomplex $\MS(K)$ on the
classes of minimal genus.  In both, a finite set of distinct vertices spans a
simplex exactly when its classes admit simultaneously pairwise disjoint
representatives.

Kakimizu proved that these complexes are connected.  We shall use his statement
in the form in which he printed it \cite[Theorem A, p.~226]{Kak92}:

\begin{quote}
\emph{Let $L$ be a non-split oriented link.  Then both $\IS(L)$ and $\MS(L)$
are connected.}
\end{quote}

The two cases are not proved by a common argument: the proof of the underlying
engine \cite[Theorem 2.1, p.~228]{Kak92} splits on p.~229 into a minimal-genus
case and an incompressible case, the latter run through McMillan's simple
moves.  This matters here, because it is the incompressible half that we need
and it is the half that is easy to lose sight of when the result is quoted for
$\MS$ alone.  We also use Kakimizu's sharper statement
\cite[Proposition 3.1(1), p.~231]{Kak92} that the edge-path metric of $\IS(L)$
coincides with the distance $d$ defined through the infinite cyclic cover.

Whether $\IS(K)$ is contractible was asked by Przytycki and Schultens
\cite{PS12}.  We attribute the question to them rather than to
\cite{Kak92}, where it does not appear.  Przytycki and Schultens proved that
$\MS(K)$ is contractible, and in the same paper isolated the reason their
method does not reach $\IS(K)$: their projection is defined by choosing, for a
surface $S$ and a vertex $v$, a distinguished representative of $v$ disjoint
from $S$, and for incompressible surfaces of unbounded genus that choice is
not known to descend to isotopy classes.

\subsection{Results}

All area infima in the main results are taken over smooth neat embeddings
in a smooth ambient isotopy class.  Here neat means transverse to the
ambient boundary.  Theorem~\ref{thm:fixed-smooth} proves fixed-boundary
attainment in this class for the metric chosen in \S\ref{ssec:data}.
Kapovich's compactness theorem then gives attainment when the boundary
slides in the longitudinal foliation.  Piecewise smooth surfaces occur as
intermediate cut-and-paste objects; their smoothing is verified before an
area comparison is made.  We do not identify the smooth infimum with the
possibly smaller infimum over all surfaces of Definition~\ref{conv:ps}.

\begin{thmA}\phantomsection\label{thm:A}
Let $K\subset S^3$ be a non-trivial knot.  Then $\IS(K)$ is contractible.
\end{thmA}

The non-triviality hypothesis excludes a case that is settled and trivial:
$\IS(\mathrm{unknot})$ is a single point, by Proposition~\ref{prop:connected},
hence contractible.  We nevertheless keep the exclusion, because every
statement of \S\ref{sec:exchange-lemma} is proved for a non-trivial knot and
we prefer not to carry a hypothesis in one place and drop it in another.

The proof factors through a combinatorial statement that involves no
three-dimensional topology.

\begin{dfn}[Exchange complex]\label{dfn:intro-exchange}
Let $X$ be a connected flag simplicial complex, let $W$ be a well-ordered set,
and let $c\colon X^{(0)}\to W$.  For vertices $u,w$ at distance $2$ in the
$1$-skeleton write $\cI(u,w)$ for the set of common neighbours of $u$ and $w$.
Call $x\in\cI(u,w)$ an \emph{apex} if $x$ is equal or adjacent to every member
of $\cI(u,w)$; equivalently, the full subcomplex on $\cI(u,w)$ is a cone with
apex $x$.  The pair $(X,c)$ is an \emph{exchange complex} if
\begin{itemize}
\item[\textup{(N1)}] every two-interval has an apex; and
\item[\textup{(N2)}] whenever $\dist(u,w)=2$, \emph{some} apex $x$ of
      $\cI(u,w)$ satisfies $c(x)<\max\{c(u),c(w)\}$.
\end{itemize}
This is Definition~\ref{dfn:exchange}, repeated verbatim.  Note that
$\cI(u,w)\neq\emptyset$ is automatic when $\dist(u,w)=2$ and is therefore not
an axiom; what (N1) asks is that the interval be a cone.
\end{dfn}

\begin{thmB}\phantomsection\label{thm:B}
Every non-empty connected flag exchange complex is contractible.
\end{thmB}

The axioms require the existence of a lower exchanging vertex without
a coherent global rule for selecting one.  This avoids the need for the
projection whose well-definedness on $\IS(K)$ is not known in
\cite{PS12}.  An equivariant version would require additional geometric
and fixed-point arguments, discussed in \S\ref{ssec:equivariant}; it is
not established here.

Two further results follow.  For the first, use the metric and longitudinal
foliation chosen in \S\ref{ssec:data}.

\begin{thmC}\phantomsection\label{thm:C}
Let $K\subset S^3$ be a non-trivial knot, and use the metric and longitudinal
foliation of \S\ref{ssec:data} to define relative area.  For every integer
$\ell\ge g(K)$ the truncation
$\IS_\ell(K)$, the full subcomplex of $\IS(K)$ on the vertices of genus at
most $\ell$, is isometrically embedded in $\IS(K)$ and is contractible.  The
same holds, whenever it is non-empty, for the full subcomplex on
\[
 D_{\ell,a_0}:=\{v:g(v)<\ell\}\ \cup
                 \{v:g(v)=\ell,\ A(v)\le a_0\},\qquad a_0>0.
\]
Here $A(v)$ denotes relative area as in Definition~\ref{dfn:complexity}.
\end{thmC}

Taking $\ell=g(K)$ gives the contractibility of
$\MS(K)$, already proved by \cite[Theorem 1.1, p.~1490]{PS12},
a statement made there for spanning surfaces of minimal
Thurston norm, which for a knot in $S^3$ is the minimal-genus condition.  The
second family in Theorem~\hyperref[thm:C]{C} refines the genus filtration
using relative area.  The connectedness of $\IS_\ell(K)$ is part
of Theorem~\hyperref[thm:C]{C}, obtained there from
Corollary~\ref{cor:initial-isometric}; it has an independent source in Chen and
Shen \cite[Proposition 5.15, p.~9 of v4]{ChenShen}, which we do not use.

The second extension concerns links, and requires a condition.  Call a link
$L=L_1\cup\dots\cup L_n$ \emph{linking-indecomposable} if for every partition
$\{1,\dots,n\}=I\sqcup I^c$ into two non-empty parts there is an index $m$ with
$\lk(L_m,L_J)\neq 0$, where $L_J$ is the sublink indexed by the part
\emph{not} containing $m$.  The condition is vacuous for $n=1$.  Its role is
to force every spanning surface to be connected, so that Kakimizu's vertices
--- which are allowed to be disconnected, and are required only to be
incompressible componentwise --- coincide with ours.

\begin{thmD}\phantomsection\label{thm:D}
Let $L$ be a non-split, linking-indecomposable link in $S^3$ which is not the
unknot.  Then $\IS(L)$ is contractible.
\end{thmD}

\subsection{Where the combinatorial theorem sits}\label{ssec:intro-position}

Theorem~\hyperref[thm:B]{B} is a statement about arbitrary flag complexes.  We
locate it against the three results with which it might be confused; the
details are in \S\ref{ssec:F-discussion}.

\emph{Dismantlability.}  Deleting a dominated vertex requires its current
link to be a cone.  The seven-vertex example of
\S\ref{ssec:not-dismantlable} has a contractible descending link at its
largest vertex which is not a cone.  Thus the given exchange order need
not be a dismantling order.  The example itself is dismantlable in a
different order; it does not exclude other dismantlability arguments.

\emph{The descending-link criterion.}  Neither the Morse criterion of
\cite[Lemma~2.3]{Zar24} nor its more general predecessor \cite{ZarMorse}
assumes local finiteness.  Write \textup{(F)} for finiteness of every
descending link.  Under \textup{(F)}, the maximum length of a descending
edge path gives a natural-number Morse height, for any cardinality of
the vertex set.  For every countable exchange complex, even without
\textup{(F)}, a bounded real encoding of the vertex well-order gives a
descending-type Morse function in the sense of \cite{ZarMorse}.
Corollaries~\ref{cor:under-F} and~\ref{cor:countable-morse} give both
arguments.  The exchange-specific steps are heredity of the descending
links and the induction proving them contractible; existing Morse theory
can supply the attachment step.  In particular \textup{(F)}, whose
status for $\IS(K)$ is unknown, is not the boundary of that comparison.
The direct proof below works for all cardinalities.

\emph{The ordering theorem.}  The complexity-convex geodesics and the
isometric embedding of the truncations are both special cases of
\cite[Lemma 9.13]{CCHO}, whose hypotheses are strictly weaker --- no apex, no
flagness, no local finiteness.  We nevertheless give the argument, as
Corollary~\ref{cor:initial-isometric}, so that \S\ref{sec:exchange-complexes}
is self-contained and \cite{CCHO} is a comparison rather than a link in the
chain leading to Theorems \hyperref[thm:A]{A}--\hyperref[thm:D]{D}.

\subsection{The shape of the argument}

The proof of Theorem~\hyperref[thm:A]{A} verifies that $\IS(K)$ is an exchange
complex for the complexity
\[
   c(v)=\bigl(g(v),\,A(v)\bigr)\in\N\times\R_{>0},
\]
ordered lexicographically, where $g(v)$ is the genus and $A(v)$ the infimum of
area over the class.  Axiom (N1) is supplied by three things together: the
connectedness of $\IS(K)$, which is Kakimizu's Theorem~A and is what makes a
two-interval non-empty; flagness; and
Proposition~\ref{prop:apex}(1), which manufactures an apex of the interval out
of the exchange lemma.  Axiom (N2) is the substance, and is supplied by an exchange lemma
for incompressible surfaces (\S\ref{sec:exchange-lemma}) proved by
minimal-surface methods: given two vertices at distance $2$, one produces a
common neighbour of strictly smaller complexity by cutting and recombining
area minimisers in the two classes along their intersection.

The geometric dependency is
\[
\begin{aligned}
 \text{Theorem~\ref{thm:fixed-smooth}}&\ \Longrightarrow\ \textup{(E)}(i),\\
 \textup{(E)}(i)+\text{compactness}+\textup{(R2)}
   &\ \Longrightarrow\ \textup{(R1)}\text{ and }\textup{(E)}(ii).
\end{aligned}
\]
Here \textup{(E)}(i) is fixed-boundary attainment, and
\textup{(E)}(ii) is attainment with boundary sliding in the foliation.
The boundary Hopf argument \textup{(R2)} applies to any smooth proper
minimal surface and does not depend on either attainment statement.
The fixed-boundary theorem supplies a smooth neat minimiser in the
prescribed relative isotopy class; Theorem~\ref{thm:R1} obtains a
sliding-boundary minimiser
using the compactness theorem of Kapovich
\cite[Proposition~A.1 and Corollary~A.3]{Schultens}.  Universal disjointness
\textup{(U)} is a separate cited input
\cite[Theorem~4 and Theorem~A.5]{Schultens}.  These inputs give the exchange
lemma and hence (N2); Theorem B then gives Theorems A, C and D.

Corollary~A.4 of \cite{Schultens} already states the sliding-boundary
conclusion, but its non-emptiness step invokes
\cite[Theorem~6.12]{HassScott}.  Our fixed-boundary proof uses individual
Plateau and disc-convergence results instead.  Smooth convex replacement
domains retain the fixed boundary arc, while a buffer permits smoothing of
the exterior splice.  Area-controlled disc cleaning preserves the relative
isotopy class, including the case where an intersection circle encloses the
original replacement disc.  Ordered smoothing preserves convergence through
the finite covering argument.  See \S\ref{ssec:fixed-smooth}.

The boundary transversality argument uses the lemma that a minimal graph over
a possibly curved plane sector
of opening angle at least $\pi$, lying in $M$ and touching $\bd M$ along one
edge, cannot be tangent to $\bd M$ at the vertex; that lemma also yields the
transversality statement \textup{(R2)} in three lines, in place of a
divergence-form comparison and an explicit barrier.  The gluing and final
boundary bootstrap also use elliptic regularity from \cite[Chs.~6 and~8]{GT}.
Proposition~\ref{prop:R1fixed} separately proves regularity of an attained
piecewise smooth minimum; it is not used to supply existence for the main
theorems.  The collar conditions \eqref{eq:N4} and
\eqref{eq:N5} of
\S\ref{sec:exchange-lemma} control the behaviour of the intersection
pattern at a boundary tangency.  Freedman, Hass and Scott note on
\cite[p.~636]{FHS} that non-transverse boundary behaviour occurs; they do not
supply a structure theorem for it, and \S\ref{ssec:tangential} does.

\subsection*{Organisation}

The paper falls into two halves that meet only at
Definition~\ref{dfn:exchange}.

The geometric half is \S\S\ref{sec:geometric-inputs}--\ref{sec:exchange-lemma}.
\S\ref{sec:geometric-inputs} fixes the metric, states the cited input
\textup{(U)}, proves smooth fixed-boundary attainment, and derives
\textup{(E)}, \textup{(R1)} and \textup{(R2)}.
\S\ref{sec:kakimizu-complex} pins down the
definitions of vertex, edge and simplex against \cite{Kak92}, and proves
flagness and the two compression lemmas.  \S\ref{sec:complexity} introduces the
complexity $c=(g,A)$, proves that the values it attains are well-ordered, and
gives the cyclic-cover construction that manufactures the candidate apices.
\S\ref{sec:exchange-lemma} is the exchange lemma itself, and is the longest
section: rounding and pushing off in \S\ref{ssec:rounding}, the transverse case
in \S\ref{ssec:exchange-thm}, and the two non-transverse configurations in
\S\S\ref{ssec:tangential}--\ref{ssec:T1}.

The combinatorial half is
\S\S\ref{sec:exchange-complexes}--\ref{sec:contractibility}, and uses no
three-dimensional topology at all: \S\ref{sec:exchange-complexes} sets out the
axioms and their heredity, \S\ref{sec:contractibility} proves
Theorem~\hyperref[thm:B]{B} and locates it against dismantlability and the
descending-link criterion.  \S\ref{sec:applications} puts the halves together
and proves Theorems \hyperref[thm:A]{A},
\hyperref[thm:C]{C} and \hyperref[thm:D]{D}, and then says where the method stops.

A reader who wants only the combinatorics may begin at
\S\ref{sec:exchange-complexes}; a reader who wants only the geometry may stop
after \S\ref{sec:exchange-lemma}, whose conclusion, axiom \textup{(N2)} for
$\IS(K)$, is Proposition~\ref{prop:apex}.

\subsection{Open questions}\label{ssec:open}

The smooth attainment theorem does not identify its infimum with the one
over every surface of Definition~\ref{conv:ps}.  An area-density theorem for
that larger class would give such an identification for classes with smooth
neat representatives.  This additional assertion is not used by
Theorems A--D; in particular the area sublevels in Theorem C use the smooth
definition of $A(v)$.

\begin{itemize}
\item We do not know whether $\IS(K)$ satisfies \textup{(F)}, equivalently
      whether a vertex can have infinitely many neighbours of smaller
      genus.  This is not needed by either the direct combinatorial proof
      or its countable Morse formulation.
\item This proof does not establish Theorem~\hyperref[thm:D]{D} for general
      links.  The failed step is identified in
      \S\ref{ssec:general-links}; no counterexample is claimed.
\item An equivariant version is not established.  The additional issues
      of invariant geometric data and the geometric fixed-point set are
      described in \S\ref{ssec:equivariant}.
\item This paper does not establish the required connectedness and
      geometric inputs for general Haken manifolds; see
      \S\ref{ssec:haken}.
\end{itemize}

\subsection*{Acknowledgements}

The mathematical content of this paper was developed with substantial
assistance from large language models, which were used for literature search,
for drafting, and for adversarial checking of the arguments; the authors are
responsible for all statements.  The status of previously reported Lean
fragments, and the absence of a formalisation of the main combinatorial
theorem, are stated in Remark~\ref{rem:lean}.

\section{Geometric inputs}\label{sec:geometric-inputs}

This section fixes the Riemannian data on $E(K)$ and assembles the facts about
area minimisers consumed by the exchange lemma.  Universal disjointness
\textup{(U)} is cited.  Theorem~\ref{thm:fixed-smooth} proves fixed-boundary
attainment \textup{(E)}(i) among smooth neat embeddings.  Kapovich's compactness
theorem then gives sliding-boundary attainment and \textup{(R1)};
Lemma~\ref{lem:R2} gives transversality.  Proposition~\ref{prop:R1fixed} is a
separate regularity implication for an attained piecewise smooth minimum.
The existence proof does not use that implication or sliding-boundary
attainment to establish its own input.

\subsection{The metric and the geometric statements}\label{ssec:data}

Choose on $E(K)$ a \emph{relative Riemannian metric}, that is, one making
$\bd E(K)$ strictly convex, with the following specified collar.  Choose a
flat metric on the boundary torus.  On a collar
$[0,\varepsilon)_r\times\bd E(K)$, with $r$ the inward coordinate and
$\bd E(K)=\{r=0\}$, take
\begin{equation}\label{eq:collar}
   dr^2+\varphi(r)^2\,g_{\bd E(K)},\qquad \varphi(0)=1,\quad \varphi'(0)<0,
\end{equation}
where we fix $\varphi(r)=1-r$ for sufficiently small $r$, and extend the metric
smoothly over the compact exterior.  For links make the same choice on each
boundary torus.  Two consequences of \eqref{eq:collar} are
used repeatedly and are recorded now.  Along the collar
$\nabla_X\bd_r=(\varphi'/\varphi)X$ for $X$ tangent to the level sets, so
\begin{equation}\label{eq:hessr}
   \Hess r=\varphi\varphi'\,g_{\bd E(K)}\;<\;0
   \qquad\text{on the level sets,}
\end{equation}
and in particular $r$ is strictly concave along sufficiently short geodesics
with initial velocity tangent to a level set.  A sufficiently short geodesic
joining two nearby points of $\bd E(K)$ therefore has $r>0$ at its interior
points.  Dually, the second fundamental
form of $\{r=0\}$ with respect to the \emph{inward} normal $\bd_r$ is
$-\varphi\varphi'g_{\bd E(K)}=g_{\bd E(K)}>0$, so $\bd E(K)$ has mean curvature
$2>0$ with respect to the inward normal, which is the sign convention of
\cite[p.~110]{HassScott}.  We also write $\widetilde{E}$ for the manifold
obtained by extending the collar to $r\in(-\varepsilon,\varepsilon)$ with the
same formula \eqref{eq:collar}; thus $E(K)\subset\widetilde{E}$ isometrically
and $\bd E(K)$ is a two-sided smooth surface in $\Int\widetilde{E}$.

Choose the flat metric so that the preferred longitudes form a linear,
\emph{real-analytic} foliation $J$ of the boundary torus; smoothness of $J$
is all the citations below require
\cite[Definition~3, pp.~884--885]{Schultens}, and real-analyticity, free here,
is used in \S\ref{sec:exchange-lemma}.  A properly embedded $F\subset E(K)$
with $\bd F\subset J$ is \emph{relatively least-area} if it is smooth and neat
and $\Area(F)$ is least among smooth neat representatives in its smooth
proper isotopy class subject to the boundary lying in $J$: the boundary
slides from leaf to leaf and is not pinned pointwise.

\begin{dfn}[Piecewise smooth surfaces]\label{conv:ps}
A compact surface $F$ properly embedded in a Riemannian $3$-manifold $X$ is
\emph{piecewise smooth} if there is a finite graph $\Gamma_F\subset F$,
embedded as a finite union of smooth arcs and circles, \emph{properly} --- so
that $\Gamma_F\cap\bd F$ is the finite set of endpoints of its arcs, and no arc
runs along $\bd F$ --- such that
$F\setminus\Gamma_F$ is a smooth submanifold of $X$ and the closure of each of
its components --- its \emph{pieces} --- is a compact surface with corners,
smooth in its interior and embedded of class $C^1$ up to its own boundary, the
two edges at each corner having distinct tangent rays.  A
smooth surface is piecewise smooth, with
$\Gamma_F=\emptyset$.  Such an $F$ admits an abstract triangulation compatible
with $\Gamma_F$; this does not assert ambient $C^\infty$ parametrisations
on the closed triangles.  \emph{Isotopy} of piecewise smooth surfaces is understood
in this tame category: ambient isotopies are locally flat and may pass through
surfaces with corners.  This does not assert that such an isotopy is a smooth
family of diffeomorphisms.  Lemma~\ref{lem:category-bridge} supplies the
comparison on smooth neat endpoints that is needed for the vertex classes
and for the geometric citations.  In the tame category the classical theorems of
Alexander --- that a locally flat $2$-sphere in the interior of a $3$-ball
bounds a ball there, and that a homeomorphism of a ball fixing the boundary pointwise is
isotopic to the identity rel the boundary --- are available.

The $C^1$ requirement is sufficient for the auxiliary regularity implication
in Proposition~\ref{prop:R1fixed}.  The proper singular graph and distinct
corner rays are part of its hypotheses.  This definition describes
intermediate frontiers and that auxiliary assertion; it does not define the
area competing class of the main theorems.

\end{dfn}

\begin{rem}[The competing class and the fixed-boundary citation]
\label{rem:ps-reading}
Hass and Scott use a class of piecewise smooth surfaces isotopic to $F$ rel
$\bd F$ \cite[Theorem~6.12, pp.~112--113]{HassScott}, without defining that
regularity term.  We do not infer membership in Definition~\ref{conv:ps}
from those words.  Instead Theorem~\ref{thm:fixed-smooth} proves attainment
of the infimum $I_{\mathrm{sm}}$ over smooth neat representatives in the
relative smooth ambient isotopy class.  If $I_{\mathrm{ps}}$ is the infimum
over Definition~\ref{conv:ps} in the corresponding tame class, then
$I_{\mathrm{ps}}\leq I_{\mathrm{sm}}$; the reverse inequality is not asserted.

The boundary of a disc cut from a sharp half-ball can have two corners.
The proof in \S\ref{ssec:fixed-smooth} first smooths the convex replacement
domain while retaining a central patch of the ambient boundary.  Its cut
curves are smooth.  The exterior splice can still have a cuspidal strip,
but the complete joined sheet is a Lipschitz graph, and its relative,
area-controlled smoothing is proved there.  The disc-convergence lemmas
are applied to exact Plateau discs before this smoothing, not to arbitrary
approximate minimisers.

Proposition~\ref{prop:R1fixed} remains a separate implication: if a minimum
in Definition~\ref{conv:ps} is attained, it is smooth and neat.  Neither
that implication nor exclusion of a one-sided double-cover limit alone
establishes existence.  The main argument uses the smooth attainment
theorem and does not require attainment in the larger competing class.
\end{rem}

\begin{thm}[\textup{(U)}: universal disjointness of relative area minimisers]
\label{thm:U}
Let $E$ be compact and irreducible with the data above.
\begin{enumerate}
\item[\textup{(i)}] Relative least-area incompressible surfaces
$F_1,F_2\subset E$ whose proper isotopy classes contain disjoint
representatives are themselves disjoint, or coincide.
\item[\textup{(ii)}] If $f_i\colon S_i\to E$, $i=1,\dots,n$, are incompressible,
pairwise non-isotopic and pairwise disjoint, then relative area minimisers
$g_i$ in the classes of the $f_i$ have pairwise disjoint images.
\end{enumerate}
\end{thm}

Clause (i) is \cite[Theorem~4, p.~885]{Schultens} and clause (ii) is
\cite[Theorem~A.5, p.~899]{Schultens}; the latter carries the number $12$ in
the author's preprint, and we cite the journal version throughout.  The force
of (U) is the universal quantifier, which lets minimisers be chosen one class
at a time, before any pair is examined.

\begin{thm}[\textup{(E)}: attainment]\label{thm:E}
Use the metric and foliation of \S\ref{ssec:data}.  Let $F\subset E(K)$
be a connected orientable incompressible smooth neat surface, with
$\bd F$ a leaf of $J$.
\begin{enumerate}
\item[\textup{(i)}] \emph{(Fixed boundary.)}  Let $\cF_{\mathrm{sm}}(F)$
be the smooth neat surfaces smoothly ambiently isotopic to $F$ rel
$\bd F$, and put
$I_{\mathrm{sm}}(F)=\inf_{G\in\cF_{\mathrm{sm}}(F)}\Area(G)$.
There is $F'\in\cF_{\mathrm{sm}}(F)$ with area $I_{\mathrm{sm}}(F)$;
it is a stable minimal surface.
\item[\textup{(ii)}] \emph{(Boundary sliding in $J$.)}  The infimum over
smooth neat surfaces in the smooth proper isotopy class of $F$, with
boundary in $J$, is attained.  Its minimiser belongs to $\cM([f])$ in the
sense of \cite[p.~897]{Schultens}: it is a stable minimal surface in the
proper isotopy class of the smooth embedding
$f\colon(S,\bd S)\to(E(K),\bd E(K))$, with
$f^{-1}(\bd E(K))=\bd S$ and boundary a parametrised curve of $J$.
\end{enumerate}
\end{thm}

\begin{proof}
Clause (i) is Theorem~\ref{thm:fixed-smooth}.  The exterior is compact,
orientable and irreducible, and the metric has the exact collar required
there.  Clause (ii) follows from Theorem~\ref{thm:R1} and
Lemma~\ref{lem:R2}, using clause (i) and
\cite[Proposition~A.1 and Corollary~A.3]{Schultens}.  No sliding-boundary
attainment is used in the fixed-boundary construction.
\end{proof}

Here \textup{(E)} without qualification means clause \textup{(ii)}, because a
vertex of $\IS(K)$ is a class in which the boundary longitude may move
(\S\ref{sec:kakimizu-complex}), so the competing class is the sliding one and
the complexity's area coordinate $A(v)$ is its infimum.

\begin{rem}[The compactness and disjointness interfaces]
\label{rem:E2-gives-R1}
Corollary~A.4 of \cite{Schultens} states sliding-boundary attainment, but
its non-emptiness step invokes \cite[Theorem~6.12]{HassScott} on p.~897.
Here non-emptiness is supplied by Theorem~\ref{thm:fixed-smooth}; only the
separate compactness and isotopy-component assertions are used to pass to
the sliding class.

On p.~899, immediately before Theorem~A.5, Kapovich calls area minimisers
in $\cM([f])$ relatively least-area.  Every member of that space is a
smooth proper stable minimal embedding and is neat by the boundary
argument of Lemma~\ref{lem:R2}.  A minimum over all smooth neat competitors
therefore minimises area on $\cM([f])$ as well.  This verifies the input
to \textup{(U)} without an area-density assertion for
Definition~\ref{conv:ps}.
\end{rem}

\subsection{Smooth fixed-boundary least-area representatives}
\label{ssec:fixed-smooth}

Let \(M\) be the compact irreducible orientable knot or allowed link exterior, with the
metric of \eqref{eq:collar}, whose boundary collar is

\[
dr^2+(1-r)^2g_{\partial M},\qquad r\geq0,
\]

where \(r\) points inward and the boundary tori are flat. Let \(F\subset M\) be a
compact connected smooth neat two-sided incompressible surface properly embedded in
\(M\), with nonempty prescribed smooth boundary \(\lambda\). Define

\[
\begin{split}
I_{\rm sm}=\inf\bigl\{\operatorname{Area}(G):{}&
G\text{ is smooth and neat, and is smoothly}\\
&\text{ambiently isotopic to }F\text{ rel }\lambda\bigr\}.
\end{split}
\]

\begin{thm}[Smooth fixed-boundary attainment]\label{thm:fixed-smooth}
The infimum \(I_{\rm sm}\) is attained by a smooth neat embedded minimal
surface \(T\) smoothly ambiently isotopic to \(F\) rel \(\lambda\).
\end{thm}

The statement applies to every permitted fixed tuple of boundary leaves in
a knot or non-split link exterior.  On each boundary torus choose a flat
metric and use the exact warped collar above, extending the metric over
the compact exterior.  These exteriors are irreducible and the permitted
spanning surfaces are connected with nonempty boundary.

We use the disk results of \cite[Theorems~2.7, 2.9 and Lemmas~3.3, 3.5, 3.6,
6.9--6.11]{HassScott}, and the local bridge construction in the proof of
\cite[Theorem~4.1, pp.~103--104]{HassScott}.  The fixed-boundary surface
assertion of \cite[Theorem~6.12]{HassScott} is not an input.

\paragraph{The outward extension.}\label{lem:fb-extension}

Attach \((-\infty,0]\times\partial M\) to the collar, with the same metric
\(dr^2+(1-r)^2g_{\partial M}\), to form \(N\). The metric matches to all orders at
\(r=0\), is complete, and has bounded geometry: on the end the torus scale is at least
one, curvature is bounded, and the injectivity radius has a positive lower bound; the
remaining part is compact. Projection of the added end to \(r=0\), extended by the
identity on \(M\), is \(1\)-Lipschitz.

For a rank-two tangent plane in the added end this projection strictly decreases its
two-dimensional Jacobian. Indeed it kills the radial component and contracts each torus
component by \(1/(1-r)<1\). Thus an area-minimizing Plateau disk in \(N\), whose
boundary lies in \(M\), cannot have an open piece outside \(M\): composing with the
retraction would strictly lower area. Interior branch points, if represented
parametrically, are isolated and do not change this conclusion. This observation
concerns disk minimization among maps and does not assume surface-class attainment.

\paragraph{Smooth convex replacement domains.}\label{lem:fb-domains}

At \(p\in\partial M\), choose a smooth defining function \(u\) for \(M\), with
\(M=\{u\leq0\}\), whose Hessian is positive definite on a fixed coordinate neighborhood
in \(N\). Such a function exists here: begin with outward signed distance \(d=-r\). Its
tangential Hessian on \(\partial M\) is positive definite. Replacing \(d\) by \(d+C
d^2\), and then shrinking the neighborhood, makes the full Hessian positive definite as
well.

Let \(v_R(x)=\operatorname{dist}_N(p,x)^2-R^2\), for \(R\) smaller than the convexity
radius. Fix a smooth even convex function \(\theta_\epsilon\) with

\[
\theta_\epsilon(t)=|t|\quad(|t|\geq\epsilon),\qquad
\theta_\epsilon\geq|t|,\qquad |\theta'_\epsilon|\leq1.
\]

Set

\[
w_{R,\epsilon}=\frac{u+v_R+\theta_\epsilon(u-v_R)}2,
\qquad U_{R,\epsilon}=\{w_{R,\epsilon}\leq0\}.
\]

The Hessian is

\[
\nabla^2w=\frac{1+\theta'}2\nabla^2u+\frac{1-\theta'}2\nabla^2v_R
+\frac{\theta''}2\,d(u-v_R)\otimes d(u-v_R),
\]

so it is positive definite. For sufficiently small \(\epsilon\ll R^2\), the zero level
is regular: near the old corner the two outward normals have a fixed angle smaller than
\(\pi\), and elsewhere one of the two defining functions is unchanged. The sublevel is a
smooth strictly convex ball, lies in the original half-ball, and agrees exactly with
\(\partial M\) on the central patch where \(v_R< -\epsilon\). It contains a prescribed
smaller relative half-ball after \(\epsilon\) is chosen small. The area bound
\(\operatorname{Area}(\partial U_{R,\epsilon})\leq C R^2\) is uniform in the rounding
parameter. Indeed the first two Hessian terms give a positive lower bound on the fixed
coordinate neighborhood, and the last is positive semidefinite. The gradient is
uniformly bounded because the first derivatives are convex weights. In normal
coordinates at \(p\), the Christoffel symbols are \(O(R)\). After decreasing \(R\), the
Euclidean Hessian is positive definite throughout the containing radius-\(R\) ball. The
domain is Euclidean convex there. Divide its outward normals among the six signed
coordinate directions so that the chosen component has absolute value at least
\(1/\sqrt3\). On each part the corresponding projection is injective by convexity, and
its graph area is at most \(\sqrt3\) times the projected area. Each projection has area
at most \(\pi R^2\); summing and comparing the smooth metric with the Euclidean metric
proves the bound. No uniform curvature bound on the rounding strip is needed. Interior
replacement balls use the ordinary squared-distance construction.

For a common generic radius, fix \(\epsilon\) throughout a small radius interval
retaining the protected core. Where \(w_{R,\epsilon}\) depends on \(v_R\), its
\(R\)-derivative is \(-2R\) times the positive weight of \(v_R\), hence is nonzero.
Where that weight vanishes, \(w_{R,\epsilon}=u\), and the extended neat input is already
transverse to \(u=0\). Parametric transversality therefore provides a full-measure set
of radii for each input surface and the already constructed smooth limit. The countable
intersection supplies a common radius without losing the protected core.

\paragraph{The cut curve at a flat transition.}\label{lem:fb-cut-curves}

Extend a smooth neat input surface a short distance through \(\partial M\). At a point
of \(\lambda\), its tangent plane contains \(T\lambda\) and a vector with nonzero
\(dr\)-component. Hence it is transverse to \(\partial M\), and also to \(\partial U\)
wherever the latter agrees with, or is tangent to, \(\partial M\). The ordinary implicit
function theorem on this \emph{extended} surface shows that its intersection with the
smooth \(\partial U\) is a smooth curve through the flat transition. One must not use a
stratified transversality assertion for the unextended manifold-with-boundary at this
point.

When \(U\) meets just one small arc of \(\lambda\), there is one cut loop containing
that arc and the other cut loops are interior. The cut loop containing \(\lambda\) is
smooth, although the piece of the old surface outside \(U\) has a cusp at the endpoint
of its shared boundary arc. Smoothness of the cut loop does \emph{not} make this cusped
exterior piece an admissible Definition~\ref{conv:ps} piece.

\paragraph{Global comparison disks.}\label{lem:fb-global-disks}

It is insufficient to minimize only among disks contained in \(U\): an outermost old
disk bounded by a cut loop need not be contained in \(U\). Use a disk which minimizes
area in \(N\). By Morrey's theorem in the homogeneously regular form recorded in
\cite[Theorem~2.7 and its following remark, p.~94]{HassScott}, the smooth cut loop,
which spans its cap disk, bounds a global least-area parametrized disk in \(N\).
This existence precedes the following confinement argument.

A smooth cut loop on the sphere \(\partial U\) bounds one of its two cap disks, of area
at most \(\operatorname{Area}(\partial U)=O(R^2)\). The global minimizing disk therefore
has this area bound. Interior monotonicity in the complete extension \(N\) implies that,
for sufficiently small \(R\), it cannot exit a fixed larger coordinate ball: a point at
a fixed positive distance from the cut loop would contribute a fixed positive amount of
area. Its image is therefore in the neighborhood on which \(w_{R,\epsilon}\) is strictly
convex. The maximum principle for \(w_{R,\epsilon}\) composed with its conformal minimal
parametrization then confines it to \(U\).

The disk is consequently also a least-area disk in \(U\), and
\cite[Theorem~2.9]{HassScott} gives an embedded disk; disks for disjoint cut loops are
disjoint. Smooth-boundary Plateau regularity and strict convexity give smoothness and
neatness up to the \emph{whole smooth cut loop}. No corner-boundary regularity is being
invoked. Because it was minimized in \(N\), its area is at most that of every old disk
in \(M\subset N\) having the same boundary.

\begin{lem}[The graph at a flat transition]\label{lem:fb-splice-graph}
Suppose an old smooth neat sheet and a new smooth disk have the same smooth cut curve,
the new disk lies inside \(U\), and the retained old sheet lies outside \(U\). At an
endpoint \(q\in\lambda\) of the actual artificial seam, their union is locally a
Lipschitz graph over a half-plane, with zero boundary trace after straightening
\(\lambda\).
\end{lem}

\begin{proof}
Choose coordinates \(x,y,r\), with \(r\geq0\) in \(M\), \(\lambda=\{y=r=0\}\), and \(x\)
along \(\lambda\). Neatness of either sheet implies that projection onto \((x,r)\) is
locally invertible. This assertion does not require the two tangent planes to be close.
Write the two graph functions as \(g_o,g_n\). The cut curve projects to \(r=h(x)\),
where \(h\geq0\), and \(h=0\) on the interval where the curve lies in \(\lambda\). Both
graph functions agree on \(r=h(x)\). After extending the smooth functions across that
curve, the union has the form

\[
g(x,r)=g_o(x,r)+\max\{r-h(x),0\}\,b(x,r)
\]

for a smooth local function \(b\); division by \(r-h(x)\) is valid since the difference
vanishes on the cut curve. This formula, or its version with the roles interchanged,
proves the Lipschitz bound even where the exterior strip is cuspidal. The trace is
\(g(x,0)=0\).
\end{proof}

Away from \(\partial M\), the analogous common graph is obtained by projecting onto the
cut-curve coordinate and the signed normal coordinate of \(\partial U\). Transversality
of the old sheet and neatness of the new disk give the two inverse projections.

\begin{lem}[Smoothing a splice]\label{lem:fb-splice-smoothing}
For a fixed embedded splice as above, there are smooth neat surfaces obtained by tame
ambient isotopies fixing \(\partial M\), supported in any sufficiently small prescribed
neighborhood of its entire seam, which have area tending to the area of the splice. A
compact protected subset disjoint from the seam is unchanged.
\end{lem}

\begin{proof}
Use a finite collection of enlarged graph charts and smaller closed cores whose
interiors cover the seam. At the ambient boundary, extend \(g\) oddly in \(r\), convolve
with an even smooth kernel, and restrict to \(r\geq0\). The resulting graph has zero
boundary trace. In an interior chart use ordinary convolution. At each step choose a
smooth cutoff equal to one on the corresponding core and zero near the boundary of its
enlarged chart, and blend the current graph with its smoothed graph. The cutoff
transition may meet the seam; it need not already be smooth. After this step the surface
is smooth on the new core. At a point already made smooth in an earlier step, both the
current graph and its convolution are smooth, so every later blend keeps it smooth.
Finitely many steps therefore smooth the whole seam, including its endpoints, without
leaving a terminating crease inside a single component. The displacements are chosen
sufficiently small that the original smaller-core cover remains a cover of the displaced
seam.

Convolution converges uniformly and in \(W^{1,1}\), with uniformly bounded first
derivatives. In a smooth ambient chart the graph area density \(J(x,g,Dg)\) is uniformly
Lipschitz in \((g,Dg)\) on this bounded set. The area difference therefore tends to
zero.

The linear graph path \((1-t)g+t g_\delta\) stays embedded. A sufficiently small
vertical displacement extends it to a fiber-preserving ambient homeomorphism, equal to
the identity near the chart boundary and on \(r=0\). This is a tame ambient isotopy. For
a fixed configuration all unselected sheets are compact and disjoint from the selected
local graph, so their positive local clearance permits a sufficiently small uniform
displacement. The finite iteration is an isotopy of the whole surface. Its support
misses the protected compact set by the initial chart choice.
\end{proof}

This lemma applies to the splices constructed above. It is \emph{not} asserted for an
arbitrary Definition~\ref{conv:ps} surface: a general cusped fold along an edge need not
be a graph.

\paragraph{Relative disk isotopy.}\label{lem:fb-relative-disks}
For the qualitative topological disk replacement, first align the two neat boundary
collars by their common \((x,r)\)-graphs, fixing \(\lambda\). Cut along the retained
connected component \(C\), which reaches the ambient boundary. The resulting manifold is
irreducible: an interior sphere bounds a ball in \(M\); since \(C\) misses the sphere
and reaches \(\partial M\), it lies outside that ball, so the ball remains in the cut
manifold. The old and new disks are proper tame disks with the same boundary in this
irreducible manifold. The usual relative disk-isotopy theorem identifies them. Undo the
collar alignment. No area estimate is required during this qualitative isotopy. After
Lemma~\ref{lem:fb-splice-smoothing} gives smooth neat endpoints,
Lemma~\ref{lem:category-bridge} supplies the smooth relative ambient isotopy required in
\(I_{\rm sm}\). That bridge is purely topological and does not use minimal-surface
existence, area complexity, or universal disjointness.

\paragraph{Smoothing a convergent collection of sheets.}

Naive pointwise smoothing is inadequate for an induction: if the tangent jump is
\(a_i\to0\), a smoothing width \(\delta_i\) may create second derivatives of size
\(a_i/\delta_i\). Nor does direct odd reflection approximate the full boundary jet of a
general smooth graph. The following version avoids both problems.

\begin{lem}[Simultaneous smoothing]\label{lem:fb-ordered-smoothing}
In a fixed enlarged graph chart let finitely many disjoint, ordered, continuous
piecewise-smooth graphs \(g_{i,1}<\cdots<g_{i,m}\) converge piecewise in \(C^1\) to the
same smooth graph \(g\). Assume they converge smoothly away from the splice seam. Then
they can be smoothed simultaneously, keeping their order and remaining unchanged off a
prescribed seam neighborhood, so that the smoothed graphs converge to \(g\) in every
\(C^k\) on the smaller chart. Their total area changes by \(o(1)\). At a boundary chart
the same conclusion holds for a single graph with trace equal to the fixed smooth trace
of \(g\).
\end{lem}

\begin{proof}
Put \(f_{i,a}=g_{i,a}-g\) and \(e_i=\max_a\|f_{i,a}\|_{W^{1,\infty}}\to0\). Choose one
nonnegative convolution kernel and one cutoff for every sheet. On the central seam
neighborhood replace the graph by

\[
\widehat g_{i,a}=g+K_{\delta_i}*f_{i,a}.
\]

For each fixed \(k\geq1\),

\[
\|D^k(K_{\delta_i}*f_{i,a})\|_\infty
\leq C_k\,e_i\,\delta_i^{1-k}.
\]

Choose \(\delta_i\downarrow0\) sufficiently slowly, by diagonal selection, that these
bounds tend to zero for every fixed \(k\). On a smaller core where the cutoff is one
this gives smooth convergence immediately. On the cutoff transition there may still be
an unsmoothed seam, so no smooth-convergence assertion is made there at this step.
Positivity of the convolution and the common cutoff preserve strict sheet order; no
lower bound on the gap between sheets is required. The graph homotopy is also ordered.
Area convergence follows from \(C^1\) convergence.

At \(r=0\), apply odd reflection to the residual \(f_i\), not to \(g_i\) itself, and add
\(g\) back afterward. The residual has zero trace. Its odd extension has
\(W^{1,\infty}\)-norm \(O(e_i)\), so the same derivative estimates apply after
convolution. This keeps the boundary fixed and recovers every derivative of the actual
limit \(g\), including its nonzero even normal derivatives.

For a finite chart cover, proceed by the sequential core procedure of
Lemma~\ref{lem:fb-splice-smoothing}. At every stage all graph residuals are still
\(o(1)\) in \(W^{1,\infty}\), also after passing through another fixed smooth chart.
Choose the new common scale slowly enough for the derivative estimates on that chart. On
a previous core the current graph already converges smoothly to the local expression of
the same fixed limit; the new convolution does so as well. Their cutoff blend therefore
preserves smooth convergence on that core, even if this core meets the new transition.
On the new core the cutoff is one. The finite iteration establishes smooth convergence
throughout the smaller-core cover. Use one cutoff and one scale for all locally ordered
sheets at each step.
\end{proof}

If several ordered boundary graphs with the same trace ever have to be treated, a
positivity-preserving half-space kernel is available: for a radially decreasing even
kernel,

\[
K_\delta(x-x',r-s)-K_\delta(x-x',r+s)\geq0\quad(r,s\geq0).
\]

This is the odd-reflection convolution kernel on the half-space and preserves
nonnegative differences. The application below only needs the single actual boundary
sheet.

An important support qualification: in regions where the graphs already converge
smoothly, there is no need to impose a prescribed numerical error such as \(2^{-i}\).
The area alteration being \(o(1)\) suffices. Imposing both an arbitrarily fast error
schedule and high-derivative convergence would unnecessarily reintroduce the scale
conflict. At the remaining seam, outside the already convergent region and away from the
new protected core, apply Lemma~\ref{lem:fb-splice-smoothing} with an arbitrary error
tending to zero.

\paragraph{Uniqueness of the boundary sheet.}\label{lem:fb-boundary-sheets}

At the first replacement near the ambient boundary, the central patch has the
fixed smooth boundary arc \(\lambda\). Apply \cite[Lemmas~6.9 and~6.11]{HassScott}
there; no convergence of the old surfaces is assumed. On a smaller central patch
disjoint from \(\lambda\), use the same statements with fixed boundary
\(\Gamma=\varnothing\). The moving-arc lemma is used only at later interfaces
where the old surfaces already converge smoothly.

At a point \(q\in\lambda\) already lying in a previously convergent region, the old
sequence has one boundary sheet in a sufficiently small chart. Its intersection with the
smooth new frontier is therefore one smooth arc converging smoothly through the
flat-transition point. For the new disk containing this arc, the hypotheses of
\cite[Lemma~6.10]{HassScott} apply to an open patch of the \emph{smooth sphere}
\(\partial U\) containing \(q\). The fact that this patch is tangent to \(\partial M\)
creates no corner in the disk's boundary data.

One must also exclude extra local pieces of the \emph{same} original disk. Choose a
smaller rounded strictly convex half-ball \(W\subset U\) around \(q\), using \(\partial
U\), rather than \(\partial M\), as its central face. Its central face contains a fixed
small patch on which the original boundary curve is the single smoothly convergent arc
just described. Choose the radius transverse to the countable extended disks and small
enough that \(W\) does not contain an entire original boundary loop.

Every cut loop on \(\partial W\) bounds a subdisk of its original disk. This subdisk is
\emph{globally area-minimizing in \(N\)}: a smaller-area replacement would lower the
original disk's area. The small cap on \(\partial W\) bounds its area, and the
confinement argument of \hyperref[lem:fb-global-disks]{the global comparison-disk
argument} puts the entire subdisk in \(W\). Strict convexity keeps its interior off
\(\partial W\). Consequently the pieces cut out in \(W\) really are least-area disks,
not multiply connected planar pieces; nested additional cut loops are impossible.
Exactly one local disk contains the central boundary arc. All other local disks, whether
from the same original disk or a different one, have empty fixed boundary on the central
face.

Apply \cite[Lemmas~6.9--6.11]{HassScott} to the collection of these remaining local
disks, whose total area is still bounded by the original total area. None can accumulate
at \(q\): a limiting surface through \(q\) could neither be properly embedded with empty
boundary there nor be a minimal subsurface of the strictly convex patch. The single disk
containing the arc is the boundary component supplied by the boundary-convergence lemma.
Thus there is one actual local boundary sheet. This is established before claiming a
common ordered graph chart.

On the interior part of the seam, the two limiting pieces match smoothly after the usual
no-bend argument. Their smooth extensions to \(q\) consequently have identical jets at
\(q\), by continuity along the seam. This justifies the common smooth reference graph at
the flat-transition endpoint, rather than assuming it as a boundary regularity
statement.

The word 'convergent' here must mean the geometric smooth disk convergence of the cited
lemmas, including local finiteness; weak varifold convergence is not sufficient for the
exclusion of extra sheets or the following covering argument.

\paragraph{Area-controlled disk cleaning.}

\begin{lem}[Area-controlled disk replacement]\label{lem:fb-cleaning}
Let \(S\) be a compact connected smooth neat two-sided incompressible surface,
properly embedded with nonempty boundary in a compact irreducible orientable
\(3\)-manifold \(M\). Let \(D,E\subset\operatorname{Int}M\) be smooth embedded
closed disks, with \(D\subset\operatorname{Int}S\),
\(\partial E=\partial D\) and \(\operatorname{Int}E\cap D=\varnothing\). It may meet
\(S\setminus D\). Suppose \(\operatorname{Area}(E)<\operatorname{Area}(D)\).

For every sufficiently small \(\tau>0\), there is a smooth embedded surface \(S'\),
smoothly ambiently isotopic to \(S\) rel \(\partial S\), such that

\[
\operatorname{Area}(S')\leq\operatorname{Area}(S)-\operatorname{Area}(D)+\operatorname{Area}(E)+\tau.
\]

Applications to the raw splices below are made only after the specific preparation in
\hyperref[lem:fb-bridge-preparation]{the collapsing-bridge argument}. No density
assertion for arbitrary Definition~\ref{conv:ps} surfaces is needed.
\end{lem}

\begin{proof}
An arbitrarily small collar alteration of \(E\), fixing its boundary and keeping it off
\(\operatorname{Int}D\), makes its interior disjoint from \(S\) near \(\partial D\). In
the normal two-dimensional fibers along that circle, keep the incoming ray of \(D\)
fixed and put the ray of \(E\) on one side of the plane of \(S\); join the adjusted germ
to the original collar inside a shrinking tube. This adjustment avoids the fixed disk
ray, preserves the embedded disk, and has area cost tending to zero with the tube width.
General position away from this collar makes the remaining intersections finitely many
transverse circles. The shared \(\partial D\) is not counted among them.

Allocate all collar, push-off, and smoothing errors below a total \(\tau\), chosen
smaller than half the initial area difference. For an intersection circle \(\gamma\),
denote its subdisk in \(E\) by \(E_\gamma\). It bounds a disk \(D_\gamma\) in \(S\), by
incompressibility. This disk is unique because \(S\) is connected with nonempty boundary
and hence is not a sphere. Since \(\gamma\) misses \(D\), either \(D_\gamma\cap
D=\varnothing\) or \(D\subset D_\gamma\).

Consider all disks \(E_\gamma\), together with the disks \(D_\gamma\) disjoint from
\(D\). Choose one, \(Q\), of smallest area. Its interior contains no intersection
circle: an inner circle would bound a strictly smaller subdisk in \(E\) when \(Q\subset
E\); when \(Q\subset S\), its unique surface disk lies inside \(Q\), remains disjoint
from \(D\), and is eligible as well. This proves cleanliness of the chosen disk without
claiming that \(S\setminus D\) is a disk.

\emph{Case 1: \(Q=D_\gamma\subset S\), disjoint from \(D\).} Replace \(E_\gamma\) in the
auxiliary disk by \(Q\). The interior of \(Q\) is disjoint from all of \(E\), so this is
an embedded disk with the same boundary. It remains disjoint from \(D\). Its area does
not increase, since \(E_\gamma\) was eligible too. Push \(Q\) very slightly to the side
of \(S\) from which the retained collar of \(E\) approaches, and round the joining
circle. The clean interior and the transverse collar model ensure that this removes
\(\gamma\), introduces no new intersection circle, and preserves embeddedness. Other
circles in the discarded subdisk disappear. The cost can be made as small as assigned.
No isotopy of this auxiliary disk relative to \(D\) is asserted or needed. In
particular, a ball bounded by \(Q\cup E_\gamma\) is allowed to contain \(D\); the
surface \(S\) has not changed.

\emph{Case 2: \(Q=E_\gamma\) and \(D_\gamma\cap D=\varnothing\).} The sphere
\(E_\gamma\cup D_\gamma\) is embedded, since \(\operatorname{Int}E_\gamma\cap
S=\varnothing\). It bounds a ball by irreducibility. The retained surface
\(S\setminus\operatorname{Int}D_\gamma\) is connected, is disjoint from this sphere
except for the joining circle, and reaches \(\partial M\). It is therefore outside the
interior ball. In particular \(D\) is outside the ball. A ball slide replaces
\(D_\gamma\) by a small parallel copy of \(E_\gamma\), followed by rounding. This
preserves the relative ambient isotopy class of \(S\), also preserving \(D\). It does
not increase area apart from the assigned error, since \(D_\gamma\) was eligible. The
transverse collar push-off removes \(\gamma\) and all circles on the discarded surface
disk without introducing new circles. The auxiliary disk remains unchanged.

\emph{Case 3: \(Q=E_\gamma\) and \(D\subset D_\gamma\).} Again
\(\operatorname{Int}E_\gamma\cap S=\varnothing\). Replace the whole \(D_\gamma\) by
\(E_\gamma\), obtaining an embedded surface. The same ball argument supplies the
relative ambient isotopy. Stop. Apart from the assigned rounding cost its area is at
most

\[
\operatorname{Area}(S)-\operatorname{Area}(D_\gamma)+\operatorname{Area}(E_\gamma)
\leq\operatorname{Area}(S)-\operatorname{Area}(D)+\operatorname{Area}(E).
\]

This is the case unavailable in the punctured-complement argument.

At each nonterminal step the intersection count strictly decreases. The disk \(D\) stays
fixed inside the current surface. The auxiliary disk stays embedded, has boundary
\(\partial D\), and misses \(\operatorname{Int}D\). The quantity

\[
\mathcal B=\operatorname{Area}(S)-\operatorname{Area}(D)+\operatorname{Area}(E)
\]

does not increase, apart from the allocated error. The current surface remains in the
initial relative class and is still incompressible. There are at most as many steps as
the initial number of circles, so a finite allocation below \(\tau\) is possible.

If Case 3 never occurs, the final auxiliary disk is disjoint from the retained
\(S\setminus D\). It and \(D\) form an embedded sphere. The connected retained surface
reaches \(\partial M\), so it is outside the ball bounded by that sphere. A final ball
slide replaces \(D\) by \(E\), and a small rounding gives \(S'\). The budget \(\mathcal
B\) proves the claimed bound. Tame ball slides with smooth endpoints give smooth
relative ambient isotopies by Lemma~\ref{lem:category-bridge}.
\end{proof}

\paragraph{The collapsing-bridge alternative.}\label{lem:fb-bridge-preparation} The
local geometric construction in \cite[proof of Theorem~4.1, pp.~103--104]{HassScott},
gives disks \(G_i,G'_i\) with common boundary, disjoint interiors, and a fixed positive
area difference. This geometric construction does not use the topology of the complement
of \(G_i\). Lemma~\ref{lem:fb-cleaning}, with an error smaller than half that
difference, supplies an embedded relative-isotopy competitor with a fixed area saving.
An intersection circle enclosing \(G_i\) is handled by Case 3 rather than being mistaken
for a disk in its complement.

Here is the preparatory step when the current surface is a raw splice. It concerns only
a \emph{fixed finite marked configuration for each fixed index} and does not use a
density theorem for arbitrary piecewise-smooth surfaces.

First choose all the local data in the bridge construction: the small auxiliary
half-ball on the old side, its two disk pieces, the small regular neighborhood of their
smooth limiting disk, the connector in the intersection of the new disk with the
relevant regular-neighborhood face, and a thin strip about that connector. The
regular-neighborhood face cuts off the two small cap disks used by the construction.
Choose this finite collection in general position with the raw seam. Marked curves
either cross the seam in finitely many transverse points or follow a smooth subarc of
it; tangential accidental contacts can be removed before choosing the smoothing width.
Choose the neighborhoods and strip thin enough that the desired area inequality has a
fixed positive margin. The cap-area estimate uses the whole small regular-neighborhood
face, not merely the connector length.

For this fixed index, the graph smoothing in Lemma~\ref{lem:fb-splice-smoothing} can be
taken in a tube of width \(h\) about the seam, with ambient fiber homeomorphisms \(H_h\)
and inverses having a Lipschitz constant bounded independently of \(h\). To see this in
one chart, the graph displacement is \(O(h)\), its first derivatives are bounded by the
fixed graph Lipschitz bound, and both the base cutoff and the fiber cutoff may have
scale \(h\). Choose the fiber cutoff width to be a sufficiently large fixed multiple of
the displacement bound; then its fiber derivative is bounded away from zero. The base
derivatives of the homeomorphism and its inverse remain bounded. Finite composition
preserves these bounds.

Off the seam these maps are eventually the identity, or converge in \(C^1\) on compact
sets if a fixed cutoff is used. Along the seam their tangential derivatives tend to the
identity: the two graph pieces have the same boundary trace and hence the same
tangential derivative there. Thus the area of any of the finitely many marked smooth
pieces changes by a quantity tending to zero; inside the shrinking tube use the uniform
Lipschitz bound and the fact that its intersection with a fixed two-dimensional piece
has area tending to zero. The lengths of marked curves also converge. For transverse
crossings the affected parameter intervals shrink to finitely many points and have
uniformly bounded speed. For a segment lying on the seam, use the tangential derivative
convergence just noted. This is stronger than \(C^0\) closeness and controls the
particular curves actually used.

Transport the marked curves and connector to the smoothed surface, making any finitely
many endpoint roundings inside still smaller marked neighborhoods. Then
\emph{reconstruct} its two-disk-plus-strip subdisk and the comparison annulus with its
two cap disks on this smooth surface, using the same small auxiliary faces and collars,
or smooth approximations of their transported marked faces. These are finite collar,
strip, and cap constructions; the preceding area and length estimates let their total
error be less than any chosen fraction of the fixed gain. The caps remain within the
chosen small-area regular-neighborhood face. Take the complementary frontier in this
reconstruction, so the two resulting disks have exactly the same boundary and disjoint
interiors by construction. An old auxiliary disk is not being inserted unchanged into a
newly smoothed surface, and a merely Lipschitz image of that old disk is not being
called smooth.

To make the approximation of transported faces precise, choose the auxiliary faces
generically so that at their finitely many transverse meetings with the seam they are
transverse to the fiber direction used for smoothing. Choose the collars along a shared
boundary arc with the same property. Each affected face is a graph over the unchanged
base, so its transported image is a Lipschitz graph.
Lemma~\ref{lem:fb-splice-smoothing}'s graph approximation applies to these graphs; no
density theorem for arbitrary locally flat surfaces is used. Build the comparison disk's
collar from the actual smooth boundary of the reconstructed disk, on the prescribed
strict side. Outside an arbitrarily thin common collar the two compact disjoint disk
pieces have positive separation, preserved by sufficiently small uniform graph
approximations. The collar area tends to zero with its width. Thus this reconstruction
retains exactly the common boundary and disjoint interiors.

Choose \(h\) and the finitely many final collar widths small enough that the loss is
below one eighth of the fixed area margin, and apply Lemma~\ref{lem:fb-cleaning} to this
smooth pair. All choices may depend on the index: only the positive margin needs to be
uniform. The local bridge operation stays away from \(\lambda\); its endpoint on
\(\lambda\) is handled instead by \hyperref[lem:fb-boundary-sheets]{the boundary-sheet
argument}.

\paragraph{Excluding bends.}\label{lem:fb-no-bend}

After the collapsing-bridge alternative is excluded, each retained outer boundary arc is
matched to one inner limiting disk. In a small interior seam chart there are \(m\) outer
copies of one smooth limiting sheet. Let \(\eta\) be their common outward unit conormal
along the seam and let \(\nu_1,\ldots,\nu_m\) be the outward unit conormals of the
matched inner pieces, counted with multiplicity. There are exactly \(m\) inner
incidences: at every finite stage each retained arc has one disk attached; separated-arc
convergence and the no-bridge conclusion preserve those incidences. Extra disks without
boundary on this local patch cannot reach the strictly convex frontier in the limit, by
the empty-boundary argument in \hyperref[lem:fb-boundary-sheets]{the boundary-sheet
argument}.

Choose a nonnegative cutoff \(\chi\) supported in a subarc and a smooth ambient vector
field \(X\), supported in the interior seam chart, with \(X=\chi\eta\) on the seam. Each
limiting piece is minimal away from the seam. Since the vector field vanishes near the
other chart edges, integration of its first variation gives

\[
\delta A(X)=\int_{\rm seam}\chi\,\eta\cdot\left(m\eta+\sum_{a=1}^m\nu_a\right)
=\int_{\rm seam}\chi\sum_{a=1}^m(1+\eta\cdot\nu_a)\geq0.
\]

If a true bend occurs, choose \(\chi\) so this integral is positive. The flow of \(-X\)
decreases the limiting weighted area by a fixed positive amount at a fixed sufficiently
small time. Piecewise \(C^1\) convergence gives convergence of the areas both before and
after that fixed flow. Thus it decreases the raw approximating areas by a fixed amount
for all sufficiently large indices.

The flow acts on the whole embedded surface and preserves its relative isotopy class,
with support disjoint from \(\lambda\). It also transports its local Lipschitz splice
charts. Lemma~\ref{lem:fb-splice-smoothing} therefore gives smooth competitors whose
area errors tend to zero, contradicting convergence to \(I_{\rm sm}\).

Consequently every \(\nu_a=-\eta\). Each matched pair joins \(C^1\). In a common graph
chart the minimal graph equations join weakly without an interface measure; uniform
ellipticity and interior regularity give a smooth minimal graph. If several inner
branches are matched to copies of the same old limiting graph, these full joined
minimal graphs are locally ordered, as limits of the ordered embedded sheets.
They coincide on the old side. Choosing a contact point in the interior of that
side, the strong maximum principle makes them coincide on the whole connected
smaller graph chart. Thus there is one smooth limiting graph there, with its
possible multiplicity, rather than merely agreement of tangent planes.
At the ambient-boundary
endpoint both sides have smooth extensions by \hyperref[lem:fb-boundary-sheets]{the
boundary-sheet argument}, and equality of their interior jets extends continuously to
that endpoint. The common reference graph in Lemma~\ref{lem:fb-ordered-smoothing} is
therefore smooth there as well.

\begin{proof}[Proof of Theorem~\ref{thm:fixed-smooth}]

Choose finitely many protected relative open cores \(V_1,\ldots,V_n\) covering \(M\),
each with closure in one of the small replacement neighborhoods of
\hyperref[lem:fb-domains]{the replacement-domain construction}. The replacement domain
\(U_j\supset\overline V_j\) may be selected from a fixed radius interval when stage
\(j\) is reached. The interval preserves the protected core.

Start with a smooth minimizing sequence \(F_i^0\) for \(I_{\rm sm}\). Inductively require:

\begin{enumerate}
\item \(F_i^{j-1}\) is smoothly and neatly embedded, is in the prescribed smooth
relative class, and its areas tend to \(I_{\rm sm}\).
\item On a relative neighborhood of the already protected compact cores, it converges
smoothly, possibly with interior multiplicity, to a smooth embedded minimal surface; at
its portion of \(\lambda\) there is the unique boundary sheet described in
\hyperref[lem:fb-boundary-sheets]{the boundary-sheet argument}.
\end{enumerate}

Choose \(U_j\) generic for the countable sequence and the current limit, using the
extended surfaces at \(\partial M\). Every cut loop is null-homotopic in \(M\) and hence
in the input surface. For the loop sharing an interval with \(\partial F_i^{j-1}\), push
that interval slightly inward in the surface to express the same assertion as the usual
simple-loop statement; its disk is the boundary-parallel subdisk bounded by the original
cut loop. Select the outermost disjoint disks. Their complement is connected and
contains the portions of \(\lambda\) outside \(U_j\), and consequently lies outside
\(U_j\). The exact replacement disks from \hyperref[lem:fb-global-disks]{the global
comparison-disk argument} are pairwise disjoint and miss this retained complement. They
give an embedded raw splice \(P_i^j\), topologically in the prescribed relative class,
with

\[
\operatorname{Area}(P_i^j)\leq\operatorname{Area}(F_i^{j-1}).
\]

Every raw splice admits the area approximation in Lemma~\ref{lem:fb-splice-smoothing}.
Therefore any fixed positive area saving obtained by a further admissible local
cut-and-paste operation contradicts the definition of \(I_{\rm sm}\), after smoothing
with error smaller than that saving. Exact monotonicity within the smooth class is not
required.

Take subsequences of the \emph{exact replacement disks} before smoothing them. Disk
convergence gives a smooth minimal limit on \(V_j\), including its fixed-boundary arc.
At the interface with the already convergent region, separated boundary arcs give
boundary convergence. Exclude the local collapsing-bridge pair by
Lemma~\ref{lem:fb-cleaning}, which supplies an embedded relative-isotopy competitor even
when an intersection circle encloses the original bridge disk. Exclude a bend by the
ambient flow in \hyperref[lem:fb-no-bend]{the no-bend argument}. Both fixed positive
savings survive the \(o(1)\) errors. At interface endpoints on \(\lambda\), use the
local subdisk extraction in \hyperref[lem:fb-boundary-sheets]{the boundary-sheet
argument} instead of a corner theorem.

This gives raw local sheets which converge piecewise smoothly to one smooth limit in the
enlarged protected region. The common graph representation required in
Lemma~\ref{lem:fb-ordered-smoothing} is obtained \emph{now}, from this piecewise \(C^1\)
convergence and local finiteness: the projection to the common tangent-plane chart is
invertible on each piece, their common boundary values glue the pieces, and the
no-thin-bridge conclusion prevents a folded connection. This step does not use the final
global covering conclusion. Finitely many local sheets can be ordered by their height
because the raw surface is embedded.

Apply Lemma~\ref{lem:fb-ordered-smoothing} simultaneously to the sheets on the already
convergent part of the new seam. On the rest of the seam use
Lemma~\ref{lem:fb-splice-smoothing}. The latter support is disjoint from all smaller
protected cores, and the former smoothing preserves smooth convergence there. Choosing a
finite subordinate collection of chart cutoffs gives a smooth neat \(F_i^j\) in the same
class with

\[
\operatorname{Area}(F_i^j)\leq\operatorname{Area}(F_i^{j-1})+o(1).
\]

Hence \(\operatorname{Area}(F_i^j)\to I_{\rm sm}\), and the induction hypotheses hold on
the union of the first \(j\) protected cores. The exact-disk lemmas have only been
applied to the exact disks; they have not been applied to the smoothed near-minimizers.

After the finite induction, the final smooth surfaces converge smoothly throughout \(M\)
to a compact embedded minimal surface \(T\), with its prescribed boundary. Smoothness
and the strict boundary barrier give neatness. Choose a \emph{boundary-adapted
transverse line bundle}, rather than metric-orthogonal normals: at \(\partial T\), take
the line in \(T\partial M\) transverse to \(T\partial T\), and extend it to a smooth
line field transverse to \(T\). Its tubular embedding can be chosen to take fibers over
\(\partial T\) into \(\partial M\), and to take all other sufficiently short fibers into
\(\operatorname{Int}M\). Geometric convergence places every final surface in this tube,
transverse to these fibers. The induced projection is a proper local diffeomorphism of
surfaces with boundary, hence a covering. Connectedness puts it over one component.
Every point of \(\lambda\) has exactly one preimage, since the surface has precisely
the prescribed boundary and the tubular projection fixes \(\lambda\) pointwise.
Thus the covering has degree one, also when \(\lambda\) has several components. It follows that
the surfaces are single graphs in this adapted line bundle. Their graph sections vanish
on \(\partial T\), so scaling those sections to zero is an isotopy fixing \(\lambda\);
boundary-adapted isotopy extension supplies the ambient isotopy. Finally smooth area
convergence gives

\[
\operatorname{Area}(T)=I_{\rm sm}.
\]

The degree-one argument at fixed boundary is more direct here than using the
primitive-slope obstruction to a one-sided double cover. No conclusion about a
minimizing sequence converging merely in the sense of varifolds would justify this step.
\end{proof}

\begin{rem}\label{lem:fb-scope}
Theorem~\ref{thm:fixed-smooth} concerns smooth neat competitors in a fixed
relative smooth ambient isotopy class.  It does not identify this infimum
with the infimum over all surfaces of Definition~\ref{conv:ps}.
Lemma~\ref{lem:fb-splice-smoothing} applies to the graph splices constructed
in the proof; it does not assert area density for arbitrary piecewise
\(C^1\) surfaces with cuspidal folds.
\end{rem}

\subsection{Auxiliary regularity and sliding-boundary attainment}\label{ssec:R1}

The local argument begins with an existing fixed-boundary minimiser in
Definition~\ref{conv:ps}; it does not assume sliding-boundary attainment.
Throughout the local argument, $M$ is a compact orientable $P^2$-irreducible
Riemannian $3$-manifold with the collar data \eqref{eq:collar}, and
$\lambda\subset\bd M$ is a smooth non-empty closed $1$-manifold.
Let $F$ be a connected orientable incompressible surface with boundary
$\lambda$, let $\cF$ be the surfaces of Definition~\ref{conv:ps} properly
isotopic to $F$ rel $\lambda$, and put
$I=\inf\{\Area(G):G\in\cF\}$.  Assume $F'\in\cF$ has $\Area(F')=I$.
This is an assumption for the auxiliary regularity assertion, not an
existence input to the main proof.  Proposition~\ref{prop:R1fixed} is an implication
about a given minimiser.

Write $\Gamma:=\Gamma_{F'}$ for a graph as in Definition~\ref{conv:ps}.
The closures of the components of $F'\setminus\Gamma$ are the pieces.  Each
piece is smooth in its interior and embedded of class $C^1$ up to its own
boundary, including corners; no higher boundary regularity is assumed.

\begin{lem}[Stationarity]\label{lem:localdisc}
Let $X$ be a smooth vector field on $M$ with compact support in $\Int M$.  Then
\begin{equation}\label{eq:firstvar}
   \sum_P\int_P\operatorname{div}_PX\,dA\;=\;0 ,
\end{equation}
the sum being over the pieces $P$ of $F'$.  Consequently
\begin{enumerate}
\item[\textup{(i)}] every piece is a minimal surface; and
\item[\textup{(ii)}] along the interior of every arc of $\Gamma$ that lies in
$\Int F'$, the two sheets of $F'$ meeting there have opposite outward
conormals, $\eta_1+\eta_2=0$.
\end{enumerate}
\end{lem}

\begin{proof}
Let $\{\varphi_t\}$ be the flow of $X$.  Each $\varphi_t$ is a diffeomorphism
of $M$ equal to the identity near $\bd M$, so $\varphi_t(F')$ is properly
embedded, is piecewise smooth in the sense of
Definition~\ref{conv:ps} --- that definition is invariant under
diffeomorphisms of $M$, with $\Gamma_{\varphi_t(F')}=\varphi_t(\Gamma)$ --- and
is isotopic to $F'$ rel $\bd F'$.  Hence $\varphi_t(F')\in\cF$ and
$\Area(\varphi_t(F'))\ge I=\Area(F')$, with equality at $t=0$; so the
derivative at $t=0$ vanishes, and that derivative is
$\sum_P\int_P\operatorname{div}_PX$, the integrand being continuous on each
piece because the piece is $C^1$ up to its boundary.  This is
\eqref{eq:firstvar}.  Note which properties of $\cF$ have been used: only that
it is closed under diffeomorphisms of $M$ fixing $\bd M$ pointwise and that
$F'$ minimises area over it.  The abstract stationarity conclusion uses these two properties.  Its
expression as a sum over pieces, and the conormal calculation below, use the
structure supplied by Definition~\ref{conv:ps}.

For (i), let $q\in\Int P$.  Since $F'$ is compact and embedded there is a ball
$B\subset\Int M$ about $q$ with $F'\cap B\subset\Int P$.  For $X$ supported in
$B$ only one term survives in \eqref{eq:firstvar}, and the divergence theorem
on the smooth surface $\Int P$ turns it into
$-\int_P\langle H_P,X\rangle=0$; as $X$ is arbitrary, $H_P=0$ near $q$.

For (ii), let $q$ be an interior point of an arc $e\subset\Gamma\cap\Int F'$.
The surface $F'$ is a disc near $q$ and $e$ is an embedded arc in it, so $e$
separates that disc into exactly two sheets, lying in pieces $P_1,P_2$; and
there is a ball $B\subset\Int M$ about $q$ with $F'\cap B$ contained in their
union.  Fix $X$ supported in $B$.  In this neighbourhood of the open edge,
choose inward parallel exhaustion arcs on each $P_i$, converging to
$e\cap B$ in $C^1$, and complete them to boundaries of compact subsurfaces
$P_i^\delta\subset\Int P_i$ outside the support of $X$.  No convergence
through a corner of a piece is required.  Since
$H_{P_i}=0$ by (i), the divergence theorem gives
$\int_{P_i^\delta}\operatorname{div}_{P_i}X=\int_{\bd P_i^\delta}
\langle\eta_i,X\rangle$; both sides converge as $\delta\to0$, the left because
the integrand is continuous on $P_i$ and the right because $\eta_i$ extends
continuously to $\bd P_i$, the piece being $C^1$ up to its boundary.  So
\eqref{eq:firstvar} becomes
$\int_{e\cap B}\langle\eta_1+\eta_2,X\rangle=0$ for every such $X$, whence
$\eta_1+\eta_2=0$ on $e\cap B$.
\end{proof}

\begin{lem}[Sheets with opposite conormals glue smoothly]\label{lem:ball}
Let $P_1,P_2\subset M$ be embedded surfaces, smooth and minimal in their
interiors and $C^1$ up to their boundaries, meeting along
a common boundary arc $e$ and otherwise disjoint near an interior point $q$ of
$e$, and suppose their outward conormals satisfy $\eta_1+\eta_2=0$ along $e$.
Then $P_1\cup P_2$ is a smooth minimal surface near $q$.
\end{lem}

\begin{proof}
At each point of $e$ the vectors $\eta_1,\eta_2$ are unit, orthogonal to the
tangent of $e$, and tangent to $P_1,P_2$ respectively; $\eta_2=-\eta_1$
therefore gives
$T P_1=\operatorname{span}(\dot e,\eta_1)=\operatorname{span}(\dot e,\eta_2)
=T P_2$ along $e$.  So $P_1$ and $P_2$ have a common tangent plane $T$ at each
point of $e$ near $q$.

Choose coordinates near $q$ with $T=\{y=0\}$.  For a small enough
neighbourhood, $P_1$ and $P_2$ are graphs $y=u_1$, $y=u_2$ over the two closed
sides of the projection of $e$ in $\{y=0\}$, and $u_1=u_2$ and $Du_1=Du_2$ on
that projection; so the function $u$ equal to $u_i$ on the $i$th side is $C^1$
on a full disc $D$.  Write the minimal surface equation for a graph in the
ambient metric in divergence form,
$\bd_i\bigl(a^i(x,u,Du)\bigr)+b(x,u,Du)=0$, with $a$ and $b$ smooth and the
equation uniformly elliptic on bounded gradients.  Each $u_i$ solves it
classically on its own side; for a test function $\varphi$ supported in $D$,
integrating by parts on each side produces two interface integrals of
$a^i(x,u,Du)\nu_i\varphi$ along the projection of $e$, and these cancel because
$Du$ is continuous across it and $\nu_1=-\nu_2$.  Hence $u$ is a Lipschitz weak
solution on all of $D$.

Interior regularity for quasilinear divergence-form elliptic equations now
applies: difference quotients of $u$ solve linear divergence-form equations
with bounded measurable uniformly elliptic coefficients, so De Giorgi--Nash--%
Moser gives $Du\in C^{0,\alpha}$, after which the coefficients
$a^i(x,u,Du)$ are $C^{0,\alpha}$ and Schauder bootstrapping gives
$u\in C^\infty$; see \cite[2nd ed., Chs.~6 and~8]{GT}.  So $P_1\cup P_2$ is a
smooth surface near $q$, and it is minimal because $u$ solves the equation.
\end{proof}

\begin{lem}[Hopf at a boundary point]\label{lem:pivot}
Let $p\in\bd M$, let $T\subset T_pM$ be a plane, and let $U\subset T$ be a
closed planar region which near $0\in\bd U$ is bounded by two $C^2$ arcs
with distinct tangent rays and interior opening angle in $[\pi,2\pi)$.
The arcs may be straight, so a circular sector is included.  Let $W$
be the graph over $U\cap B_\delta(0)$ of a $C^1$ map into the orthogonal
complement of $T$ which vanishes to first order at $0$; so $p\in W$ and
$T_pW=T$.  Assume that $W\subset M$, that $W$ is a minimal surface over the
interior of $U$, and that $W$ meets $\bd M$ only over $\bd U$ --- over one
edge, say, or over the vertex alone.  Then
\[
   dr_p|_T\neq0,\qquad\text{equivalently}\qquad T\neq T_p\bd M .
\]
\end{lem}

\begin{proof}
Put $u:=r\circ\Phi$ on $U\cap B_\delta(0)$, where $\Phi$ is the graph map.
Since $W\subset M=\{r\ge0\}$ we have $u\ge0$; $u$ vanishes at $0$ and, by
hypothesis, nowhere on $\Int U$, so $u>0$ there; and $u$ is
smooth over $\Int U$ and $C^1$ up to $0$.

For a minimal surface $W$ and a smooth function $r$,
\[
   \Delta_Wu=\operatorname{tr}_W\Hess r+\langle\operatorname{grad}r,H\rangle
   =\operatorname{tr}_W\Hess r ,
\]
$H$ vanishing.  In the collar $\Hess r=\varphi\varphi'\,g_{\bd M}\le0$ on
vectors tangent to the level sets of $r$, by \eqref{eq:hessr}, while
$\Hess r(\bd_r,\cdot)=0$ because $r$ is a geodesic coordinate; so $\Hess r\le0$
as a quadratic form and $\Delta_Wu\le0$.  In the graph coordinates $x^1,x^2$
write the induced metric as $h_{ij}$.  The matrix $(h^{ij})$ is continuous
up to $0$ and uniformly positive definite.  The first-order coefficients
are bounded as well; this needs minimality, since the graph is assumed only
$C^1$ at $0$.  Indeed, in smooth ambient coordinates with graph map
$\Phi=(x^1,x^2,w)$, the minimal immersion equation gives, for $a=1,2$,
\[
 b^a:=\Delta_W x^a
   =-h^{ij}\widetilde\Gamma^a_{bc}(\Phi)
                  \partial_i\Phi^b\partial_j\Phi^c,
\]
where $\widetilde\Gamma$ denotes the ambient Christoffel symbols.
Thus $b^a$ extends continuously to $0$, and on the interior
$\Delta_W=h^{ij}\partial_i\partial_j+b^a\partial_a$ has bounded
coefficients, uniform ellipticity, and no zeroth-order term.

There is an open disc $B\subset\Int U$ tangent to $\bd U$ at $0$.
For an opening angle greater than $\pi$, choose a smaller straight sector
of opening still greater than $\pi$ inside $U$ near $0$ and place such a
disc in it.  For opening angle $\pi$, the two $C^2$ boundary arcs have
opposite tangent rays and together form a $C^{1,1}$ graph near $0$;
its quadratic bound supplies an interior tangent disc.  Shrink $B$ to
lie in the graph neighbourhood.  On $B$ we have $-u<0=-u(0)$, the interior
sphere condition holds, and the outward normal derivative
exists because $u$ is $C^1$ up to $0$.  Hopf's boundary point lemma
\cite[2nd ed., Lemma~3.4, p.~34]{GT} gives $\bd_\nu(-u)(0)>0$, that is
$\bd_\nu u(0)<0$ for the outward $\nu$.  The differential of $\Phi$ at $0$ is
the inclusion $T\hookrightarrow T_pM$, so $\bd_\nu u(0)=dr_p(\nu)$, whence
$dr_p|_T\ne0$.
\end{proof}

The lemma consumes less convexity than the divergence-form comparison it
replaces: only $\Hess r\le0$, that is $\varphi'\le0$ in \eqref{eq:collar}, and
not the strict positivity of the mean curvature of $\bd M$.  It is used twice
below, once inside Proposition~\ref{prop:R1fixed} and once, in the simplest
case $U$ a half-plane, as the whole of Lemma~\ref{lem:R2}.

The argument below never compares $F'$ with a least area disc and never
leaves the piecewise smooth category.  It assumes an attained minimum;
it does not supply the smooth-class existence theorem proved above.

\begin{prop}[Minimisers with fixed boundary are smooth up to $\bd M$]
\label{prop:R1fixed}
Let $F,\cF,I,F'$ be as at the start of this subsection.  Then $F'$ is a smooth
embedded minimal surface, smooth up to $\bd M$, with $F'\cap\bd M=\bd F'
=\lambda$; it is neat and stable.
\end{prop}

\begin{proof}
By clause (i) of Lemma~\ref{lem:localdisc} every piece is a minimal surface,
and by clause (ii) of that lemma together with Lemma~\ref{lem:ball} the surface
$F'$ is smooth and minimal near every interior point of every arc of $\Gamma$
lying in $\Int F'$.  Three things remain: the vertices of $\Gamma$ in
$\Int F'$, the points of $\lambda$, and stability.

\emph{Vertices in the interior.}  Let $q$ be a vertex of $\Gamma$ in
$\Int F'$.  Finitely many arcs of $\Gamma$ emanate from $q$ and cut a disc
neighbourhood $N$ of $q$ in $F'$ into closed sectors $S_1,\dots,S_m$ in cyclic
order, each contained in a piece, hence minimal in its interior and embedded
of class $C^1$ up to its own boundary, in particular up to $q$.  Along the interior of each shared arc
the two adjacent sectors have the same tangent plane, by the previous
paragraph; and the tangent plane of a sector extends continuously to $q$.  So
all the $S_i$ have one and the same tangent plane $T$ at $q$.

Take coordinates with $q=0$ and $T=\{y=0\}$, and shrink $N$ so that each $S_i$
is the graph of a $C^1$ function over a closed sector $\Sigma_i\subset T$ with
vertex $0$ and opening angle $\theta_i$, the $\Sigma_i$ consecutive.  Here
$\theta_i\in(0,2\pi)$, the two boundary arcs of $S_i$ having distinct tangent
rays at $q$ by the corner clause of Definition~\ref{conv:ps}; no bound
$\theta_i<\pi$ is available, and none is used.
Let $\pi$ denote the orthogonal projection onto $T$.  Then $\pi$ carries
$N\setminus\{q\}$ onto a punctured disc $D_\varepsilon\setminus\{0\}$ as a
covering map of some degree $k$, with $2\pi k=\sum_i\theta_i$.  Over each point
of $D_\varepsilon\setminus\{0\}$ the $k$ preimages are distinct points of the
embedded surface $F'$, hence have pairwise distinct heights; the heights depend
continuously on the point and never coincide, so they order the sheets
consistently and the covering is trivial.  But $N\setminus\{q\}$, a disc minus
an interior point, is connected.  Hence $k=1$.

So $N$ is the graph of a function $u$ over $D_\varepsilon$.  It is $C^1$, the
sectors having matching tangent planes along the arcs; it is smooth off the
arcs; and it is a weak solution of the minimal surface equation on
$D_\varepsilon\setminus\{0\}$, the interface terms cancelling as in the proof
of Lemma~\ref{lem:ball}.  It is then a weak solution on all of
$D_\varepsilon$: cutting a test function off at radius $\delta$ about $0$
changes the weak formulation by at most $C\delta^{-1}\cdot\pi\delta^2\to0$,
$u$ being Lipschitz.  By the regularity quoted in Lemma~\ref{lem:ball}, $u$ is
smooth.  So $F'$ is smooth near $q$, and $F'$ is a smooth minimal surface on
all of $\Int F'$.

\emph{Points of $\lambda$.}  Since $\Gamma$ is properly embedded, a point
$p\in\lambda$ either lies off $\Gamma$ or is one of the finitely many endpoints
of arcs of $\Gamma$ on $\lambda$.  In the first case some relative
neighbourhood of $p$ misses $\Gamma$ altogether, so lies in
$F'\setminus\Gamma$, which by Definition~\ref{conv:ps} is a smooth submanifold
of $M$ --- smooth, that is, up to $\bd M$ and including its boundary points ---
and $F'$ is smooth up to $\bd M$ at $p$.  (The clause about the pieces would
give only $C^1$ here, which is why it is not the clause used.)  So only the
finitely many endpoints remain; fix such a $p$.

Those arcs cut a half-disc neighbourhood $N$ of $p$ in $F'$ into closed sectors
$S_0,\dots,S_m$, where $S_0$ and $S_m$ carry the two arcs $\ell_0,\ell_m$ into
which $\lambda$ is divided at $p$.  Exactly as in the interior case, all the
$S_i$ share a tangent plane $T$ at $p$; take coordinates with $p=0$,
$T=\{y=0\}$, and write $\Sigma_i\subset T$ for the projected sector of $S_i$.
Its boundary arcs are smooth projections of the arcs of $\Gamma$ and
$\lambda$, and its tangent opening is $\theta_i\in(0,2\pi)$ as before.
Put $\Theta:=\sum_i\theta_i$.  The extreme tangent rays of
$\Sigma_0$ and $\Sigma_m$ are the two tangent rays of $\lambda$ at
$p$, which are opposite; hence
\begin{equation}\label{eq:anglecong}
   \Theta\equiv\pi \pmod{2\pi},\qquad\text{and in particular }\Theta\ge\pi .
\end{equation}

What is wanted is a sector of opening angle in $[\pi,2\pi)$ over which $F'$ is
a graph.  Let $j$ be least with $\Theta_j:=\theta_0+\dots+\theta_j\ge\pi$; such
a $j$ exists by \eqref{eq:anglecong}.  If $\Theta_j<2\pi$, put
$U:=\Sigma_0\cup\dots\cup\Sigma_j$ and $W:=S_0\cup\dots\cup S_j$.  If instead
$\Theta_j\ge2\pi$ then $j\ge1$, since $\theta_0<2\pi$, and $\Theta_{j-1}<\pi$
by the minimality of $j$, so that the single angle $\theta_j$ already exceeds
$\pi$; put $U:=\Sigma_j$ and $W:=S_j$.

Either way, after shrinking the neighbourhood, $U$ is a possibly curved
sector with smooth boundary arcs and tangent opening in $[\pi,2\pi)$,
and $W$ is a $C^1$ graph over it.  The projected sectors fit along their
shared arcs; the opening strictly below $2\pi$ keeps the two extreme
arcs distinct near the vertex.  The surface $W$ lies in $M$ and is minimal
over $\Int U$, by the previous
step; and $W$ meets $\bd M$ only over $\bd U$, because $F'\cap\bd M=\lambda$
by properness while $\lambda\cap N=\ell_0\cup\ell_m$, and $\ell_0,\ell_m$ are
the extreme edges, carried by $S_0$ and $S_m$.  Lemma~\ref{lem:pivot}
applies and gives $dr_p|_T\ne0$.

Now let $v$ be a unit vector tangent at $p$ to one of the sectors.  That sector
is $C^1$ at $p$ and contained in $M=\{r\ge0\}$, so its points at parameter $t$
in the direction $v$ have $r=t\,dr_p(v)+o(t)\ge0$, forcing $dr_p(v)\ge0$.  The
sectors are in cyclic order about $p$ and consecutive ones share an edge, so
the set of such $v$ is the image in the unit circle of $T$ of an interval of
length $\Theta$, swept monotonically from the tangent ray of $\ell_0$; and it
lies in the closed half-circle $\{dr_p\ge0\}$, whose
length is $\pi$ because $dr_p|_T\ne0$.  Were $\Theta\ge2\pi$ that sweep would
cover the whole circle, which no half-circle contains; so $\Theta<2\pi$, the
sweep is injective, its image is an arc of angular length $\Theta$, and
$\Theta\le\pi$.  With
\eqref{eq:anglecong}, $\Theta=\pi$.  In particular each $\theta_i<\pi$ after
all; that bound is a conclusion here, not a hypothesis, which is why the
selection of $U$ above could not be allowed to use it.

So the projected sectors $\Sigma_i$ tile a one-sided neighbourhood
$H\subset T$ of the projection of $\lambda$, and $N$ is the graph of a
function $u$ over $H$.  It is $C^1$ on $\overline H$ and, by the previous step
and the cancellation of the interface terms as in Lemma~\ref{lem:ball}, a weak
solution of the minimal surface equation on $\Int H$.  The curved side of
$\bd H$ is the projection of $\lambda$, a smooth curve, so near $0$ the domain
$H$ has smooth boundary; and there $u$ takes the boundary values given by
$\lambda$ itself, which are smooth.

Freeze the equation to give the first boundary regularity improvement explicitly.  In the
interior of $H$ expand the minimal-graph equation as
\[
 \begin{gathered}
   Lu:=A^{ij}(x)D_{ij}u=f_0(x),\qquad
   A^{ij}(x)=a^i_{p_j}(x,u,Du),\\
   f_0=-b-\sum_i a^i_{x_i}-\sum_i a^i_zD_i u,
 \end{gathered}
\]
where the functions on the right are evaluated at $(x,u,Du)$.
The matrix $A$ is symmetric (it is the Hessian of the graph-area
integrand in its gradient variables), continuous on $\overline H$,
and uniformly positive definite after shrinking $H$; $f_0$ is
continuous and bounded.  Here we use the already established
$C^1$ regularity, not a prior H\"older estimate for $Du$.

Choose a bounded smooth planar domain $\Omega\subset H$ whose
boundary agrees with $\bd H$ near $0$ and is capped off inside $H$
away from $0$.  A smooth boundary-flattening chart constructs such a
domain by rounding the two distant corners of a half-disc.  Let
$\varphi$ be a smooth extension of the prescribed boundary height,
and choose a smooth cutoff $\chi$, equal to $1$ near $0$ and zero
near the artificial part of $\bd\Omega$, with support so small that
the remaining boundary on its support is part of $\bd H$.
Then $w=\chi(u-\varphi)$ is $C^1$ on $\overline\Omega$, smooth in
$\Omega$, and zero on $\bd\Omega$, and
\[
   Lw=h:=\chi(f_0-L\varphi)
       +2A^{ij}(D_i\chi)D_j(u-\varphi)
       +(u-\varphi)L\chi\in C^0(\overline\Omega).
\]
Continuous coefficients on $\overline\Omega$ are bounded and have
vanishing mean oscillation: the mean oscillation over
$\Omega\cap B(x,r)$ is bounded by their continuity modulus at $2r$.
Thus \cite[Lemma~3.4, pp.~756--757]{CavaliereTransirico}, applied to
$-L$ with both lower-order coefficients zero and $p=4$, gives
$v\in W^{2,4}(\Omega)\cap W^{1,4}_0(\Omega)$ satisfying $Lv=h$.
Both $v$ and $w$ belong to
$W^{2,2}_{\mathrm{loc}}(\Omega)\cap C^0(\overline\Omega)$ and have
the same zero boundary values.  The linear uniqueness theorem
\cite[Theorem~9.5, p.~225]{GT} therefore gives $v=w$; its condition
(9.3) holds here because $L$ has zero first- and zeroth-order terms.
Sobolev embedding yields $u=\varphi+v\in C^{1,1/2}$ near $0$.
The coefficients of the graph equation are now H\"older continuous
there; boundary Schauder regularity and bootstrapping
\cite[Ch.~6]{GT} give $u\in C^\infty$ up to the prescribed boundary
arc.  Hence $F'$ is smoothly embedded up to $\bd M$ near $p$.

It is worth recording that this last step, and not a matching of jets, is what
lets Definition~\ref{conv:ps} ask only $C^1$ of the pieces.

\emph{Globally and stability.}  The two steps give a smooth structure near
every point of the compact set $F'$, so $F'$ is a smooth embedded minimal
surface, smooth up to $\bd M$; and $F'\cap\bd M=\bd F'=\lambda$ is part of the
hypothesis that $F'$ is properly embedded, which is the membership condition of
$\cF$.  At a boundary point its smooth graph over its tangent plane has a
smooth planar boundary arc, of opening angle $\pi$.  Lemma~\ref{lem:pivot}
therefore gives $dr|_{TF'}\ne0$, so the embedding is neat.  This uses only the
already proved fixed-boundary regularity and does not appeal to the later
sliding-boundary theorem.  Finally $F'$ has least area among the piecewise smooth surfaces
isotopic to it rel $\bd F'$, so all nearby normal deformations fixing the
boundary increase area: $F'$ is stable.
\end{proof}

\begin{thm}[\textup{(R1)}: boundary regularity of relative minimisers]
\label{thm:R1}
Use the metric and foliation of \S\ref{ssec:data}.  Let $F\subset E(K)$
be connected, orientable, incompressible, smooth and neatly embedded,
with $\bd F$ a leaf of $J$, and put
\begin{multline*}
 A:=\inf\bigl\{\Area(G):G\text{ smooth and neatly embedded,}\\
 G\text{ smoothly properly isotopic to }F,\quad\bd G\subset J\bigr\}.
\end{multline*}
Then there is a smooth neat stable minimal surface $\Sigma$ in this class
with $\Area(\Sigma)=A$.  It is smooth up to its boundary, and
$\Sigma\cap\bd E(K)=\bd\Sigma$.
\end{thm}

\begin{proof}
Take a smooth neat minimising sequence $G_i$ for $A$.  Each is
incompressible.  The fixed-boundary theorem gives a smooth neat stable
minimal surface $F_i'$ in its relative smooth ambient isotopy class,
with $\bd F_i'=\bd G_i$ and
\[
 A\leq\Area(F_i')=I_{\mathrm{sm}}(G_i)
   \leq\Area(G_i)\longrightarrow A.
\]
These are smooth proper embeddings of the same compact connected source
surface $S$, in one smooth proper isotopy class.  They belong to
$\cM(S)$ of \cite[p.~897]{Schultens}, and for large $i$ their areas are at
most $a:=A+1$.  Thus they belong to $\cM_a(S)$.

The stability used here is stability under variations fixing the boundary.
Small smooth such variations remain neat and in the relative isotopy class,
so fixed-boundary area minimality supplies it.  This is the stability used
in \cite[Proposition~A.1]{Schultens}; sliding-boundary minimality has not
been assumed to produce this sequence.

By \cite[Corollary~A.3, p.~898]{Schultens}, $\cM_a(S)$ is a disjoint union
of finitely many open and closed subsets, each consisting of mutually
isotopic surfaces.  Pass to a subsequence in one such subset $\cN$.
By \cite[Proposition~A.1, p.~897]{Schultens}, $\cM_a(S)$ is compact, so
after another subsequence $F_i'\to\Sigma$ in $C^1$ with
$\Sigma\in\cM_a(S)$.  Closedness gives $\Sigma\in\cN$, hence it is smoothly
properly isotopic to $F$.  Continuity of area gives $\Area(\Sigma)=A$.

Membership in $\cM(S)$ means that $\Sigma$ is a stable minimal surface
which is the image of a smooth proper embedding
$f\colon(S,\bd S)\to(E(K),\bd E(K))$, with
$f^{-1}(\bd E(K))=\bd S$ and boundary a parametrised curve of $J$.
The independent boundary argument of Lemma~\ref{lem:R2} makes it neat.
Thus it is itself a competitor for $A$ and gives the claimed attainment.
\end{proof}

The same proof works for the link exteriors of \S\ref{ssec:links}.
Theorem~\ref{thm:fixed-smooth} allows the prescribed curve on each boundary
torus.  Kapovich's standing hypothesis \cite[p.~897]{Schultens} requires
$f^{-1}(\bd_iM)$ to be one component of $\bd S$ for each ambient boundary
component, which is exactly the spanning-surface convention here.  We use
this as the $n$-component form of \textup{(R1)}.

\begin{lem}[\textup{(R2)}: transversality along the boundary]\label{lem:R2}
Let $\Sigma\subset E(K)$ be a smooth properly embedded minimal surface,
smooth up to its boundary, with $\Sigma\cap\bd E(K)=\bd\Sigma$.  Then
$T_p\Sigma\neq T_p\bd E(K)$ for every $p\in\bd\Sigma$; that is, $\Sigma$ meets
$\bd E(K)$ transversally along $\bd\Sigma$, and is therefore neatly embedded.
\end{lem}

\begin{proof}
This is Lemma~\ref{lem:pivot} in its simplest case and does not require
any attainment theorem.  Let $p\in\bd\Sigma$ and $T:=T_p\Sigma$.  In a
boundary chart at $p$ the surface $\Sigma$ is the graph over a smooth planar
half-neighbourhood $U\subset T$, of tangent opening angle $\pi$, of a
smooth map vanishing to first order at $p$; the graph lies in $E(K)$, is
minimal, and meets $\bd E(K)$ exactly over the boundary arc through $0$.
Lemma~\ref{lem:pivot} gives $dr_p|_T\ne0$, that is $T_p\Sigma\ne T_p\bd E(K)$.
\end{proof}

\begin{rem}\label{rem:R2-convexity}
The proof consumes only $\Hess r\le0$, that is $\varphi'\le0$ in
\eqref{eq:collar}, and not the strict positivity of the mean curvature of
$\bd E(K)$ that an earlier version of this argument used through a
divergence-form comparison and an explicit annular barrier.  With
Lemma~\ref{lem:pivot} in hand that machinery is not needed, and
this Hopf argument uses \cite[Lemma~3.4]{GT}.  The gluing argument in
Lemma~\ref{lem:ball} and the boundary bootstrap in
Proposition~\ref{prop:R1fixed} still use weak elliptic regularity from
\cite[Chs.~6 and~8]{GT}.
\end{rem}

\begin{rem}[The route through least area \emph{maps}]\label{rem:alt-route}
Minimising among maps in a relative homotopy class would provide an
alternative to Theorem~\ref{thm:fixed-smooth} only after establishing
fixed-boundary regularity, embeddedness, and return to the prescribed
relative isotopy class.  These are separate requirements: a statement
about minimal maps alone does not supply the representative used here.
Theorem~\ref{thm:fixed-smooth} works directly with embedded competitors in
that class.  No extension of the closed-surface embedding results of
\cite{FHS} is needed for its proof.  Sliding-boundary attainment still
uses the compactness argument of Theorem~\ref{thm:R1}.
\end{rem}

\section{The Kakimizu complex}\label{sec:kakimizu-complex}

For the geometric statements in this section, use the warped-collar metric
and longitudinal foliation of \S\ref{sec:geometric-inputs}.  Smooth
fixed-boundary attainment is Theorem~\ref{thm:fixed-smooth}; minimisation
with boundary allowed to vary through the foliation is \textup{(E)}(ii).

\subsection{Surfaces, vertices, simplices}

A \emph{knot} is a smooth oriented embedding $K\hookrightarrow S^3$.  Choose an
open tubular neighbourhood $N(K)$ and write $E(K)=S^3\setminus\Int N(K)$ for
the compact \emph{exterior}; its boundary is a torus.  The \emph{preferred
longitude} is the oriented slope on $\bd E(K)$ whose linking number with $K$ is
zero.  A \emph{Seifert surface in the exterior} is a compact, connected,
orientable, neatly embedded surface $S\subset E(K)$ whose oriented boundary is
a preferred longitude; since $S$ is connected with a single boundary component,
\begin{equation}\label{eq:euler-genus}
   \chi(S)=1-2g(S),
\end{equation}
and $g(K)=\min\{g(S)\}$.  A \emph{compressing disc} for $S$ is an embedded disc
$D\subset E(K)$ with $D\cap S=\bd D$ such that $\bd D$ is essential on $S$;
if none exists, $S$ is \emph{incompressible}.  For a two-sided properly
embedded surface, this is equivalent to $\pi_1$-injectivity, by the Loop
Theorem in the form of \cite[Corollary~3.3, p.~48]{Hatcher3M}; the reverse
implication follows directly from a compressing disc.  This identifies
the convention here with that of \S\ref{sec:intro}.  \emph{Boundary
incompressibility is not required}, and no step below uses it; see
\S\ref{ssec:definitions-used}.

An \emph{ambient isotopy} is a smooth family of diffeomorphisms
$H_t\colon E(K)\to E(K)$ with $H_0=\mathrm{id}$; a \emph{normal isotopy} is a
smooth family $f_t\colon S\to E(K)$ of neat embeddings.  Two surfaces determine
the same vertex if and only if they are ambient isotopic in $E(K)$.  The
boundary longitude is allowed to move on the boundary torus: the isotopy is not
required to fix any chosen boundary curve pointwise.  That an ambient isotopy
restricts to a normal isotopy is immediate; the converse is
Lemma~\ref{lem:isotopy-extension}, so on Seifert surfaces the two equivalence
relations agree.  The sliding-boundary area minimum of
\S\ref{sec:geometric-inputs} is taken over smooth neat representatives of
this class and therefore computes an invariant of a vertex.  The tame
cut-and-paste isotopies used below have smooth neat endpoints;
Lemma~\ref{lem:category-bridge} identifies their endpoint classes with
these smooth ambient isotopy classes.

The exterior is irreducible \cite[p.~228]{Kak92}; this is used in
Lemma~\ref{lem:clean-discs}.

\begin{dfn}\label{dfn:complexes}
The vertices of $\IS(K)$ are the ambient isotopy classes of incompressible
Seifert surfaces in $E(K)$.  A finite set of pairwise distinct vertices spans a
simplex if and only if its classes admit \emph{simultaneously} pairwise
disjoint representatives.  For an integer $\ell\ge g(K)$,
\[
   \IS_\ell(K)=\IS(K)\bigl[\{[S]:g(S)\le\ell\}\bigr]
\]
is the \emph{full} subcomplex on the vertices of genus at most $\ell$; in
particular $\IS_{g(K)}(K)=\MS(K)$ is the minimal-genus complex.  We write
$\dist(u,w)$ for the graph distance in the $1$-skeleton.  A complex is
\emph{flag} if every finite clique in its $1$-skeleton spans a simplex.
\end{dfn}

Two notions of disjointness must be kept apart.  An \emph{edge} speaks only of
a single pair of vertices, and different edges may use mutually incompatible
representatives; a \emph{simplex of dimension at least $2$} requires
representatives chosen once and for all that are pairwise disjoint as a family.
The passage from the former to the latter is Proposition~\ref{prop:flag}, whose
proof uses \textup{(E)} and \textup{(U)}; it is never implicit.

Finally, $\dist$ is not merely a graph metric but Kakimizu's own distance: by
\cite[Proposition~3.1(1), p.~231]{Kak92} the minimum length of an edge path in
$\IS(L)$ between two classes equals the distance $d$ obtained by counting lifts
to the infinite cyclic cover.  We use this in
\S\ref{sec:exchange-lemma}, where the hypothesis of the exchange lemma is
$\dist(u,w)=2$ and what the construction is handed is a pair of surfaces at
Kakimizu distance $2$.  The statement is strictly stronger than connectedness
and does not appear in version~1 of this work.

\subsection{Which definitions these are}\label{ssec:definitions-used}

Three conventions have to be pinned down before Kakimizu's theorems may be
quoted, because the literature does not agree on them.

\emph{Vertices.}  We use Kakimizu's definition, not the paraphrase of it that
has entered the literature.  A spanning surface for a link $L$ is, in
\cite[p.~225]{Kak92}, a surface $S=\Sigma\cap E(L)$ with $\Sigma$ oriented in
$S^3$, $\bd\Sigma=L$, $\Sigma$ having no closed component and \emph{possibly
disconnected}, and $\Sigma\cap N(L)$ a collar of $\bd\Sigma$; it is
incompressible if \emph{each component} is incompressible in $E(L)$.  In
particular $\bd$-incompressibility is not required.  Przytycki and Schultens
describe the same complex as having for vertices the classes of surfaces
``incompressible and $\bd$-incompressible'' \cite{PS12}; that description adds
a condition which is not in \cite{Kak92}, and a manuscript which inherits it
owes a lemma identifying the two vertex sets.  We do not impose it, so the
vertex set of Definition~\ref{dfn:complexes} is Kakimizu's on the nose; and the
condition is in any case automatic, which we record now, both because it
settles the discrepancy and because it is the hypothesis under which the
rigidity theorems of \cite{Wald68} are stated.

\begin{lembdyprime}[Incompressible spanning surfaces are $\bd$-incompressible]
Let $L$ be a link in $S^3$ and let $\Sigma\subset E(L)$ be an incompressible
spanning surface, neatly embedded, with $\bd\Sigma$ the union of the
longitudes.  Then $\Sigma$ is $\bd$-incompressible: there is no embedded disc
$D\subset E(L)$ with $D\cap\Sigma=a$ and $D\cap\bd E(L)=b$, where
$\bd D=a\cup b$ and $a$ is an arc which is essential in $\Sigma$, that is,
which does not cut a disc off $\Sigma$ together with an arc of $\bd\Sigma$.
For $n=1$ this says that every incompressible Seifert surface in a knot
exterior is $\bd$-incompressible.
\end{lembdyprime}

\begin{proof}
Suppose such a $D$ exists.  Since $D\cap\Sigma=a$ and
$\lambda:=\Sigma\cap\bd E(L)=\bd\Sigma$ lies in $\Sigma$, the interior of $b$
misses $\lambda$; so $b$ is an arc properly embedded in the surface obtained by
cutting $\bd E(L)$ along $\lambda$, and both endpoints of $b$ lie on the same
torus $T_i$, on which the cut yields an annulus $A_i$ with boundary circles
$\lambda_i^+,\lambda_i^-$, the two sides of $\lambda_i$ in $T_i$.  A properly
embedded arc in an annulus either has both endpoints on the same boundary
circle, and is then $\bd$-parallel, or joins the two.

\emph{Both endpoints on the same circle.}  Then $b$ together with an arc
$\lambda_0\subset\lambda_i$ bounds a disc $\Delta\subset A_i$, so
$\Int\Delta\cap\lambda=\emptyset$.  Now $D\cup\Delta$ is an embedded disc, the
two meeting exactly in $b$, its boundary is the closed curve
$a\cup\lambda_0\subset\Sigma$, and its interior misses $\Sigma$; pushing
$\Delta$ off $\bd E(L)$ makes it a disc meeting $\Sigma$ exactly in its
boundary.  If $a\cup\lambda_0$ is essential in $\Sigma$ this is a compressing
disc, contradicting incompressibility.  If it is inessential it bounds a disc
in $\Sigma$, so $a$ cuts a disc off $\Sigma$ together with an arc of
$\bd\Sigma$ and was not essential.

\emph{Endpoints on the two different circles.}  Near a point of $\lambda_i$ the
pair $(E(L),\Sigma)$ is the pair $(\{z\ge0\},\{x=0,z\ge0\})$, so the two sides
of $\lambda_i$ in $T_i$ are the traces of the two sides of $\Sigma$.  Perturb
$\bd D$ by pushing $a$ off $\Sigma$ to the side on which $D$ lies.  The
interior of $b$ misses $\Sigma$, since $b\subset\bd E(L)$ and
$\Sigma\cap\bd E(L)=\lambda$, so the resulting closed curve meets $\Sigma$ only
near the two endpoints of $b$; and at such an endpoint it crosses or does not
according as $b$ leaves the endpoint on the far side of $\Sigma$ from $D$ or on
the near one.  In the case at hand $b$ leaves its two endpoints on opposite
sides, so exactly one of them is the far side and there is exactly one
crossing.  Hence the algebraic intersection number of $\bd D$ with $\Sigma$ is
$\pm1$.  But $\Sigma$ is Poincar\'e dual to the total linking homomorphism
$\varphi\colon H_1(E(L))\to\Z$, $\gamma\mapsto\lk(\gamma,L)$, which is the
homomorphism of \S\ref{ssec:links} --- nothing proved there is needed here, the
duality holding for any spanning surface of any link by a count of meridians.
So that number is $\varphi(\bd D)$, which
vanishes because $\bd D$ bounds the disc $D$ in $E(L)$.
\end{proof}

The only remaining discrepancy, connectedness, is
vacuous for a knot: a spanning surface for a knot has no closed component and
exactly one boundary circle, hence is connected, so ``connected and
incompressible'' and ``incompressible componentwise'' select the same surfaces.

\emph{Non-splitness.}  Kakimizu's Theorem~A is stated for a non-split oriented
link, and the hypothesis is used, not decorative: it is what makes $E(L)$
irreducible.  For a knot it is vacuous.  It returns in
Theorem~\hyperref[thm:D]{D}, where it is assumed.

\emph{Edges.}  Kakimizu spans an edge on two classes with disjoint
representatives \cite[1.3(b), p.~227]{Kak92}; Przytycki and Schultens instead
require Kakimizu distance one, and point out that the two conditions differ for
links, precisely because spanning surfaces there are allowed to be
disconnected.  We use the disjointness convention, as in
Definition~\ref{dfn:complexes}.  For a knot the two agree, since spanning
surfaces are then connected, so the discrepancy is vacuous here; for
Theorem~\hyperref[thm:D]{D} it is neutralised by the linking condition, which
forces every spanning surface to be connected.

\subsection{Connectedness, flagness, non-emptiness}

\begin{prop}[Connectedness]\label{prop:connected}
$\IS(\mathrm{unknot})$ is a single point.  For every non-trivial knot $K$, the
complex $\IS(K)$ is connected.
\end{prop}

\begin{proof}
The exterior of the unknot is a solid torus $V$, the complementary handlebody
of a genus-one Heegaard splitting of $S^3$, and the preferred longitude is a
meridian of $V$.  If $S\subset V$ is an incompressible Seifert surface then
$\pi_1(S)\to\pi_1(V)\cong\Z$ is injective by the Loop Theorem criterion
just cited.  A free group of
rank $2g(S)\ge2$ does not embed in $\Z$, so $g(S)=0$ and $S$ is a meridian
disc; two meridian discs may be isotoped to share their boundary, made disjoint
by innermost-circle surgery, and then the sphere they assemble bounds a ball
by Alexander--Schoenflies, which isotopes one to the other.  For non-trivial
$K$ the statement is \cite[Theorem~A, p.~226]{Kak92}, quoted in
\S\ref{sec:intro}.
\end{proof}

The engine behind that theorem is \cite[Theorem~2.1, p.~228]{Kak92}, whose
proof splits on p.~229 into a minimal-genus case and an incompressible case; it
is the second that we need.  Nothing is claimed here about the connectedness of
a truncation $\IS_\ell(K)$.  That is proved in \S\ref{sec:applications}, as
part of Theorem~\hyperref[thm:C]{C} and by way of
Corollary~\ref{cor:initial-isometric}; the implication runs from the global
statement to the truncations and never in the reverse direction, and an
independent source is recorded in Remark~\ref{rem:chen-shen}.

\begin{prop}[Flagness]\label{prop:flag}
$\IS(K)$ is flag, and hence so is every full subcomplex $\IS_\ell(K)$.
\end{prop}

\begin{proof}
Let $\{v_1,\dots,v_n\}$ be a clique in the $1$-skeleton.  Its vertices are
distinct, so the corresponding surfaces are pairwise non-isotopic and all
incompressible.  By \textup{(E)}(ii) fix, once and for all and one class at a time,
a relative area minimiser $A_i$ in the class $v_i$.  Let $i\neq j$.  Since
$\{v_i,v_j\}$ is an edge, the two classes contain disjoint representatives, so
Theorem~\ref{thm:U}(i) applies to the pair $A_i,A_j$ and yields that they are
disjoint or coincide; and $v_i\neq v_j$ excludes coincidence.  The $A_i$ were
fixed before any pair was examined and do not depend on $(i,j)$, so
$\{A_1,\dots,A_n\}$ is a family of \emph{simultaneously} pairwise disjoint
representatives, which is the simplex required.  Fullness makes the truncations
retain the same simplices.
\end{proof}

The proof consumes only the two-surface clause of \textup{(U)}; the
$n$-surface clause is not needed, because the choice made in advance already
converts pairwise information into simultaneous information.  Flagness is
therefore obtained here from \textup{(U)} and the smooth attainment
conclusion \textup{(E)}(ii), based on Theorem~\ref{thm:fixed-smooth}.  The same
mechanism, applied to a vertex and two of its neighbours, is what the exchange
lemma is handed, and we isolate it here because it is the exact point at which
the argument for links breaks.

\begin{prop}[Disjointness from two neighbours at once]\label{prop:disjoint}
Let $u,w$ be vertices of $\IS(K)$ with $\dist(u,w)=2$, let $x\in\cI(u,w)$ be a
common neighbour, and let $S_+\in u$, $S_-\in w$ and $S\in x$ be relative area
minimisers supplied by \textup{(E)}, chosen one class at a time.  Then
\[
   S\cap(S_+\cup S_-)=\emptyset .
\]
\end{prop}

\begin{proof}
Since $x$ and $u$ are adjacent, their classes contain disjoint
representatives, so Theorem~\ref{thm:U}(i) applies to $S$ and $S_+$: they are
disjoint or they coincide.  A Seifert surface for a knot is connected, so
``coincide'' means that $S$ and $S_+$ are equal as sets and hence that $x=u$,
which is excluded.  Therefore $S\cap S_+=\emptyset$, and the same argument on
the edge $\{x,w\}$ gives $S\cap S_-=\emptyset$.  The three minimisers were
fixed before any pair was examined, so the two conclusions hold for one and the
same $S$.
\end{proof}

The step that fails for a general link is the sentence beginning ``A Seifert
surface for a knot is connected''.  For a link, Kakimizu's spanning surfaces
may be disconnected, Theorem~\ref{thm:U}(i) then yields only that each
component of $S$ is disjoint from or equal to a component of $S_+$, and two
distinct vertices may share a component; the intersection $S\cap S_+$ need not
be empty.  This is the obstruction analysed in \S\ref{ssec:general-links}, and
the linking condition of Theorem~\hyperref[thm:D]{D} removes it by forcing
every spanning surface to be connected.

\begin{prop}[Non-emptiness]\label{prop:nonempty}
$\IS(K)\neq\emptyset$ for every knot $K$, and $\IS_\ell(K)\neq\emptyset$ for
every integer $\ell\ge g(K)$.
\end{prop}

\begin{proof}
Seifert's algorithm, applied to a regular diagram of $K$, produces a compact
connected oriented surface $S_0\subset S^3$ with $\bd S_0=K$; made transverse
to $N(K)$, it meets it in an annulus, and $S=S_0\cap E(K)$ is a Seifert surface
in the exterior; connectedness of $S_0$ is part of the definition over which
the existence statement quantifies \cite[Definition~2.1 and Theorem~2.2,
pp.~15--16]{Lickorish}.
If $S$ is incompressible then $[S]$ is a vertex.  Otherwise compress and keep
the component carrying the longitudinal boundary; by Lemma~\ref{lem:euler} each
such step strictly lowers the genus, so after finitely many steps the process
stops at an incompressible Seifert surface.  For the truncations, a Seifert
surface of least genus is incompressible, since a compression would produce one
of smaller genus, so it is a vertex of genus $g(K)$ and lies in every
$\IS_\ell(K)$ with $\ell\ge g(K)$.
\end{proof}

\subsection{Compression and the robustness of neighbours}

Throughout this subsection $A$ and $B$ are transverse, neatly embedded compact
surfaces.  Cutting both along $A\cap B$, the resulting components are called
\emph{patches} and the intersection curves and arcs duplicated on their
boundaries are called \emph{exchange seams}; a \emph{disc patch} is a disc
component of the result, never a component of a set-theoretic symmetric
difference.  This is the vocabulary of \S\ref{sec:exchange-lemma}.

\begin{lem}[Euler characteristic and compression]\label{lem:euler}
Let $A,B$ be compact and transverse.  Cut along $A\cap B$ and reglue by the
orientation-compatible normal exchange, and let $P_{\mathrm{tot}}$ be the union
of the output components.  Then
$\chi(P_{\mathrm{tot}})=\chi(A)+\chi(B)$, so that if
$P_{\mathrm{tot}}=P\sqcup C_1\sqcup\dots\sqcup C_m$ then
\begin{equation}\label{eq:euler-split}
   \chi(P)=\chi(A)+\chi(B)-\sum_j\chi(C_j).
\end{equation}
Moreover, compressing a connected Seifert surface along an essential circle and
keeping the component that carries the longitudinal boundary strictly lowers
the genus; iterating therefore terminates at an incompressible Seifert surface.
\end{lem}

\begin{proof}
Triangulate $A\cup B$ so that each intersection curve is a subcomplex.  Cutting
duplicates exactly the cells of the intersection trace and the normal regluing
identifies the same cells again, so the Euler correction is the same before and
after the exchange; additivity over components gives \eqref{eq:euler-split}.
For the second statement: compressing along a non-separating circle removes a
handle and keeps the unique boundary component, so the genus drops by one;
compressing along a separating essential circle discards a closed side, which
has positive genus since otherwise the circle would bound a disc, and this
again lowers the genus of the component with boundary.  The genus is a
non-negative integer, so the process is finite, and it stops only where no
compression is available.
\end{proof}

\begin{lem}[Prescribed compressions performed away from a fixed neighbour]
\label{lem:clean-discs}
Let $M$ be a compact irreducible orientable smooth $3$-manifold, let $P\subset M$ be
connected, two-sided and smoothly neatly embedded, with non-empty boundary,
and let $Q\subset M$ be incompressible, smoothly neatly embedded and disjoint
from $P$.  Given any finite compression sequence for $P$, each step retaining
a specified connected component with non-empty boundary,
there is a compression sequence whose intermediate and final ambient isotopy
classes agree with the given ones one by one and all of whose discs are
disjoint from $Q$.  In particular the final surface has a representative
disjoint from $Q$.
\end{lem}

\begin{proof}
It suffices to clean a single prescribed compressing disc $D$ and then induct
along the sequence, transporting the next prescribed disc to the current
representative at each step.  Put $D$ transverse to $Q$.  As
$\bd D\subset P$ and $P\cap Q=\emptyset$, the intersection $D\cap Q$ consists
of circles.  Each is null-homotopic in $M$, being the boundary of a subdisc of
$D$, hence null-homotopic in $Q$ by incompressibility, hence bounds a disc in
$Q$; choose such a disc $E\subset Q$ innermost, so $\Int E\cap D=\emptyset$.
Its boundary circle bounds in $D$ a unique subdisc $F$ not containing $\bd D$,
and $E\cup F$ is a sphere with a corner along that circle; rounding the corner
by the qualitative part of Lemma~\ref{lem:round}, at a width small enough that
the tube misses $P$, gives a smoothly embedded sphere, and by irreducibility it
bounds a ball $B$.  By clause (5) of that lemma the rounding does not change
the region bounded, up to the tube; so $B$ agrees off the tube with the region
bounded by $E\cup F$, and we use it in that form below.  The sphere is
disjoint from $P$; since $B$ lies in the interior of $M$ while $P$ is connected
with its boundary on $\bd M$, the ball is disjoint from $P$ as well.
Pushing $F$ across $B$ to a parallel copy of $E$
is a locally flat piecewise smooth isotopy of $D$ rel $\bd D$, keeps the disc
interiors off $P$, and removes that circle together with every intersection
circle inside $F$.  Repeating produces $D'$ disjoint from $Q$ and isotopic to
$D$ in that category rel boundary.  Round the resulting disc and the two
compressed surfaces at positive widths, keeping the collars of $P$ fixed;
the widths can be chosen so small that the cleaned disc and its compressed
surface remain disjoint from $Q$.  The two smooth neat boundary-bearing
outputs are piecewise smooth isotopic rel their unchanged boundary, hence
smoothly ambient isotopic by Lemma~\ref{lem:category-bridge}.  Use this smooth
ambient diffeomorphism when transporting the next prescribed disc.
The component with boundary remains
disjoint from $Q$ and furnishes the input for the next step; the discarded
closed components play no role.
\end{proof}

\begin{cor}[Robustness of neighbours]\label{cor:robust}
Let $P$ be connected, two-sided and smoothly neatly embedded, with non-empty boundary and disjoint
from an incompressible surface $Q$.  Then \emph{every} ambient isotopy class
obtained by compressing $P$ down to an incompressible descendant with boundary
is, in $\IS(K)$, equal or adjacent to $[Q]$.
\end{cor}

\begin{proof}
Apply Lemma~\ref{lem:clean-discs} to the chosen compression sequence; its final
representative is disjoint from $Q$.  If the two classes coincide they are
equal, and otherwise their disjoint representatives span an edge.
\end{proof}

\subsection{Normal isotopies are ambient}

\begin{lem}[Isotopy extension; {\cite[Theorem~2.4.6, p.~52]{Wall2016}}]
\label{lem:isotopy-extension}
Let $M$ and $F$ be compact smooth manifolds with boundary and let
$f\colon F\times[0,1]\to M$ be a smooth isotopy through neat embeddings, that
is, $f_t^{-1}(\bd M)=\bd F$ and $f_t(F)$ meets $\bd M$ transversally along its
boundary.  Then there is a smooth ambient isotopy $H_t\colon M\to M$ with
$H_0=\mathrm{id}_M$ and $H_t\circ f_0=f_t$ for all $t$.
\end{lem}

Wall's hypotheses are compactness of $F$ and neatness of the embeddings; there
is no orientability hypothesis and no dimension condition relating $F$ to $M$,
and the ambient isotopy is covered by a diffeotopy of $M$, which is what we use.

Taking $M=E(K)$, which is compact with boundary, and $F$ a Seifert surface,
every normal isotopy is a smooth family of neat embeddings, so the lemma
applies: a normal isotopy class, as used in relative area minimisation, is
exactly an ambient isotopy vertex.

\begin{rem}[No silent strengthening]\label{rem:setwise}
The ambient isotopy produced above preserves $\bd E(K)$ \emph{setwise}, and the
longitudinal boundary curve is allowed to move.  Strengthening this to ``fixing
the boundary pointwise'' would make the statement false: on
$M=S^1\times[0,1]$ the rotational isotopy of the boundary is realised by no
ambient isotopy fixing the boundary pointwise.  In citations, and in
formalisation, this must be written setwise.
\end{rem}

\subsection{Comparison with the piecewise smooth category}

\begin{lem}[The category bridge]\label{lem:category-bridge}
Let $M$ be a compact orientable smooth $3$-manifold and let $F_0,F_1$ be
compact two-sided neat smooth surfaces, each component having non-empty
boundary.  A locally flat ambient isotopy from $F_0$ to $F_1$ fixing their
common boundary pointwise can be replaced, as an isotopy of images, by a
smooth ambient isotopy fixing $\bd M$ pointwise.

If $\bd M$ consists of tori and each $F_i$ has one oriented essential boundary
curve on each torus, the same comparison holds for isotopies allowing those
curves to move; the resulting smooth ambient isotopy then preserves $\bd M$
setwise.  In particular the piecewise smooth isotopy classes used in
Definition~\ref{conv:ps}, restricted to smooth neat spanning surfaces, are
exactly the smooth ambient isotopy classes used for vertices.
\end{lem}

\begin{proof}
We specify the relative comparison result being used.  For a regular
decomposition $M=N\cup\overline{M\setminus N}$ with smooth separating
interface $E$, Cerf compares the orbit of $E$ under boundary-fixing
homeomorphisms with its orbit under $C^r$ diffeomorphisms whose $r$-jet on
$\bd M$ is the identity.  The inclusion induces a bijection on path
components \cite[III, \S3.2.1, Corollary~3, pp.~366--367]{CerfEmbeddings}.
The hypothesis called the Smale conjecture in that statement is supplied by
\cite{HatcherSmale}.

Two details concerning this citation are useful.  Endpoint parametrisations
may be made smooth: a homeomorphism between compact smooth surfaces is
isotopic to a diffeomorphism \cite[Theorem~1.13, p.~42]{FarbMargalit}.
If it is already the identity on a boundary collar, this can be done relative
to a smaller collar: triangulate outside that collar, straighten the finitely
many image vertices and edges, smooth their disc neighbourhoods, and fill the
remaining discs; the Alexander isotopy on each disc fixes its boundary.
Extend this change of parametrisation through a bicollar of $E$.  Next, a
smooth endpoint in the topological orbit belongs to the differentiable orbit:
approximate its ambient homeomorphism by a diffeomorphism, using density in
\cite[III, \S3.1.3(1), p.~364]{CerfEmbeddings}; the remaining smooth embedding
is $C^0$-close to the inclusion and lies in its differentiable orbit by the
$C^0$ openness in part~(2) of that same theorem.  Corollary~3 can therefore be
applied to both endpoints.  Take $r\ge2$; the resulting path of $C^r$
embeddings can first be straightened in a boundary collar, relative to the
endpoints, and then approximated relative to the endpoints and that collar
by a smooth isotopy, since embeddings and transversality are open in $C^1$
on a compact parameter interval.  Smooth isotopy extension then supplies a
smooth ambient isotopy.  Thus no differentiability in the parameter is being
inferred from a merely continuous path.

First suppose the given isotopy $H_t$ fixes
$\lambda=\bd F_0=\bd F_1$ pointwise.  Smoothly straighten the endpoints in a
common boundary collar, so that there they are $\lambda\times[0,\epsilon]$.
Write $b_t=H_t|_{\bd M}$.  For a collar cutoff $\chi$ equal to $1$ at the
boundary and $0$ outside that collar, the homeomorphisms
\[
   C_t(x,r)=\bigl(b_{t\chi(r)}^{-1}(x),r\bigr)
\]
extend by the identity and make $C_tH_t$ fix all of $\bd M$.  At the final
time $C_1$ preserves $F_1$, because every $b_s$ fixes $\lambda$.  We can also
arrange that the isotopy fixes a smaller boundary collar: compress $M$ into
the complement of that collar, conjugate the isotopy there, and extend by
the identity.  Since the endpoint surfaces are products in the collar, this
does not change either endpoint image.

Choose product neighbourhoods $N_i\cong F_i\times[-1,1]$ with identical
boundary product coordinates.  The transported neighbourhood $H_1(N_0)$ and
$N_1$ have isotopic bicollars relative to $F_1$ and the fixed boundary collar.
This is collar uniqueness after cutting along $F_1$
\cite[Theorem~2, pp.~123--124]{ArmstrongCollars}: perform every shrinking and
sliding step relative to a smaller common collar, taking all shrinking
factors, including the initial half-collar factor, equal to $1$ there and
making their transitions inside the common collar.  Consequently the
ordered double frontiers
\[
   E_i=F_i\times\{-1,1\}
      =N_i\cap\overline{M\setminus N_i}
\]
are topologically ambient isotopic rel $\bd M$.  These are separating regular
interfaces, even when $F_i$ itself does not separate $M$, so the comparison
just established applies.  Its smooth ambient isotopy carries $N_0$ to
$N_1$: the side is specified by the annular bands $N_i\cap\bd M$, which are
fixed, and every component of $F_i$ has boundary.

The image of the middle surface of $N_0$ is now smoothly isotopic inside
$N_1$ to its positive frontier, using the transported product coordinates,
and that frontier is smoothly isotopic to $F_1$, using the chosen product
coordinates of $N_1$.  These two products have the same boundary coordinates.
The boundary therefore moves out and back along one fixed annular coordinate.
If its displacement is $a(t)$, apply the flow for time $-a(t)$ of the
corresponding smooth boundary vector field, extended into a collar of $M$.
This correction is the identity at both endpoints and fixes the boundary
curve throughout the corrected surface isotopy.  Its velocity is zero on
$\lambda$; extending that velocity by a vector field zero on $\bd M$ gives
the relative form of smooth isotopy extension.  This proves the first claim.

For the moving-boundary claim, first use a smooth collar isotopy to make the
endpoint boundary curves agree and then straighten the endpoint surfaces in
that collar.  The endpoint boundary homeomorphism $g$ preserves each such
curve $\lambda$ and is isotopic to the identity on its torus.  It can be
isotoped to the identity while preserving $\lambda$ setwise.  Indeed, first
undo its restriction to $\lambda$ by interpolating increasing lifts of
circle homeomorphisms.  After cutting the torus along $\lambda$, the remaining
homeomorphism is a boundary-fixed annulus homeomorphism.  Its twist number
$n$ sends a transverse homology class $\mu$ to $\mu+n[\lambda]$; hence $n=0$.
A spanning arc can then be fixed by the disc bigon isotopy, and cutting along
it leaves a disc, where the Alexander isotopy completes the argument.
Extending this boundary isotopy into the collar preserves the product
endpoint surface.  Append it to $H_t$.  Its boundary path $b_t$ is now a
based loop, although it need not preserve $\lambda$ at intermediate times.

By \cite[Theorem~1.5.2, p.~217]{HamstromHomeotopy},
$\pi_1(\operatorname{Homeo}_0(T^2))\cong\mathbb Z^2$.  Evaluation at a point
has the smooth translation section $T^2\to\operatorname{Homeo}_0(T^2)$.
On fundamental groups this is a split surjection from $\mathbb Z^2$ to
$\mathbb Z^2$, hence an isomorphism.  Thus, on each boundary torus, $b_t$ is
homotopic as a based loop to a smooth translation loop $\beta_t$.  Extend
$\beta_t$ to a smooth collar-supported ambient isotopy $B_t$.  The boundary
loop of $B_t^{-1}H_t$ has a based nullhomotopy $c(t,u)$, with
$c(t,0)=\beta_t^{-1}b_t$, $c(t,1)=\mathrm{id}$ and
$c(0,u)=c(1,u)=\mathrm{id}$.  For a collar function $\psi$ equal to $0$ at
the boundary and $1$ outside the collar, set
\[
   D_t(x,r)=\bigl(c(t,\psi(r))^{-1}(x),r\bigr).
\]
Then $D_tB_t^{-1}H_t$ fixes $\bd M$ throughout, and $D_0=D_1=\mathrm{id}$.
The first part gives a smooth ambient isotopy from $F_0$ to $B_1^{-1}(F_1)$;
following it by $B_t$ gives the required smooth isotopy to $F_1$.
\end{proof}

\section{Complexity and the frontier construction}\label{sec:complexity}

Use the warped-collar metric and longitudinal foliation of
\S\ref{sec:geometric-inputs}.  Theorem~\ref{thm:fixed-smooth} supplies
smooth fixed-boundary minimisers, and \textup{(E)}(ii) supplies the
sliding-boundary minima used here.

Two things are assembled here.  The first is the complexity function that will
be fed to Definition~\ref{dfn:exchange}, together with the one fact about it
that is not formal, namely that the set of values it attains is well-ordered.
The second is the cyclic-cover construction which, given a vertex and two of
its neighbours, manufactures the pair of candidate surfaces on which
\S\ref{sec:exchange-lemma} operates.

\subsection{The complexity}

\begin{dfn}[Complexity]\label{dfn:complexity}
Let $v$ be a vertex of $\IS(K)$.  Write $g(v)$ for the genus of any
representative, which by \eqref{eq:euler-genus} is determined by $v$, and
\[
   A(v)\;:=\;\inf\bigl\{\Area(F)\;:\;F\in v,\ \bd F\subset J,\ 
                      F\text{ is smooth and neat}\bigr\},
\]
the infimum of area over smooth neat representatives of the smooth ambient
isotopy class $v$, with boundary on a leaf of the longitudinal foliation.
The area-refined subcomplexes of Theorem~\hyperref[thm:C]{C} use this
infimum.  Set
\[
   c(v)\;:=\;\bigl(g(v),A(v)\bigr)\in\N\times\R_{>0},
\]
and order $\N\times\R_{>0}$ lexicographically.  Finally put
\[
   W\;:=\;c\bigl(\IS(K)^{(0)}\bigr),
\]
the set of \emph{attained} values, with the induced order.
\end{dfn}

The infimum is attained, by \textup{(E)}(ii) of Theorem~\ref{thm:E} applied to
the normal isotopy class of a representative --- which is the vertex, by
Lemma~\ref{lem:isotopy-extension} --- so $A(v)$ is a minimum and is realised by
a stable minimal surface, smoothly neatly embedded up to its boundary.  Both coordinates
are therefore invariants of $v$ and not of a representative.

It is $W$, and not $\N\times\R_{>0}$, that is required to be well-ordered in
Definition~\ref{dfn:exchange}, and the distinction is not cosmetic:
$\N\times\R_{>0}$ is not well-ordered, since $\R_{>0}$ is not.  What has to be
proved is that the values actually taken by $c$ are sparse.  That is the
content of the next lemma, and it is the only place in this paper where a
statement about the totality of vertices of $\IS(K)$ is needed.

\begin{lem}[Well-foundedness of the complexity]\label{lem:wf}
Let $K$ be a non-trivial knot, with the relative metric and the foliation $J$
of \S\ref{sec:geometric-inputs}.  Then:
\begin{enumerate}
\item[\textup{(i)}] for every $h\in\N$ and every $a>0$ the set
      $\{v\in\IS(K)^{(0)}:g(v)=h,\ A(v)\le a\}$ is finite;
\item[\textup{(ii)}] consequently, for each $h$ the set
      $W_h:=\{A(v):g(v)=h\}$ is a closed discrete subset of $\R_{>0}$ and is
      well-ordered of order type at most $\omega$;
\item[\textup{(iii)}] $W$ is well-ordered of order type at most $\omega^2$; in
      particular there is no infinite strictly descending chain in $W$, and
      every non-empty subset of $W$ has a least element;
\item[\textup{(iv)}] $\IS(K)$ has countably many vertices.
\end{enumerate}
\end{lem}

\begin{proof}
Fix $h$ and let $S$ be the compact orientable surface of genus $h$ with one
boundary circle.  We use the compactness theorem of Kapovich's appendix to
\cite{Schultens} in the form printed there, and we transcribe its membership
conditions rather than paraphrasing them.  Its setting
\cite[p.~897]{Schultens} is a $P^2$-irreducible compact Riemannian
$3$-manifold $M$ with smooth strictly convex boundary, together with a compact
family $\cF$ of smooth curves on $\bd M$, the case of interest being that in
which $\bd M$ is a single torus and $\cF$ is a smooth foliation of it by closed
curves.  For a proper embedding $f\colon(S,\bd S)\to(M,\bd M)$ whose class is
admissible --- meaning that $f^{-1}(\bd M)$ is a single component of $\bd S$
and $f(S)$ is incompressible --- one writes $M([f])$ for the set of stable
minimal surfaces in the proper isotopy class $[f]$ whose boundary is a
parametrised curve of $\cF$, modulo reparametrisation, and
\[
   \cM(S)\;=\;\bigcup_{[f]}M([f]),
\]
the union being over all such classes, topologised by $C^1$ convergence;
$\cM_a$ denotes the subset of surfaces of area at most $a$.  Two features of
this definition are used below and are worth naming.  The union runs over
\emph{all} admissible proper isotopy classes at once, and the boundary of a
member is required only to be a curve of $\cF$, not a prescribed one.

These hypotheses hold for $M=E(K)$, $\cF=J$: the exterior is compact and
orientable and is irreducible \cite[p.~228]{Kak92}, and an orientable
$3$-manifold contains no two-sided projective plane, so $E(K)$ is
$P^2$-irreducible; $\bd E(K)$ is a torus, made strictly convex by
\eqref{eq:collar}; and $J$ is a smooth --- indeed real-analytic --- foliation
of it by closed curves.  A Seifert surface in the exterior is properly
embedded, incompressible, and meets $\bd E(K)$ in its single boundary circle,
so its class is admissible.

Now let $v$ be a vertex with $g(v)=h$ and $A(v)\le a$, and let $\Sigma_v$ be a
minimiser supplied by \textup{(E)}(ii).  It is a stable minimal surface with
boundary a leaf of $J$ and area $A(v)\le a$, so $\Sigma_v\in\cM_a$.  By
\cite[Corollary~A.3, p.~898]{Schultens} the set $\cM_a$ is the disjoint union
of finitely many open and closed subsets $\cM_a([g])$, each consisting of
surfaces isotopic to one another.  Distinct vertices are distinct ambient
isotopy classes, hence, by Lemma~\ref{lem:isotopy-extension} and the convention
fixed with it, distinct proper isotopy classes; so $v\mapsto$ the piece
containing $\Sigma_v$ is injective on the set in (i), which is therefore finite.

For (ii): $W_h\cap(0,a]$ is the image of the finite set of (i), hence finite,
for every $a>0$; a subset of $\R_{>0}$ meeting every interval $(0,a]$ in a
finite set is closed and discrete in $\R_{>0}$ and is order-isomorphic to an
initial segment of $\N$.  For (iii): lexicographically, $W$ is the ordered sum
$\sum_{h\in\N}W_h$ of well-orders of type at most $\omega$ indexed by
$\N$, hence a well-order of type at most $\omega\cdot\omega=\omega^2$.  For
(iv): each set in (i) is finite and
$\IS(K)^{(0)}=\bigcup_{h\in\N}\bigcup_{n\in\N}\{v:g(v)=h,\ A(v)\le n\}$.
\end{proof}

\begin{rem}[Least elements without choice]\label{rem:wf-zf}
Remark~\ref{rem:choice} records that in general the passage from ``no infinite
descending chain'' to ``every non-empty subset has a least element'' uses
dependent choice.  For the complexity of
Definition~\ref{dfn:complexity} it does not.  Let $\emptyset\neq T\subseteq W$.
The set $\{h:(h,\alpha)\in T\text{ for some }\alpha\}$ is a non-empty subset of
$\N$ and so has a least element $h_0$.  Pick any $\alpha_1$ with
$(h_0,\alpha_1)\in T$ --- a single instance, not a choice function.  By
Lemma~\ref{lem:wf}(i) the set
$\{\alpha:(h_0,\alpha)\in T,\ \alpha\le\alpha_1\}$ is finite and non-empty, so
it has a least element $\alpha_0$, and $(h_0,\alpha_0)=\min T$.  No form of
choice is used.  The use of choice flagged in Remark~\ref{rem:choice} is
elsewhere: in the well-ordering of the fibres in Lemma~\ref{lem:ties}.
\end{rem}

\begin{rem}[Two remarks on the hypotheses of Lemma~\ref{lem:wf}]
\label{rem:wf-honest}
First, the appendix \cite[p.~897]{Schultens} carries two standing hypotheses on
the surface: that it is connected, and that its preimage of each boundary
component of $M$ is a single component of its own boundary.  For a knot both
are automatic, a spanning surface in the sense of
\S\ref{ssec:definitions-used} having no closed component and exactly one
boundary circle; for links they are supplied by
Lemma~\ref{lem:link-connected}.

Second, on this route the hypothesis that $K$ is non-trivial is not consumed.
Every knot exterior, the unknot's included, satisfies those hypotheses, so
the finiteness statement (i) holds for the unknot as well --- where, by
Proposition~\ref{prop:connected}, it is anyway visible.  The hypothesis is kept
in the statement because it is the standing hypothesis of
Theorem~\hyperref[thm:A]{A} and of every result of
\S\ref{sec:exchange-lemma}, and because a version of this lemma proved instead
by normalising surfaces in the infinite cyclic cover \emph{does} consume it.
The exponent in $\omega^2$ is an upper bound only.  Order type at most
$\omega$, together with finite fibres, would imply \textup{(F)} of
\S\ref{ssec:F-discussion}.  The converse does not follow: \textup{(F)}
constrains the lower neighbours of each vertex, not its entire initial
segment.  We do not know whether $\IS(K)$ satisfies \textup{(F)}.
\end{rem}

\subsection{The cyclic-cover frontier construction}

The exchange lemma is handed a vertex $x$ and two vertices $u,w$ adjacent to
it and at distance $2$ from one another.  What
Proposition~\ref{prop:disjoint} supplies is one surface disjoint from two
others which meet each other; the construction below turns that configuration
into two new Seifert surfaces.  It is stated for surfaces, not for classes,
and no minimality is assumed or concluded.

\begin{lem}[Frontier construction]\label{lem:frontier}
Let $S,S_+,S_-$ be smoothly neatly embedded Seifert surfaces in $E(K)$,
with $\bd S_+$ and $\bd S_-$ leaves of $J$, and with
\begin{equation}\label{eq:frontier-hyp}
   S\cap(S_+\cup S_-)=\emptyset ,
\end{equation}
with $S_+$ meeting $S_-$, the pair $S_+,S_-$ transverse, $\bd S_+\cap\bd
S_-=\emptyset$, and with no disc patch produced by cutting $S_+$ and $S_-$
along $S_+\cap S_-$.  Then there
are Seifert surfaces $P_\uparrow,P_\downarrow$ such that
\begin{enumerate}
\item each is disjoint from $S$, $S_+$, $S_-$ and from the other;
\item the ambient isotopy classes of $P_\uparrow,P_\downarrow$ can be fixed
      before $Z$ is chosen: every compact surface $Z$ disjoint from
      $S_+\cup S_-$ is disjoint from simultaneous representatives of these
      same two classes, still satisfying \textup{(1)};
\item writing $T_\uparrow,T_\downarrow$ for the possibly disconnected frontier
      outputs containing $P_\uparrow,P_\downarrow$,
      \begin{equation}\label{eq:frontier-euler}
         \chi(T_\uparrow)+\chi(T_\downarrow)=\chi(S_+)+\chi(S_-),
      \end{equation}
      and every component discarded from either output is a closed surface of
      Euler characteristic at most $0$.
\end{enumerate}
\end{lem}

\begin{proof}
Let $p\colon\widetilde E\to E(K)$ be the infinite cyclic cover determined by
the linking number homomorphism $\pi_1(E(K))\twoheadrightarrow\Z$ and let
$\tau$ generate its positive deck group.  A Seifert surface carries the
primitive class dual to $p$, hence is non-separating, and cutting along $S$
produces a cut-open fundamental domain between a lift $S_0$ and
$S_1=\tau S_0$.  Let $C$ denote its interior in the stacking direction,
so $p|_C\colon C\to E(K)\setminus S$ is a homeomorphism.  Hypothesis
\eqref{eq:frontier-hyp} is what makes $S_+$ and $S_-$ lie in
$E(K)\setminus S$, so the inverse of that homeomorphism lifts them to
$S^+_0,S^-_0\subset C$; this is the only place \eqref{eq:frontier-hyp} is used,
and it is used for both surfaces at once.  Each of $S^\pm_0$ is a wall in the
corresponding stacking of $C$-translates and its complement in $\widetilde E$
has two components; write $M^\pm_\uparrow,M^\pm_\downarrow$ for them, the upper
one being the one containing $S_1$.  Put
\begin{equation}\label{eq:frontiers}
   \widetilde B_\downarrow
     =\mathrm{fr}\bigl(M^-_\downarrow\cap M^+_\downarrow\bigr),\qquad
   \widetilde B_\uparrow
     =\mathrm{fr}\bigl(M^-_\uparrow\cap M^+_\uparrow\bigr).
\end{equation}
These are assembled from complementary subsurfaces of $S^+_0\cup S^-_0$ and are
piecewise smooth with corners exactly along $\Gamma:=S^+_0\cap S^-_0$, which by
transversality and $\bd S_+\cap\bd S_-=\emptyset$ is a finite family of
disjoint smooth closed curves in the interior.  Round the corners by
Lemma~\ref{lem:round} to get smooth neatly embedded
$\widehat B_\uparrow,\widehat B_\downarrow$, then push each into its own open
region at scale $\sigma$ by Lemma~\ref{lem:pushoff}.  The two regions are
opposite, so the results are disjoint from the lifted inputs and from each
other; they lie in $C$, on which $p$ is injective, and project to embedded
surfaces $T_\uparrow(\eta,\sigma),T_\downarrow(\eta,\sigma)$ disjoint from
$S,S_+,S_-$.  Here $\eta$ denotes the two smooth positive rounding widths;
the rounding charts are fixed for this pair of input surfaces.

We record why neither small positive parameter determines an ambient isotopy
class.  For two admissible width choices $\eta^0,\eta^1$, use
$\eta^u=(1-u)\eta^0+u\eta^1$, $0\le u\le1$, separately on the two frontiers.
The pointwise upper and derivative bounds are convex, and
$\min_{\Gamma\times[0,1]}\eta^u>0$.  The positive-width assertion of
Lemma~\ref{lem:round} therefore gives a smooth family of neatly embedded
roundings, constant near the ambient boundary.  The two roundings remain
disjoint, since they occupy opposite open sectors near each seam and have
disjoint sheets away from the seams.

Choose the inward push-off vector fields smoothly along this family.  Near
the boundary all the roundings agree, so use the same leaf-preserving
boundary field as in Lemma~\ref{lem:pushoff}; in the interior extend the
smooth family of inward normals.  A partition of unity preserves the
positive inward component.  Compactness of the parameter interval and of
the surfaces gives one positive push-off size valid throughout the family.
Any two choices of inward field on a fixed rounding can likewise be joined
by their convex interpolation after reducing the push-off size.  Thus
shrinking an initial push-off, varying the width and field, and enlarging
the final push-off gives a smooth isotopy between any two sufficiently
small outputs.  Lemma~\ref{lem:isotopy-extension}, applied to the disjoint
union, makes this an ambient isotopy.  The isotopy can be supported in any
fixed neighbourhood containing these operations: extend its velocity there
and multiply by a cutoff equal to one on the moving part.  In particular
it can fix $S$.  No smooth isotopy with a zero-width endpoint is used.
We may consequently suppress both $\eta$ and $\sigma$ when referring to
the output classes.

Each frontier is the frontier of a region in an orientable cover, hence
two-sided and consistently oriented across the resolved seams, and the
algebraic intersection number of a lifted meridian with its boundary is $+1$.

We count the boundary curves, since the conclusion depends on it.  The
preimage of $\bd E(K)$ in $\widetilde E$ is an annulus, the longitude lifting
and the meridian not, and $\bd S^+_0$ and $\bd S^-_0$ are two disjoint core
circles of it; so one of them separates the other from the positive end.  Say
$\bd S^-_0$ lies below $\bd S^+_0$.  Points of $S^+_0$ near $\bd S^+_0$ then lie
above $S^-_0$, that is in $M^-_\uparrow$, and $S^+_0$ separates $M^+_\uparrow$
from $M^+_\downarrow$; so they are frontier points of
$M^-_\uparrow\cap M^+_\uparrow$ and $\bd S^+_0\subset\widetilde
B_\uparrow$.  Symmetrically, points of $S^-_0$ near $\bd S^-_0$ lie below
$S^+_0$, so $\bd S^-_0\subset\widetilde B_\downarrow$.  Since $\Gamma$ is a
family of closed curves in the interior, cutting along it creates no new
boundary, and $\bd\widetilde B_\uparrow=\bd S^+_0$,
$\bd\widetilde B_\downarrow=\bd S^-_0$: each frontier has exactly one boundary
circle.  Hence each output has exactly one
component with boundary, a connected oriented Seifert surface whose boundary is
a longitude, and the others
are closed.  These are $P_\uparrow,P_\downarrow$, and clause (1) holds.

For (2), first fix any admissible positive widths and push-offs, hence the
two output classes just constructed.  After $Z$ is given, choose a
neighbourhood $N$ of $S_+\cup S_-$ with closure disjoint from $Z\cup S$;
this is possible by compactness and the stated disjointness hypotheses.
Multiply both fixed widths by a sufficiently small number $\lambda>0$.
The width and derivative constraints remain satisfied, and the actual
metric support radii in Lemma~\ref{lem:round} tend uniformly to zero as
$\lambda\downarrow0$, since this pair of input surfaces and its transverse
intersection are fixed.  The new roundings therefore lie in $N$.  Choose
their positive push-offs small enough that they too stay in $N$, in their
respective open regions.  They avoid $Z$ and still satisfy (1).  The
positive-parameter isotopy above identifies their boundary-bearing
components with the already fixed classes of $P_\uparrow,P_\downarrow$.
Only these representatives depend on $Z$; the classes do not.

For (3): the two frontiers exhaust the cut-open $S^+_0\cup S^-_0$, and
Lemma~\ref{lem:euler} applied to this exchange gives
\eqref{eq:frontier-euler}.  A discarded component is closed, and being cut from
the orientable $S_+\cup S_-$ it is orientable, so it has $\chi\le0$ unless it
is a sphere; suppose one were.  If it carried no exchange seam it would be a
whole component of one of the inputs with the intersection curves deleted, and
no such component is closed, the inputs being connected with non-empty
boundary.  So it carries a seam, and an innermost seam on it bounds a
seam-free disc on the sphere: a disc patch, which the hypothesis excludes.
\end{proof}

\begin{rem}\label{rem:frontier-no-genus}
The corresponding construction in the literature assumes minimal genus in its
hypotheses and asserts it in its conclusion.  Nothing of the kind is used or
claimed here: Lemma~\ref{lem:frontier} says nothing about the genus of either
output beyond what \eqref{eq:frontier-euler} forces, and that is all
\S\ref{sec:exchange-lemma} consumes.  The hypothesis
\eqref{eq:frontier-hyp} is, by contrast, indispensable, and it is the exact
sentence that fails for a general link; see \S\ref{ssec:general-links}.
\end{rem}

\section{The exchange lemma}\label{sec:exchange-lemma}

Use the warped-collar metric and longitudinal foliation of
\S\ref{sec:geometric-inputs}.  Relative area minimisers are the smooth neat
minimisers of \textup{(E)}(ii), obtained from
Theorem~\ref{thm:fixed-smooth} and the sliding-boundary compactness argument.

This section proves axiom (N2) for $\IS(K)$.  Given two vertices at distance
$2$ and a common neighbour, it produces two new vertices, each an apex of the
two-interval, whose complexities cannot both fail to drop.  The mechanism is
Lemma~\ref{lem:frontier} applied to relative area minimisers, and the
difficulty is quantitative: rounding the corners of a frontier gains area,
pushing the result off the inputs costs area, and the gain must be pinned down
before the cost --- and, in \S\S\ref{ssec:tangential}--\ref{ssec:T1}, before the further
parameter introduced by making two tangent minimisers transverse.

Throughout, $K$ is a non-trivial knot, and by Theorem~\ref{thm:R1} and
Lemma~\ref{lem:R2} every minimiser supplied by \textup{(E)} may be taken smooth
up to $\bd E(K)$ and is then neatly embedded.

\subsection{Rounding and pushing off}\label{ssec:rounding}

The frontier of Lemma~\ref{lem:frontier} is piecewise smooth with corners,
whereas $A(\cdot)$ is an infimum over smooth neatly embedded surfaces.  Two
steps separate the two, and their order is the pivot of the quantifier
structure below.  Lemma~\ref{lem:round} is stated pointwise --- varying angle,
varying width, gain an integral --- because in \S\ref{ssec:T2} the
angle degenerates near a tangency point and a single global width would
collapse with it.

\begin{lem}[Corner rounding with local quantitative data]\label{lem:round}
Let $M$ be an orientable compact smooth Riemannian $3$-manifold and let
$\Gamma\subset\Int M$ be a finite disjoint union of smooth embedded circles.
Let $\widetilde B$ be an embedded surface, smooth and neat away from
$\Gamma$, with two smooth sheets extending smoothly across each circle and
meeting transversally there.  Suppose the chosen wedge has opening angle
$\alpha(s)\in(0,\pi)$, and write
\begin{equation}\label{eq:angle}
 \theta(s)=\min\{\alpha(s),\pi-\alpha(s)\},\qquad
 m(s)=\cot(\alpha(s)/2),\qquad Q(\theta)=\csc(\theta/2).
\end{equation}
There is an adapted tubular coordinate system $\Xi(s,\xi_1,\xi_2)$,
$s$ arclength on $\Gamma$, with a positive variable radius $\varrho(s)$,
in which the sheets are exactly $\xi_1=m(s)|\xi_2|$ and the metric equals
the Euclidean metric on the zero section.  The radius is chosen so that
the coordinate map is globally injective and its image contains no other
parts of $\widetilde B$.  Fix such a system and continuous positive local
data $\varrho,d,K$, with $K\ge1$, satisfying
\begin{equation}\label{eq:reach}
 \begin{gathered}
 |m'(s)|\le K(s),\qquad
 \|\Xi^*h(s,\xi)-\delta\|\le K(s)|\xi|
       \quad (|\xi|<\varrho(s)),\\
 d(s)\le\min\{1,\operatorname{dist}(s,\bd M)\}.
 \end{gathered}
\end{equation}
The distance to the empty boundary is read as $+\infty$.
These are local bounds in the specified tube; in particular $K(s)$ is
not a supremum over other portions of $\Gamma$.

There is a continuous non-decreasing $g_0:(0,\pi/2]\to(0,1)$ such that,
on putting
\begin{equation}\label{eq:etamax}
 \begin{aligned}
 E(\theta,\varrho,d,K)&:=\min\Bigl\{
  \frac{\varrho}{4Q(\theta)},\frac1{8KQ(\theta)},
  \sqrt{\frac{g_0(\theta)}{8K^2}},\\
 &\hspace{4em}
  \frac{g_0(\theta)}{48KQ(\theta)^2},
  \frac d{8Q(\theta)},1\Bigr\},\\
 H(\theta)&:=\frac{\sqrt{g_0(\theta)}}{8\cot(\theta/2)},\\
 \eta_{\max}(s)&:=E(\theta(s),\varrho(s),d(s),K(s)),\qquad
 \kappa(s):=H(\theta(s)).
 \end{aligned}
\end{equation}
every $C^\infty$ width satisfying
\begin{equation}\label{eq:width}
 0<\eta(s)\le\eta_{\max}(s),\qquad |\eta'(s)|\le\kappa(s)
\end{equation}
gives a compact smooth neatly embedded surface $\widehat B(\eta)$ with:
\begin{enumerate}
\item The change is supported in the \emph{coordinate} tube
      $T_\eta:=\Xi\{(s,\xi):|\xi|<2Q(\theta(s))\eta(s)\}$.
      A point of this tube over $s$ is at ambient distance at most
      $4Q(\theta(s))\eta(s)$ from $s$.  Thus the tube misses $\bd M$ and
      $\bd\widehat B=\bd\widetilde B$.  Here $\eta$ is a coordinate
      half-width, not a metric tube radius.
\item The rounded sheet lies on the closed wedge side.  If that side is an
      open region $U$ bounded by $\widetilde B$, then
      $\widehat B\subset\overline U$ and it has a one-sided collar into $U$.
\item The old and new surfaces are ambiently isotopic in the locally flat
      piecewise smooth category, by an isotopy supported in $T_\eta$.
\item The decrease is non-negative in each coordinate slice and
      \begin{equation}\label{eq:gain}
       \Area(\widehat B(\eta))\le\Area(\widetilde B)
                   -\int_\Gamma g_0(\theta(s))\eta(s)\,ds.
      \end{equation}
\item If $\widetilde B$ is closed and separating, so is $\widehat B$.
      Corresponding complementary regions agree outside $T_\eta$.
\item Any two positive smooth admissible widths in this fixed coordinate
      system give smoothly ambiently isotopic rounded surfaces, rel boundary.
      The isotopy is supported in the union of their coordinate tubes.
      For the two opposite frontiers the two interpolations may be performed
      simultaneously and stay disjoint.
\end{enumerate}
Moreover, if on a subset $V\subset\Gamma$ one has
$\theta\ge\theta_\bullet$, $\varrho\ge\varrho_\bullet$,
$d\ge d_\bullet$ and $K\le K_\bullet$, then on that subset
\begin{equation}\label{eq:etalower}
 \eta_{\max}\ge E(\theta_\bullet,\varrho_\bullet,d_\bullet,K_\bullet)>0,
 \qquad \kappa\ge H(\theta_\bullet)>0.
\end{equation}
For all sectors formed by the same two transverse surfaces one may use common
data: take the minimum of their admissible radii and boundary clearances and
the maximum of their local $K$'s.  The resulting bounds in
\eqref{eq:etamax} are the same for all four sectors.
\end{lem}

\begin{proof}
\emph{The coordinate system and its scope.}
Orient each component of $\Gamma$.  The wedge bisector and the ambient
orientation give a smooth frame of its normal plane bundle.  Start with the
normal exponential map in this frame.  The tubular neighbourhood theorem
gives an injective map on a neighbourhood of the zero section; decrease its
radius to exclude every portion of $\widetilde B$ other than the two chosen
sheet collars.  Compact embeddedness makes these choices possible.  In each
normal fibre, extend the two sheet collars smoothly across the origin and
take smooth defining functions for them.  Their two normal differentials
are linearly independent.  The inverse function theorem applied to the pair
of defining functions, normalized to have their prescribed linear parts,
therefore gives a fibre-preserving change of coordinates with identity
first derivative at the zero section which straightens both sheets exactly.
The defining functions and the normal frame are chosen along the whole
circle, so these local inverses agree on overlaps after restricting their
domains.  Shrinking the radius once more keeps their images in the original
injective normal tube.  The resulting global system satisfies
\begin{equation}\label{eq:straight}
 \widetilde B=\{\xi_1=m(s)|\xi_2|\},\qquad
 \Xi^*h(s,0)=\delta.
\end{equation}
It is smooth, including in $s$.  The fundamental theorem of calculus in each
fibre bounds the metric error by $K(s)|\xi|$; together with $|m'(s)|$,
these estimates admit a continuous positive local majorant $K$.  Increasing
$K$ and decreasing $\varrho$ are allowed.  All coordinate estimates on a
fixed smaller tube depend on finitely many local derivatives of the two
defining functions, the ambient metric and the inverse of their normal
Jacobian.  No derivative bound at a different circle or a distant point is
used.  A radius bound by itself would not exclude unrelated sheet pieces;
their exclusion is part of the choice of tube above.

\emph{The profile.}
Take a smooth odd function $\chi$ with $0\le\chi\le1$ on $[0,1]$,
$\chi=1$ near $1$, and $\chi(0)=0$.  Put
$\Theta(t)=c_0+\int_0^t\chi$, where $c_0=1-\int_0^1\chi>0$.
Then $\Theta$ is even, $\Theta\ge|t|$, and it equals $|t|$ near
$|t|=1$, with all derivatives matching.  Set
\[
 g(m)=\int_{-1}^1\left(\sqrt{1+m^2}
                    -\sqrt{1+m^2\Theta'(t)^2}\right)dt,
 \qquad g_0(\theta)=\tfrac14g(\tan(\theta/2)).
\]
The function $g$ is positive and increasing: its derivative has non-negative
integrand since $u\mapsto u/\sqrt{1+m^2u}$ increases for $0\le u\le1$.
Thus $0<g_0(\theta)<1$, $g_0$ is non-decreasing, and
$g(m)\ge4g_0(\theta)$ for either complementary sector angle.

On $|\xi_2|\le\eta(s)$ replace the wedge by the graph
$\xi_1=F(s,\xi_2):=m(s)\eta(s)\Theta(\xi_2/\eta(s))$.
Since the width is smooth and positive, this is a smooth graph, and it
agrees with the old sheets near the matching edges $|\xi_2|=\eta(s)$.
It is contained
in the closed wedge side; where the profile changes it lies strictly inside
that side.  Its coordinate radius is at most
$\eta\sqrt{1+m^2}\le\eta Q(\theta)$.  The metric estimate and
$\eta\le1/(8KQ)$ imply that $\Xi$ is at most $2$-Lipschitz along the
radial segments in $T_\eta$, so the stated ambient support bound follows.
The caps $\eta\le\varrho/(4Q)$ and $\eta\le d/(8Q)$ keep that support
inside the valid tube and away from the ambient boundary.  Embeddedness
follows from the global injectivity and sheet-exclusion properties of the
tube.  Compactness and neatness give the one-sided collar in (2).

\emph{The area estimate.}
With $t=\xi_2/\eta(s)$,
\begin{equation}\label{eq:Fderiv}
 \begin{aligned}
 F_{\xi_2}&=m\Theta'(t),\qquad
 F_s=\eta m'\Theta(t)+m\eta'(\Theta(t)-t\Theta'(t)),\\
 F_s^2&\le2\eta^2K^2+8m^2\eta'^2.
 \end{aligned}
\end{equation}
The old Euclidean area density is
$\sqrt{1+m^2+m'^2\xi_2^2}\ge\sqrt{1+m^2}$.
Using $\sqrt{a+b}\le\sqrt a+b/(2\sqrt a)$, $a\ge1$, the Euclidean
decrease per unit $s$ is at least
$\eta[g(m)-2\eta^2K^2-8m^2\eta'^2]$.
Both Euclidean profile integrals are at most $3\eta Q$: indeed
$\eta^2K^2\le g_0/8$ and
$F_s^2\le g_0/4+g_0/8<1$, by \eqref{eq:etamax}.
For any tangent plane the metric perturbation changes its area density by
at most $2K\eta Q$ times its Euclidean density, since
$\|h-\delta\|\le K\eta Q\le1/8$.  Hence the two metric errors together
cost at most $12K\eta^2Q^2$, independently of the graph slope, and
\begin{equation}\label{eq:gainslice}
 \begin{aligned}
 \text{decrease per unit }s&\ge
 \eta\bigl[g(m)-2\eta^2K^2-8m^2\eta'^2-12K\eta Q^2\bigr]\\
 &\ge\tfrac{27}{8}g_0(\theta)\eta.
 \end{aligned}
\end{equation}
Here the three losses are respectively at most $g_0/4$, $g_0/8$ and
$g_0/4$, while $g(m)\ge4g_0$.  This proves (4), with a strict margin,
slice by slice.

\emph{Isotopies and common data.}
The old and new graphs cobound a product in each fibre.  Interpolating their
heights gives a locally flat piecewise smooth isotopy, extended within
$T_\eta$ by a fibrewise increasing homeomorphism, fixed outside a slightly
larger neighbourhood of those graphs.  The spare factor $2$ in the support
radius provides this neighbourhood.  This proves (3), and carries
complementary regions to the corresponding regions, proving (5).
This isotopy to the cornered object is not asserted to be smooth.

For two positive smooth admissible widths use instead
\begin{equation}\label{eq:round-interpolation}
 \eta_u=(1-u)\eta_0+u\eta_1,\qquad 0\le u\le1.
\end{equation}
The inequalities \eqref{eq:width} are convex and hence persist.  On the
compact parameter space $\Gamma\times[0,1]$ the width has a positive
minimum, so the profile formula gives a smooth family of neat embeddings,
fixed near the boundary.  Its support is in $T_{\eta_0}\cup T_{\eta_1}$.
Smooth isotopy extension, Lemma~\ref{lem:isotopy-extension}, gives (6),
with support in that open set.  The two opposite wedge sides are disjoint
away from the zero section, and the rounded sheets avoid the zero section;
the two families therefore remain disjoint and may be extended together.

Finally $Q(\theta)$ decreases and $g_0,H$ increase with $\theta$, so
\eqref{eq:etalower} follows directly from \eqref{eq:etamax}, with $E$
non-increasing in its last argument.  For a pair of transverse surfaces
construct the sector charts using the same normal tube.  There are only
four sectors.  Minimum radii and clearances and maximum local $K$'s preserve
all estimates; $m\le\cot(\theta/2)$ for every sector.  Thus the common
$E,H$ apply to all four, even though their individual slopes and derivative
bounds differ.
\end{proof}

\begin{lem}[Pushing off]\label{lem:pushoff}
Let $\widehat B$ be a compact $C^\infty$ neatly embedded surface satisfying
clause (2) of Lemma~\ref{lem:round} for an open region $U$, with
$\bd\widehat B$ a leaf of the
longitudinal foliation.  Then there are $\sigma_0>0$ and $C<\infty$, depending
only on $\widehat B$ and on the ambient metric near it, such that for every
$\sigma\in(0,\sigma_0]$ there is a compact $C^\infty$ neatly embedded
$T(\sigma)$ with $T(\sigma)\subset U$; $\bd T(\sigma)$ a leaf; $T(\sigma)$
isotopic to $\widehat B$ and any two of them isotopic to each other; and
$\bigl|\Area(T(\sigma))-\Area(\widehat B)\bigr|\le C\sigma$.
\end{lem}

\begin{proof}
Let $\nu$ be the unit normal of $\widehat B$ pointing into $U$ and $n$ the
outward normal of $\bd E(K)$, extended through the boundary collar; put
$X:=\nu-\psi(r)\langle\nu,n\rangle n$ with $\psi\equiv1$ near $r=0$ and
supported in the collar.  Then $X$ is tangent to $\bd E(K)$, and neatness with
compactness gives $\langle\nu,n\rangle^2\le1-c$, so
$\langle X,\nu\rangle\ge c>0$.  On $\bd E(K)$ replace $X$ by a field
transverse to $J$ whose flow carries leaves to leaves, oriented towards $U$,
and interpolate in $r$, shrinking the collar so that
$\langle X,\nu\rangle\ge c/2$ persists.  Its flow $\Psi_s$ exists for
$s\in[0,\sigma_0]$, preserves $\bd E(K)$ and carries leaves to leaves; put
$T(\sigma):=\Psi_\sigma(\widehat B)$.  Transversality makes
$(x,s)\mapsto\Psi_s(x)$ an embedding into the one-sided collar for
$\sigma_0$ small, which gives the first three clauses, the last of them because
$\Psi$ is an ambient isotopy and $\Psi_{\sigma'}\Psi_\sigma^{-1}$ joins two
push-offs.  With $K_1:=\sup|\nabla X|$ near $\widehat B$, the Jacobian of
$\Psi_s|_{\widehat B}$ obeys $\bd_s\log J=\operatorname{div}X\le2K_1$ in
absolute value, so
$|\Area(T(\sigma))-\Area(\widehat B)|\le(e^{2K_1\sigma}-1)\Area(\widehat B)
\le C\sigma$ with $C:=2K_1e^{2K_1\sigma_0}\Area(\widehat B)$, independent of
$\sigma$.
\end{proof}

The order of the two steps is forced: a normal push-off of a cornered surface
is not embedded, while the frontier of an inner parallel set is again cornered.
Hence $C$ is determined by $\widehat B$, that is, by the cornered frontier and
the once-chosen width $\eta$: \emph{the rounding gain is fixed before
$\sigma$ is}.

The next lemma excludes a disc in the symmetric difference.  It carries an
error term because in \S\S\ref{ssec:T2}--\ref{ssec:T1} one of the two surfaces is a
perturbation of a minimiser and not a minimiser.

\begin{lem}[No disc patch]\label{lem:no-disc}
Let $A_0,B$ be relative area minimisers in two distinct vertices of $\IS(K)$,
let $A$ be smoothly neatly embedded, incompressible, normally isotopic to $A_0$ with
$\bd A$ a leaf of $J$, and put $e:=\Area(A)-\Area(A_0)\ge0$.  Suppose $A$ and
$B$ are transverse with $\bd A\cap\bd B=\emptyset$, and suppose that cutting
$A$ and $B$ along $\Gamma:=A\cap B$ produces a disc patch.  For a closed curve
$\gamma\subset\Gamma$ and a smooth width admissible for the common sector
bounds of Lemma~\ref{lem:round}, let
$\cG(\gamma)$ denote the rounding gain \eqref{eq:gain} along $\gamma$; by
\eqref{eq:angle} it is one and the same number for either of the two pairs of
opposite sectors at $\gamma$.  Then there are closed curves
$a,b\subset\Gamma$, not necessarily distinct, with
\begin{equation}\label{eq:twogains}
   \cG(a)+\cG(b)\ \le\ e .
\end{equation}
\end{lem}

The gain uses one common admissible width for all sector choices.
At each intersection circle apply the last clause of Lemma~\ref{lem:round}
to the pair of sheets of $A$ and $B$, taking the minimum chart radius and
the maximum local $K$ over the four sectors.  Use also
$d(x)\le\min\{1,d(x,\bd E(K))\}$, which protects every boundary in the
construction.  The bounds $E,H$ of \eqref{eq:etamax} then apply both to
the frontier and to either disc exchange; no equality of the individual
slopes $m$ and $1/m$ is asserted.  The disc replacement below uses the
original sheet pieces exactly.  Thus the same charts, angle and width
apply at the original circle; no transport to a displaced seam is needed.

\begin{proof}
No patch has a boundary arc on its frontier, since the endpoints of such an arc
would lie in $\bd A\cap\bd B=\emptyset$; so $\Gamma$ is a finite disjoint union
of closed curves in $\Int E(K)$ and the frontier of a disc patch is a single
one of them.

\emph{Step $0$: the discs $D_A(\gamma)$.}  A simple closed curve $\gamma$ in
the interior of $A$ bounds at most one disc in $A$: the two sides of a
separating $\gamma$ cannot both be discs, or $A$ would be a sphere, and the
disc side is the one missing $\bd A$.  Write $D_A(\gamma)$ for that disc when
it exists and $A_\gamma$ for the closure of $A\setminus D_A(\gamma)$; since $A$
is connected, a component of $A_\gamma$ missing $\gamma$ would be a component
of $A$, so $A_\gamma$ is connected, and so is $A_\gamma\setminus\gamma$.  Use
the same notation for $B$.  Let $\cD\subset\Gamma$ be the set of curves bounding
a disc in $A$ or in $B$.  For $\gamma\in\cD$ \emph{both} discs exist: $\gamma$
bounds a disc, hence is null-homotopic in $E(K)$; a two-sided incompressible
surface is $\pi_1$-injective, so $\gamma$ is null-homotopic in the other
surface as well; and a null-homotopic simple closed curve on a surface bounds a
disc there.  A disc patch is $D_A(\gamma)$ or $D_B(\gamma)$ for $\gamma$ its
frontier, by the uniqueness just proved, so $\cD\ne\emptyset$.  Since
$D_A(\gamma)\subset A$ and $A\cap B=\Gamma$, the sets
\[
   \cA:=\{\gamma\in\cD:\Int D_A(\gamma)\cap B=\emptyset\},\qquad
   \cB:=\{\gamma\in\cD:\Int D_B(\gamma)\cap A=\emptyset\}
\]
are the sets of curves whose disc on the respective side is a disc patch.

\emph{Step $1$: $\cA$ and $\cB$ are non-empty.}  $\Gamma$ is finite, hence so
is $\cD$.  Let $\gamma\in\cD$ have $D_A(\gamma)$ minimal under inclusion among
$\{D_A(\gamma'):\gamma'\in\cD\}$.  If a curve $\gamma'$ of $\Gamma$ met
$\Int D_A(\gamma)$ it would lie in it, and would cut off a subdisc of
$D_A(\gamma)$, which is $D_A(\gamma')$ by Step $0$; then
$D_A(\gamma')\subsetneq D_A(\gamma)$ contradicts minimality.  So
$\gamma\in\cA$, and symmetrically $\cB\ne\emptyset$.

\emph{Step $2$: one-sided exchange.}  Let $\gamma\in\cA$ and write
$D:=D_A(\gamma)$, $D':=D_B(\gamma)$.  Since $D\cap B=\gamma$,
the union $D\cup D'$ is a locally flat piecewise smooth embedded sphere
in $\Int E(K)$.  Qualitatively round its corner by
Lemma~\ref{lem:round}.  The smooth sphere bounds a ball by irreducibility;
pull this ball back by the rounding isotopy to obtain a tame ball $Q_0$
with boundary exactly $D\cup D'$.  All these operations take place in
$\Int E(K)$, so $Q_0$ misses $\bd E(K)$.  Now
\begin{equation}\label{eq:ballempty}
   B\cap\Int Q_0=\emptyset .
\end{equation}
Indeed, $B_\gamma\setminus\gamma$ is connected by Step $0$, is disjoint
from $D\cup D'$, and contains $\bd B$, which is outside $Q_0$.
It therefore lies outside $Q_0$, while $D'$ lies on $\bd Q_0$.

Define the exact replacement
\[
   B_0:=B_\gamma\cup D,\qquad
   \Area(B_0)=\Area(B)-\Area(D')+\Area(D).
\]
This is a properly embedded piecewise smooth surface with a corner only
at $\gamma$, and with $\bd B_0=\bd B$.  It is ambiently isotopic to $B$
in the tame category.  Here is the ball-slide construction, including
its support.  Choose an exterior collar
$c:(D\cup D')\times[0,\epsilon]\to E(K)$ adapted to $B_\gamma$, so that
$c(\gamma\times[0,\epsilon])$ is its annular collar at $\gamma$.
To construct it, straighten the two rays of $D\cup D'$ and the exterior
ray of $B_\gamma$ by a fibrewise angular homeomorphism in the wedge
coordinates.  The first two rays become a line and the third its outward
normal ray.  A smooth exterior collar tangent to the straightened
annulus then pulls back to the required tame collar.  Shrink its width
to exclude all other parts of $B_\gamma$.  The union
$R:=Q_0\cup c((D\cup D')\times[0,\epsilon])$ is a ball in $\Int E(K)$.
The discs $B\cap R$ and $B_0\cap R$ have the same boundary on $\bd R$.
For the first, the two complementary closures are
$c(D'\times[0,\epsilon])$ and $Q_0\cup c(D\times[0,\epsilon])$;
for the second they are the same expressions with $D,D'$ exchanged.
Each is a ball: it is either $D^2\times[0,\epsilon]$ or a ball with a
collar attached along one boundary disc.  Map the first disc to the
second fixing its boundary, keep $\bd R$ fixed, and extend the resulting
boundary maps over these two balls by Alexander extension.  This
gives a homeomorphism of $R$ rel $\bd R$ carrying $B\cap R$ to
$B_0\cap R$.  Alexander's isotopy makes it an ambient isotopy, extended
by the identity outside $R$.  The support is the slightly larger ball
$R$, not just $Q_0$ with its boundary fixed.  No smooth isotopy to the
cornered endpoint $B_0$ is asserted.

The two sheets of $B_0$ at $\gamma$ are the original sheets of $A$ and
$B$.  Their smaller wedge has angle in $(0,\pi)$ and the same $\theta$
as the four original sectors.  Applying Lemma~\ref{lem:round} with the common admissible width
therefore gives a smooth neat surface $B''$ with
\[
 \Area(B'')\le\Area(B)-\Area(D')+\Area(D)-\cG(\gamma).
\]
The ball slide and the rounding isotopy show that $B$ and $B''$ are
tame ambiently isotopic rel boundary.  They are both smooth and neat,
so Lemma~\ref{lem:category-bridge} puts them in the same smooth ambient
isotopy class rel boundary.  Hence $B''$ is an admissible competitor for
$B$, and $\Area(B'')\ge\Area(B)$ by minimality.  Consequently

\begin{equation}\label{eq:exchA}
   \Area\bigl(D_B(\gamma)\bigr)+\cG(\gamma)
      \ \le\ \Area\bigl(D_A(\gamma)\bigr),
   \qquad \gamma\in\cA .
\end{equation}
The argument used only that the disc in the \emph{first} surface is a patch, so
exchanging the roles of $A$ and $B$ it applies to $\gamma\in\cB$ and produces a
smooth neatly embedded $A''$, isotopic to $A$ and hence to $A_0$, with the same
boundary leaf and
$\Area(A'')\le\Area(A)-\Area(D_A(\gamma))+\Area(D_B(\gamma))-\cG(\gamma)$.
Here $\Area(A'')\ge\Area(A_0)=\Area(A)-e$, whence
\begin{equation}\label{eq:exchB}
   \Area\bigl(D_A(\gamma)\bigr)+\cG(\gamma)
      \ \le\ \Area\bigl(D_B(\gamma)\bigr)+e,
   \qquad \gamma\in\cB .
\end{equation}

\emph{Step $3$: the two curves.}  Put
$\varphi(\gamma):=\Area(D_A(\gamma))-\Area(D_B(\gamma))$ on $\cD$, so that
\eqref{eq:exchA} reads $\varphi\ge\cG$ on $\cA$ and \eqref{eq:exchB} reads
$\varphi\le e-\cG$ on $\cB$.  Choose $a\in\cA$ with $\Area(D_A(a))$ least and
$b\in\cB$ with $\Area(D_B(b))$ least.  Then
$\Area(D_A(a))\le\Area(D_A(b))$: if $b\in\cA$ this is the choice of $a$, and
otherwise $\Int D_A(b)$ meets $\Gamma$, and choosing among the curves of
$\Gamma$ inside it one, $c$, whose disc $D_A(c)$ --- a subdisc of $D_A(b)$, by
Step $0$ --- is minimal under inclusion, the argument of Step $1$ gives
$c\in\cA$ and hence
$\Area(D_A(a))\le\Area(D_A(c))<\Area(D_A(b))$.  Symmetrically
$\Area(D_B(b))\le\Area(D_B(a))$.  Adding the two,
$\varphi(a)\le\varphi(b)$, so
$\cG(a)\le\varphi(a)\le\varphi(b)\le e-\cG(b)$, which is
\eqref{eq:twogains}.
\end{proof}

The exchange is performed one surface at a time because
\eqref{eq:ballempty} is available only on that side: a simultaneous
replacement of both discs would require $D_A(\gamma)$ and $D_B(\gamma)$ to be
patches for one and the same $\gamma$, which no innermost choice supplies.
What replaces it is the pair \eqref{eq:exchA}, \eqref{eq:exchB} at two possibly
distinct curves, glued by the monotonicity of $\varphi$.  Note also that the
proof uses only that $A_0$ and $B$ separately minimise, not that $A$ does;
the no-product-region statements in the literature assume both surfaces
minimal, which the perturbed $A$ of \S\S\ref{ssec:T2}--\ref{ssec:T1} is not.

\subsection{The exchange theorem}\label{ssec:exchange-thm}

\begin{thm}[Incompressible exchange lemma]\label{thm:exchange}
Let $K$ be a non-trivial knot and let $x,u,w$ be
vertices of $\IS(K)$ with
\[
   \dist(x,u)=\dist(x,w)=1,\qquad \dist(u,w)=2 .
\]
Then there are vertices $w_\uparrow,w_\downarrow$, determined before any
further vertex is considered, such that
\begin{enumerate}
\item[\textup{(a)}] each of $w_\uparrow,w_\downarrow$ is equal or adjacent to
      each of $x,u,w$;
\item[\textup{(b)}] \emph{for these two fixed vertices}, every vertex $z$
      equal or adjacent to both $u$ and $w$ is equal or adjacent to both
      $w_\uparrow$ and $w_\downarrow$;
\item[\textup{(c)}] $g(w_\uparrow)+g(w_\downarrow)\le g(u)+g(w)$;
\item[\textup{(d)}] if \textup{(c)} is an equality then
      $A(w_\uparrow)+A(w_\downarrow)<A(u)+A(w)$.
\end{enumerate}
Clause \textup{(a)} asserts adjacency pair by pair and does \emph{not} assert
that the five vertices carry simultaneously disjoint representatives; they
cannot, since $\dist(u,w)=2$ forces the two minimisers to meet.
\end{thm}

\begin{proof}[Proof when the minimisers are transverse]
By \textup{(E)} fix, one class at a time, relative area minimisers
$S\in x$, $S_+\in u$, $S_-\in w$; by Proposition~\ref{prop:disjoint} the single
surface $S$ is disjoint from both others, and $S_+\cap S_-\ne\emptyset$ since
$\dist(u,w)=2$.  Assume in this proof that $S_+$ and $S_-$ are transverse
with $\bd S_+\cap\bd S_-=\emptyset$.

Write $\Gamma:=S_+\cap S_-$, which by neatness and
$\bd S_+\cap\bd S_-=\emptyset$ is a finite family of disjoint smooth closed
curves in the interior, of total length $\ell_0>0$.  In the normal slice at
$y\in\Gamma$ the two tangent planes cut the plane into four sectors of angles
$\alpha,\pi-\alpha,\alpha,\pi-\alpha$ with $\alpha(y)\in(0,\pi)$, and
transversality together with the compactness of $\Gamma$ makes
\begin{equation}\label{eq:theta0}
   \theta_0:=\min_\Gamma\min\{\alpha,\pi-\alpha\}>0 .
\end{equation}
Choose a constant width $\eta\le\min_\Gamma\eta_{\max}$ once and for all: the
minimum is positive because $\Gamma$ is compact and $\theta\ge\theta_0>0$ on
it, and a constant width satisfies the constraint on $\eta'$ in
\eqref{eq:width} vacuously.  With this width the gain \eqref{eq:gain} along a
closed curve $\gamma\subset\Gamma$ is at least
$g_0(\theta_0)\,\eta\,\mathrm{length}(\gamma)>0$, uniformly in $\gamma$.  Now
$S_-$ is itself a relative area minimiser, so Lemma~\ref{lem:no-disc} applies
with $e=0$: were there a disc patch it would supply closed curves $a,b$ with
$\cG(a)+\cG(b)\le0$, whereas both gains have just been bounded below by a
positive number.  So cutting $S_+$ and $S_-$ along $\Gamma$ produces no disc
patch, Lemma~\ref{lem:frontier} applies, and it yields
$P_\uparrow,P_\downarrow$.

By \eqref{eq:frontier-euler} and $\chi\le0$ for the discarded closed
components, $\chi(P_\uparrow)+\chi(P_\downarrow)\ge\chi(S_+)+\chi(S_-)$;
substituting \eqref{eq:euler-genus} gives
\begin{equation}\label{eq:genus-sum}
   g(P_\uparrow)+g(P_\downarrow)\;\le\;g(u)+g(w).
\end{equation}
Compress each $P$ along essential circles, always keeping the component with
the longitudinal boundary, until it is incompressible; by
Lemma~\ref{lem:euler} the genus does not increase and the process terminates.
Call the resulting vertices $w_\uparrow,w_\downarrow$; with
\eqref{eq:genus-sum} this is (c).  Each $P$ is disjoint from each of
$S,S_+,S_-$, so Corollary~\ref{cor:robust}, applied separately to the three,
gives (a).  For (b): if $z\notin\{u,w\}$ is equal or adjacent to both, take a
minimiser $Z\in z$; by Theorem~\ref{thm:U}(i) it is disjoint from $S_+$ and
from $S_-$, so by clause (2) of Lemma~\ref{lem:frontier} the coarse outputs
have representatives disjoint from $Z$, and Corollary~\ref{cor:robust} applies
again.  As proved there, both the positive rounding width and the push-off
scale may be decreased after $Z$ is supplied.  Their smooth isotopy classes
stay fixed, and the chosen compression sequence is transported along those
isotopies.  Thus the output \emph{vertices} do not move.  The area comparison
below uses the original quantitative width, not these alternate representatives.
If $z\in\{u,w\}$, (a) already gives the conclusion.

Suppose now (c) is an equality.  Then \eqref{eq:genus-sum} and every
compression inequality is an equality, so no compression occurred and each $P$
is already incompressible and represents the corresponding $w$.  The two
frontiers \eqref{eq:frontiers} pick out, in the normal slice at $y\in\Gamma$, a
pair of opposite sectors, so both have the same opening angle $\alpha(y)$ and
$\theta_0$ of \eqref{eq:theta0} bounds the angle for both.  The two cornered
frontiers partition $S_+\cup S_-$ up to the measure-zero set $\Gamma$, so
\begin{equation}\label{eq:frontier-area}
   \Area(\widetilde B_\uparrow)+\Area(\widetilde B_\downarrow)
     =\Area(S_+)+\Area(S_-).
\end{equation}
Round both frontiers with the width $\eta$ already fixed;
by \eqref{eq:gain} the total gain is at least
$2\varepsilon_0$ with $\varepsilon_0:=g_0(\theta_0)\,\eta\,\ell_0>0$, which
depends on $S_\pm$ and on this one choice and on nothing later.  Apply
Lemma~\ref{lem:pushoff} to the two rounded surfaces, obtaining $\sigma_0$ and
$C$ fixed before $\sigma$, and choose $\sigma\in(0,\sigma_0]$ with
$C\sigma<\varepsilon_0$.  Combining, and discarding the closed components,
which have positive area,
\[
\begin{aligned}
   \Area(P_\uparrow)+\Area(P_\downarrow)
   &\;\le\;\Area(S_+)+\Area(S_-)-2\varepsilon_0+C\sigma\\
   &\;<\;A(u)+A(w)-\varepsilon_0 .
\end{aligned}
\]
These are smooth neatly embedded competitors with boundary a leaf of $J$, so
$A(w_\uparrow)+A(w_\downarrow)$ is at most the left-hand side, which is (d).
\end{proof}

\subsection{Non-transverse minimisers}\label{ssec:tangential}

The proof just given assumed two things, and both may fail.  We record first
that the failures are of exactly two kinds.

Both $\bd S_+$ and $\bd S_-$ are leaves of the foliation $J$, and two leaves of
a foliation are equal or disjoint.  So either $\bd S_+\cap\bd S_-=\emptyset$ or
$\bd S_+=\bd S_-$, and in each case the two surfaces may or may not be
transverse in the interior.  Three configurations therefore remain to be
treated beyond the one just settled:
\begin{enumerate}
\item[\textup{(T0)}] $\bd S_+\cap\bd S_-=\emptyset$ and $S_+\pitchfork S_-$;
      this is the case proved above;
\item[\textup{(T1)}] $\bd S_+\cap\bd S_-=\emptyset$, with $S_+$ and $S_-$
      tangent at one or more interior points;
\item[\textup{(T2)}] $\bd S_+=\bd S_-$, so that the two surfaces meet along a
      whole curve of common boundary and the argument above cannot even start.
\end{enumerate}
The surfaces of (T2) need not be tangent along the shared leaf, and in general
are not: by \eqref{eq:N4} in clause (2) of Proposition~\ref{prop:collar} they
are transverse at all but finitely many of its points.  What fails in (T2) is
not transversality but the
disjointness of the boundaries, on which the cutting and pasting of
\S\ref{ssec:exchange-thm} depends throughout.  Configuration (T1) is not a
curiosity of the write-up either: two distinct minimal
surfaces with disjoint boundaries may perfectly well touch at an interior
point, and there it is transversality that fails.  Freedman, Hass and Scott
record on \cite[p.~636]{FHS} that
non-transverse behaviour occurs and give no structure theorem for it;
Proposition~\ref{prop:collar} is that structure theorem, its first clause
covering the interior tangencies of both (T1) and (T2), its second the collar
of (T2).  Both configurations are then disposed of by perturbing $S_-$ to a
transverse competitor and running the transverse argument on the perturbation;
what has to be controlled in each case is that the area lost to the
perturbation stays below the area gained by rounding, and it is
Lemma~\ref{lem:no-disc}, with its error term $e$, that makes the two
comparable.  We take them in the order \textup{(T2)}, \S\ref{ssec:T2}, then
\textup{(T1)}, \S\ref{ssec:T1}, and not in the order of their labels: the
argument for \textup{(T1)} is the argument for \textup{(T2)} with the collar
deleted, so writing \textup{(T2)} first is what makes \textup{(T1)} short.  The
labels themselves group by boundary behaviour, \textup{(T0)} and \textup{(T1)}
being the cases with $\bd S_+\cap\bd S_-=\emptyset$.

Throughout \S\S\ref{ssec:tangential}--\ref{ssec:T1} write $A:=S_-$ and
$B:=S_+$, keeping $S$ fixed and disjoint from both.

\emph{The interior picture.}  Let $p$ be a point of $\Int A\cap\Int B$ at which
$T_pA=T_pB$.  In normal coordinates $(x_1,x_2,y)$ at $p$ with the common
tangent plane $\{y=0\}$, both surfaces are graphs over a disc $D$ about the
origin, say $B=\{y=\beta\}$ and $A=\{y=\hat f\}$ with
$\beta(0)=\hat f(0)=0$ and $D\beta(0)=D\hat f(0)=0$; both are smooth by
Theorem~\ref{thm:R1}.  Both graphs are minimal, so subtracting the two minimal
graph equations and applying the mean value theorem along the segment joining
the two gradients gives, classically on $D$,
\begin{equation}\label{eq:Luint}
   L_p f_p:=\bd_i\bigl(a^{ij}\bd_jf_p\bigr)+b^i\bd_if_p+cf_p=0,
   \qquad f_p:=\hat f-\beta ,
\end{equation}
with $a^{ij}$ smooth and uniformly elliptic and $b^i,c$ smooth.  The tangency
points of $A$ and $B$ in $D$ are exactly the points of $\{f_p=0,\ df_p=0\}$.

\emph{The collar picture, in configuration \textup{(T2)} only.}  Suppose the
two boundary leaves coincide along a longitude $L$.  Take a vector
field on $\bd E(K)$ transverse to $J$ whose flow carries leaves to leaves,
extend it into the collar, and use its direction as a coordinate $y$, the
inward distance $r\ge0$, and arclength $s$ along $L$.  The resulting chart
$(s,r,y)$ is bi-Lipschitz onto its image with a constant
$C_{\mathrm{met}}'<\infty$ depending only on the ambient metric near $L$ and on
the chosen vector field.  By
Lemma~\ref{lem:R2} both surfaces are neat, so the extension may be taken
transverse along $L$ to both tangent planes, and both are graphs over one
half-cylinder:
\begin{equation}\label{eq:graphs}
   \bd E(K)=\{r=0\},\quad B=\{y=0\},\quad A=\{y=f(s,r)\},\quad f(s,0)=0 .
\end{equation}
By Theorem~\ref{thm:R1} the function $f$ is smooth up to $\{r=0\}$.  Both
graphs are minimal on $\{r>0\}$, so subtracting the two minimal graph
equations and applying the mean value theorem along the segment joining the
two gradients gives, classically on the closed half-disc,
\begin{equation}\label{eq:Lu}
   Lf:=\bd_i\bigl(a^{ij}\bd_jf\bigr)+b^i\bd_if+cf=0,
\end{equation}
with $a^{ij}$ smooth and uniformly elliptic and $b^i,c$ smooth, all up to
$\{r=0\}$.

\begin{prop}[Structure of the tangency set]\label{prop:collar}
Let $A$ and $B$ be as above: distinct relative area minimisers in two vertices
at distance $2$, meeting each other.  Then:
\begin{enumerate}
\item[\textup{(1)}] \emph{(Interior; no hypothesis on the boundaries.)}  Write
$\mathcal{T}\subset\Int A\cap\Int B$ for the set of interior tangency points.
Every $p\in\mathcal{T}$ is isolated among the interior tangency points
and has an integer
order $m_p\ge1$ and constants $c_1,c_2>0$ with
$|df_p|\ge c_1\rho^{m_p}-c_2\rho^{m_p+1}>0$ at
distance $\rho$ from $p$ on a punctured neighbourhood, $f_p$ being the local
difference of \eqref{eq:Luint}; so $p$ is the only critical point of $f_p$ in a
disc $D_p$.  Moreover
\begin{align}
   &\text{for each }p\in\mathcal{T},\ f_p\text{ takes both positive and negative
     values in every}\notag\\
   &\text{neighbourhood of }p.\label{eq:N5}
\end{align}
\item[\textup{(2)}] \emph{(Collar; in configuration \textup{(T2)}.)}  If
$\bd A=\bd B=L$, then with $f$ as in \eqref{eq:graphs} there is $r_0>0$ with
$\mathcal{T}\subset\{r\ge r_0\}$ and
\begin{align}
   &\text{in the open collar }\{0<r<r_0\}\text{ one has }df\neq0
     \text{ everywhere,}\notag\\
   &\text{while the critical points of }f\text{ on }\{r=0\}\text{ are finite
     in}\notag\\
   &\text{number and all lie on }\{f=0\}.\label{eq:N4}
\end{align}
\item[\textup{(3)}] In both configurations $\mathcal{T}$ is finite; put
$N:=\#\mathcal{T}$.
\end{enumerate}
\end{prop}

We refer to \eqref{eq:N4} and \eqref{eq:N5} as the collar conditions, though
strictly \eqref{eq:N4} concerns the collar and \eqref{eq:N5} the interior
tangency set.  In configuration \textup{(T2)} they are used together.  In
configuration \textup{(T1)} only \eqref{eq:N5} is available, and only it is
needed, because there is no collar to control: the whole intersection is then
at a positive distance from $\bd E(K)$.

\begin{proof}
Both clauses rest on the same two devices, applied on different domains: the
zeroth-order term is removed by dividing by a positive solution, and the
resulting divergence-form equation without zeroth-order term is factorised in
the sense of Vekua.  We set the first up once, for a domain $\Delta$ which is
either a disc $D_R(p)$ about an interior tangency point, carrying the operator
$L_p$ of \eqref{eq:Luint}, or a half-disc $D_R^+$ about a point of $L$,
carrying the operator $L$ of \eqref{eq:Lu}.  In both cases we write $L$ for the
operator and $f$ for its solution.

\emph{Removing the zeroth-order term.}  Solve $Lh=0$ on $\Delta$ with $h=1$ on
$\bd\Delta$.  We do
\emph{not} assume a sign condition on the zeroth-order coefficient, so the
existence theorems of \cite[2nd ed., Ch.~8]{GT}, which all carry hypothesis
\textup{(8.8)}, are not available as stated; smallness of $R$ replaces the
sign condition.  Indeed for $u\in W_0^{1,2}(\Delta)$ the Poincar\'e
inequality gives $\|u\|_{L^2}\le CR\|Du\|_{L^2}$, whence the associated
bilinear form satisfies
\[
   B[u,u]\;\ge\;\bigl(\lambda-C\|b\|_\infty R
        -C^2\|c\|_\infty R^2\bigr)\|Du\|_{L^2}^2 ,
\]
which is positive definite for $R$ small; $B$ is bounded on
$W_0^{1,2}(\Delta)$ because the coefficients are, so the Lax--Milgram theorem
\cite[2nd ed., Theorem~5.8, p.~83]{GT} yields a unique weak solution.
Writing $h=1+\varpi$, the function $\varpi$ solves $L\varpi=-c$ with zero
boundary data.  Rescale $x\mapsto x/R$ to the unit disc or half-disc, under which the
drift becomes $O(R)$ and the zeroth-order coefficient $O(R^2)$; testing the
equation against $\varpi$ and using the same Poincar\'e inequality gives
$\|\varpi\|_{L^2}=O(R^2)$.  The global maximum principle
\cite[2nd ed., Theorem~8.15, p.~189]{GT}, applied to $\varpi$ and to
$-\varpi$, both of which lie in $W_0^{1,2}$ and so are $\le0$ on the boundary
in the generalised sense, then gives
$\|\varpi\|_\infty\le C\bigl(\|\varpi\|_{L^2}+\nu^{-1}\|c\|_{L^{q/2}}\bigr)
=O(R^2)$, so $h\ge\tfrac12$ for $R$ small.  We use Theorem~8.15 rather than
the sharper \cite[2nd ed., Theorem~8.16, p.~191]{GT} precisely because the
latter assumes hypothesis \textup{(8.8)}: in Theorem~8.15 the zeroth-order
coefficient enters only through $\lambda^{-1}|d|$ in \textup{(8.6)}, so its
sign is unconstrained, and the $\|\varpi\|_{L^2}$ term that this costs is
supplied by the energy estimate just made.  Put $v:=f/h$;
expanding $L(hv)=0$ with $Lh=0$ gives, weakly,
\begin{equation}\label{eq:Lv}
   \bd_i\bigl(h\,a^{ij}\bd_jv\bigr)
     +\bigl(a^{ij}\bd_jh+h\,b^i\bigr)\bd_iv=0,
\end{equation}
with Lipschitz leading coefficients, bounded drift and \emph{no zeroth-order
term}; in the half-disc case one has in addition $v=0$ on $L$, since $f$ does.
Multiplying by the reciprocal square root of the
determinant of the leading matrix normalises that determinant to $1$ and
changes only the drift.  The zeroth-order term has to be removed afresh on
each domain on which the factorisation or the maximum principle is used: with
a zeroth-order term of uncontrolled sign both fail.

\emph{Proof of \textup{(1)}.}  Let $p\in\mathcal{T}$, run the previous
paragraph on $\Delta=D_R(p)$ with the operator $L_p$ of \eqref{eq:Luint}, and
write $v_p:=f_p/h_p$ for the resulting solution of \eqref{eq:Lv}; note
$v_p(p)=0$.  Choose a smooth Jordan subdomain $\Omega\Subset D_R(p)$
containing $p$.  As in the paragraph ``An isothermal chart up to $L$'' below ---
the argument is the same and shorter, no boundary arc being distinguished ---
Lichtenstein's theorem \cite[Theorem~4.1, case $m=1$]{HvdM} supplies a
$C^{1,\alpha}$ diffeomorphism of $\overline{\mathbb{D}}$ onto
$\overline\Omega$, conformal for the metric \eqref{eq:gmetric}, and by the
density law recorded there the equation for $\hat v_p$ becomes \eqref{eq:flat}
on the disc, with no boundary condition and no reflection needed.  The factorisation
of the paragraph ``Finite order'' below then applies verbatim: either
$\mathfrak{w}:=2\bd_z\hat v_p$ has isolated zeros of finite integer order $m$
with $c_1'|z-z_0|^m\le|\mathfrak{w}|\le c_2'|z-z_0|^m$ nearby, or
$\mathfrak{w}\equiv0$.  In the second case $\nabla\hat v_p\equiv0$, so
$\hat v_p$ is constant, and the constant is $0$ because $\hat v_p$ vanishes at
the image of $p$; then $f_p\equiv0$ near $p$, the set on which the two sheets
agree is open by this argument applied at each of its points and closed by
continuity, so the two surfaces coincide, contradicting $\dist(u,w)=2$.
Transporting through the bi-Lipschitz chart gives $|d\hat v_p|\asymp\rho^m$ and
$|\hat v_p|=O(\rho^{m+1})$ at distance $\rho$ from $p$, and $f_p=h_pv_p$ with
$h_p\ge\tfrac12$ gives $|df_p|\ge c_1\rho^{m}-c_2\rho^{m+1}>0$ on a punctured
neighbourhood of $p$.  In particular $p$ is the only critical point of $f_p$ in
a disc $D_p$, and, the tangency points in $D_p$ being critical points of $f_p$,
$p$ is isolated in $\mathcal{T}$.

For \eqref{eq:N5}, suppose $f_p\ge0$ near $p$, the other sign being identical.
Then $v_p\ge0$ on the disc on which it is defined and $v_p(p)=0$, so by the
strong maximum principle at this regularity
\cite[2nd ed., Theorem~8.19, p.~198]{GT}, whose hypotheses are met since the
operator is in divergence form with no zeroth-order term and
$v_p\in W^{1,2}$, we get $v_p\equiv0$, hence $f_p\equiv0$, and the same
contradiction as above.

\emph{Proof of \textup{(2)}.}  Assume $\bd A=\bd B=L$ and run the removal of
the zeroth-order term on half-discs $\Delta=D_R^+$ against $L$, writing $v:=f/h$
as there.

\emph{An isothermal chart up to $L$.}  Choose a smooth Jordan subdomain
$\Omega$ of the half-disc whose boundary contains a subarc of $L$ around the
point in question.  Write $a=(a^{ij})$ for the leading matrix of \eqref{eq:Lv}
after the normalisation $\det a\equiv1$, $b$ for its drift, and
\begin{equation}\label{eq:gmetric}
   g:=a^{-1},
\end{equation}
a Riemannian metric with $\det g=1$; it is the \emph{inverse} of the leading
matrix that carries the relevant conformal structure, as the computation below
shows.  Being the inverse of a Lipschitz uniformly elliptic matrix, $g$ is
Lipschitz on $\overline\Omega$, hence $C^{0,\alpha}$, and $\bd\Omega$ is
$C^{1,\alpha}$; so by the global form of Lichtenstein's theorem
\cite[Theorem~4.1, case $m=1$]{HvdM} there is a diffeomorphism
$\tau\colon\overline{\mathbb{D}}\to\overline\Omega$ of class $C^{1,\alpha}$,
conformal from the flat metric to $g$.
Let $\varphi$ be a M\"obius map of $\mathbb{D}$ onto the upper half-plane
carrying the arc $\tau^{-1}(L\cap\bd\Omega)$ into $\R$ and $\tau^{-1}$ of the
base point to $0$, and put $F:=\varphi\circ\tau^{-1}$.  Then $F$ is
$C^{1,\alpha}$ up to $L$, sends $L$ into $\R$ and the interior into the upper
half-plane, and is bi-Lipschitz, $DF$ being continuous and invertible on a
compact set.

It remains to see that in the coordinate $z=F(\cdot)$ the operator becomes the
flat Laplacian with a bounded drift.  A divergence-form operator transforms
under a $C^{1,\alpha}$ diffeomorphism by the density law: testing
$\int(a\,\nabla v)\cdot\nabla\phi\,dx$ against $\phi$ and substituting
$x=F^{-1}(z)$ turns the leading matrix into
$DF\,a\,DF^{\mathsf T}/|\det DF|$ and the drift into $DF\,b/|\det DF|$, no
derivative of $DF$ appearing anywhere.  Three consequences.  First, the
transformed drift is again in $L^\infty$, since $DF$ and $(\det DF)^{-1}$ are
bounded.  Second, the transformed leading matrix again has determinant $1$, the
factors $(\det DF)^2$ cancelling.  Third, that matrix is a positive multiple of
the identity exactly when $a=\lambda\,(DF^{\mathsf T}DF)^{-1}$ for a positive
function $\lambda$, that is, exactly when $g=a^{-1}$ is a conformal multiple of
the pullback $F^*\delta$ of the flat metric --- which is what the choice of
$\tau$ arranged, $\varphi$ being conformal.  A positive multiple of the
identity of determinant $1$ is the identity.  Hence \eqref{eq:Lv} becomes
\begin{equation}\label{eq:flat}
   \Delta\hat v+\hat b\cdot\nabla\hat v=0,\qquad \hat b\in L^\infty,\qquad
   \hat v=0\ \text{ on }\ \R .
\end{equation}

\emph{Odd reflection.}  Extend $\hat v$ to the full disc by
$\hat v(z):=-\hat v(\bar z)$ for $\operatorname{Im}z<0$, and extend
$\hat b$ by $\hat b_1(\bar z)=\hat b_1(z)$, $\hat b_2(\bar z)=-\hat b_2(z)$.
For a test function $\phi$ supported in the full disc, split
$\int\nabla\hat v\cdot\nabla\phi-\int(\hat b\cdot\nabla\hat v)\phi$ over the
two half-discs and substitute $z\mapsto\bar z$ in the lower one: the two
interior terms cancel by \eqref{eq:flat}, and the two conormal boundary terms
along $\R$ cancel because the leading matrix is the identity, so its
off-diagonal entry vanishes on $\R$ and the conormal derivative is
$\bd_2\hat v$, which is even under the reflection while the outward normals are
opposite.  Hence the extension solves \eqref{eq:flat} weakly on the full disc.
This is the one point at which the vanishing of the off-diagonal coefficient on
$L$ is used, and it is why the isothermal chart is taken before the reflection
and not after.

\emph{Finite order.}  Put $\mathfrak{w}:=2\bd_z\hat v$.  Since
$\bd_{\bar z}\bd_z=\tfrac14\Delta$ and
$\hat b\cdot\nabla\hat v=\tfrac12[(\hat b_1+i\hat b_2)\mathfrak{w}+(\hat b_1-i\hat b_2)\overline{\mathfrak{w}}]$
for real $\hat v$, equation \eqref{eq:flat} reads
\[
   \bd_{\bar z}\mathfrak{w}=-\tfrac14(\hat b_1+i\hat b_2)\,\mathfrak{w}
                 -\tfrac14(\hat b_1-i\hat b_2)\,\overline{\mathfrak{w}} ,
\]
a generalized analytic function with coefficients in $L^\infty\subset L^p$ for
every $p>2$; this is exactly Vekua's standing hypothesis
\cite[Ch.~III, \S1, condition~(1.10), p.~137]{Vekua}, the equation already
being in the canonical form \cite[Ch.~III, (1.5), p.~131]{Vekua}, whose
leading coefficient is identically $1$.  By
the similarity principle \cite[Ch.~III, \S4.1, Basic Lemma, p.~144, and formula
(4.3), p.~146]{Vekua}, $\mathfrak{w}=\Phi e^{\omega}$ with $\Phi$ holomorphic and $\omega$
bounded and H\"older; so unless $\mathfrak{w}\equiv0$ its zeros are isolated of finite
integer order $m$, with $c_1'|z-z_0|^m\le|\mathfrak{w}|\le c_2'|z-z_0|^m$ nearby
\cite[Ch.~III, \S4.2, p.~145, Theorem~3.5, pp.~146--147, and the
inequalities (4.15), p.~152]{Vekua}.  The
alternative $\mathfrak{w}\equiv0$ forces $\nabla\hat v\equiv0$, hence $\hat v\equiv0$
since $\hat v=0$ on $\R$, hence $f\equiv0$ on an open set; the set where the
two minimal sheets agree is then open by this argument applied at each of its
points and closed by continuity, so the sheets coincide, contradicting
$\dist(u,w)=2$.  Transporting through the bi-Lipschitz $F$ gives
$|d\hat v|\asymp\rho^{m}$ and $|\hat v|=O(\rho^{m+1})$ at distance $\rho$, and
$f=hv$ with $h\ge\tfrac12$ gives
$|df|\ge c_1\rho^{m}-c_2\rho^{m+1}>0$ on a punctured half-neighbourhood.

Every critical point of $f$ on $L$ is therefore isolated in the closed
half-neighbourhood, and at a non-critical boundary point continuity gives a
neighbourhood free of critical points; $L$ is compact, so finitely many
critical points remain and a uniform $r_0$ exists with no critical point in
$\{0<r<r_0\}$.  They lie on $\{f=0\}$ because $f$ vanishes on $L$.  Since a
tangency point in the collar would be a critical point of $f$, this also gives
$\mathcal{T}\subset\{r\ge r_0\}$, and clause (2) is proved.

\emph{Proof of \textup{(3)}.}  In configuration \textup{(T1)} one has
$A\cap B\cap\bd E(K)=\bd A\cap\bd B=\emptyset$, so $A\cap B$ is a compact
subset of $\Int E(K)$; in configuration \textup{(T2)}, clause (2) confines
$\mathcal{T}$ to the compact set $A\cap B\cap\{r\ge r_0\}$, which again misses
$\bd E(K)$ --- a point of $A\cap B$ on $\bd E(K)$ lies on $L$, where $r=0$.  In
both cases $\mathcal{T}$ is a subset of a compact subset $C$ of
$\Int A\cap\Int B$.  Tangency is a closed condition, the tangent planes varying
continuously, so $\mathcal{T}$ is closed in $\Int A\cap\Int B$ and hence closed
in $C$; being also discrete by clause (1), it is finite.
\end{proof}

\subsection{Configuration \textup{(T2)}: a common boundary leaf}\label{ssec:T2}

Reduce the collar width supplied by Proposition~\ref{prop:collar}, if
necessary, so that the closure of this smaller collar misses the finite
interior tangency set.  Write $r_0$ for the reduced width.  The function
$f$ is still defined on a larger collar and \eqref{eq:N4} still holds;
in particular the interior perturbation discs may be chosen disjoint
from the closed collar used below.

\begin{lem}[No small closed intersection curves]\label{lem:no-small}
Let $f_t$ be the perturbation of $f$ described in \eqref{eq:perturb} below.
There are $\rho_0>0$ and $\delta_1>0$ such that for every $t\in(0,\delta_1]$
every closed component $\gamma$ of $A_t\cap B$ has
$\operatorname{diam}(\gamma)\ge\rho_0$.
\end{lem}

We first describe the perturbation, since the lemma depends on its exact form.
Its sign is a free parameter, and fixing it correctly is what keeps the
intersection curves away from $L$; so we begin with that choice.

\begin{lem}[A sign, and an arc of $L$ swept clear]\label{lem:bdry-sign}
Put $a(s):=\bd_rf(s,0)$ and $\Lambda:=\sup|\bd_r^2f|<\infty$.  Then
$a\not\equiv0$, and there are $\varsigma\in\{\pm1\}$, a closed arc $I\subset L$,
and constants $a_0>0$, $r_1>0$, all depending only on $A$ and $B$, such that
\[
   \varsigma f(s,r)\ \ge\ \tfrac12a_0\,r\ >\ 0
   \qquad\text{for }s\in I,\ 0<r\le r_1 .
\]
\end{lem}

\begin{proof}
Since $f(s,0)\equiv0$ we have $\bd_sf(s,0)\equiv0$, so
$df(s,0)=(0,a(s))$ and the critical points of $f$ on $\{r=0\}$ are exactly the
zeros of $a$.  By \eqref{eq:N4} there are finitely many, so $a\not\equiv0$.
Pick $s_0$ with $a(s_0)\ne0$, set $\varsigma:=\operatorname{sign}a(s_0)$ and
$a_0:=\tfrac12|a(s_0)|$, and use the continuity of $a$ to choose a closed arc
$I\ni s_0$ with $\varsigma a\ge a_0$ on $I$.  Put
$r_1:=\min\{a_0/(2\max\{1,\Lambda\}),\,r_0/4\}$.  For $s\in I$ and $0\le\rho\le r_1$,
\[
   \varsigma\,\bd_rf(s,\rho)\ \ge\ \varsigma a(s)-\Lambda\rho\ \ge\ a_0-\Lambda r_1
   \ \ge\ \tfrac12a_0 ,
\]
and integrating in $r$ gives the claim.  Nothing in the construction refers to
the perturbation parameter.
\end{proof}

Take a smooth cut-off $0\le\chi(r)\le1$ equal to $1$ near $r=0$, supported in the collar
$\{r<r_0\}$, whose transition band is a compact subannulus of $0<r<r_0$, and
pairwise disjoint discs $D_p\Supset D_p'$ outside that closed collar,
around the points of $\mathcal{T}$,
with bumps $\psi_p$ equal to $1$ on $D_p'$; with $\varsigma$ as in
Lemma~\ref{lem:bdry-sign} set
\begin{equation}\label{eq:perturb}
   A_t:=\Bigl\{\,y=f+\varsigma t\chi(r)-\varsigma t
      \sum_{p\in\mathcal{T}}\psi_p\,\Bigr\}.
\end{equation}
Near the boundary $A_t\cap B$ is the level set $\{f=-\varsigma t\}$, regular for
small $t\ne0$ by \eqref{eq:N4}; on the compact transition band, $|df|$ has a positive minimum by
\eqref{eq:N4}, so $d(f+\varsigma t\chi)$ remains nonzero for small $t$.
Outside that band in the collar, $\chi$ is constant and its differential
is zero.  Consequently
\[
  df_t\ne0\quad\hbox{throughout }0<r<r_0,
  \qquad f_t:=f+\varsigma t\chi(r),
\]
for all sufficiently small $t>0$.  This assertion uses no uniform lower
bound for $|df|$ as $r\downarrow0$.  On $D_p'$ one has
$A_t\cap B=\{f=\varsigma t\}$, and by the last clause of
Proposition~\ref{prop:collar} the only critical point of $f$ in $D_p$ is $p$,
with critical value $f(p)=0$, so every $t\ne0$ is a regular value of either
sign; on the
annulus $D_p\setminus D_p'$ one has $|df|\ge c>0$, hence
$|d(f-\varsigma t\psi_p)|\ge c/2$ for $t$ small.  Elsewhere $A_t=A$ and
transversality
is open.  Thus there is $\delta^{(1)}>0$ such that for $t\in(0,\delta^{(1)}]$
the surface $A_t$ is neat, incompressible, transverse to $B$, with $\bd A_t$
the leaf $\{y=\varsigma t\}$ disjoint from $\bd B$, disjoint from $S$, and
$A_t\to A$ in $C^\infty$.  \emph{The interior perturbation is of
subtract-a-constant type, and this cannot be weakened}: a family guaranteeing
only transversality would lose the fact that $f_t$ and $f$ differ by a constant
near $p$, and the proof below would fail.

\begin{proof}[Proof of Lemma~\ref{lem:no-small}]
Otherwise there are $t_n\to0$ and closed components $\gamma_n$ with
$\operatorname{diam}(\gamma_n)\to0$; by $C^0$ convergence, points
$x_n\in\gamma_n$ subconverge to some $x\in A\cap B$.  If $x$ is a transverse
point outside the collar, then in a fixed ball at $x$ the intersection is a
single arc crossing it, and $C^1$ stability keeps that true for small $t$, so
$\gamma_n$ cannot lie in the ball.  If $x=p\in\mathcal{T}$, take $D_p$ as
above; for large $n$ the curve $\gamma_n$ bounds a disc $D_n\subset D_p$, and
$f_t$ vanishes on $\gamma_n$.  If $f_t\equiv0$ on $D_n$ then $df_t\equiv0$
there, impossible since $f_t$ has at most the one critical point $p$ in $D_p$;
otherwise $f_t$ attains a strict interior extremum on $D_n$, at a critical
point, hence at $p$ --- but on $D_p'$ the functions $f_t$ and $f$ differ by the
constant $-\varsigma t$, and by \eqref{eq:N5} $p$ is not a local extremum of $f$, hence
not of $f_t$.  If $0\le r(x)<r_0$, use a sufficiently small half-disc chart
at $x$.  For large $n$, the Jordan disc $D_n$ bounded by $\gamma_n$ lies
in that chart and in $r>0$: the linear coordinate $r$ cannot have a smaller
minimum inside a planar disc than on its boundary.  The preceding
collar-wide statement gives $df_t\ne0$ there.  Since $f_t=0$ on
$\bd D_n$, an interior maximum or minimum would contradict that statement;
if there is neither, $f_t\equiv0$ and its differential vanishes, also a
contradiction.  Each case is a contradiction.
\end{proof}

\emph{Choice of $\rho_2$.}  The next two results share a scale, and we fix it
here rather than inside either of them.  Choose $\rho_2>0$ with the balls
$B_{2\rho_2}(p)$, $p\in\mathcal{T}$, pairwise disjoint,
\begin{equation}\label{eq:rho2}
   4N\rho_2<\rho_0/8,\qquad \rho_2\le r_0/8,\qquad
   \rho_2\le r_1/(2C_{\mathrm{met}}') ,
\end{equation}
and put $P:=\bigcup_pB_{2\rho_2}(p)$, $G:=E(K)\setminus P$ and
$G':=E(K)\setminus\bigcup_pB_{\rho_2}(p)$, so that $G\subset G'$.  Only
$\rho_0$, $r_0$, $r_1$ and $\mathcal{T}$ enter this choice.
Lemma~\ref{lem:collar-empty} below quotes clause (1) of
Corollary~\ref{cor:uniform}, which is stated after it; there is no circle, as
the proof of Corollary~\ref{cor:uniform} uses only
Lemma~\ref{lem:no-small}, Proposition~\ref{prop:collar} and
Lemma~\ref{lem:round}.

\begin{lem}[Every closed component has uniform length away from $L$]
\label{lem:collar-empty}
Let $\varsigma,I,a_0,r_1$ be as in Lemma~\ref{lem:bdry-sign} and write
$\mathcal{C}:=L\times(0,r_0)$ for the open collar.  Then
\begin{enumerate}
\item[\textup{(i)}] for every sufficiently small $t>0$ the set $\Gamma_t:=A_t\cap B$ is disjoint
      from $I\times(0,r_1]$;
\item[\textup{(ii)}] no closed component of $\Gamma_t$ contained in
      $\mathcal{C}$ is null-homotopic in $\mathcal{C}$;
\item[\textup{(iii)}] there are $c_\star>0$ and $\delta_\star>0$, depending only on $A$, $B$ and the
      perturbation scheme, such that for every
      $t\in(0,\delta_\star]$ every closed component $\gamma$
      of $\Gamma_t$ satisfies
      \[
         \mathrm{length}\bigl(\gamma\cap G\cap\{r\ge r_1\}\bigr)\ \ge\ c_\star .
      \]
\end{enumerate}
\end{lem}

\begin{proof}
(i) On the relevant arc of the collar,
$\varsigma f_t=\varsigma f+t\chi\ge\tfrac12a_0r>0$ for $r>0$,
by Lemma~\ref{lem:bdry-sign} and $\chi\ge0$.
Thus $f_t$ cannot vanish there, whether or not $\chi=1$.

(ii) Suppose a closed component $\gamma\subset\mathcal C$ bounds a disc
$D\subset\mathcal C$.  The actual function $f_t=f+\varsigma t\chi$
is zero on $\bd D$ and has no critical point anywhere in the collar, as
proved following \eqref{eq:perturb}.  Any nonzero value on $D$ would give
an interior extremum; identically zero would also give a critical point.
Both alternatives are impossible.

(iii) Take $\delta_\star\le\min\{\delta_1,\delta^{(1)},\delta_2\}$,
where $\delta_2$ is supplied independently by Corollary~\ref{cor:uniform}
below.  Note first that $\gamma\cap\bd E(K)=\emptyset$, since
$A_t\cap\bd E(K)=\bd A_t$ is the leaf $\{y=\varsigma t\}$ while
$B\cap\bd E(K)=\bd B=L=\{y=0\}$.  Next, $\mathcal{T}$ lies in
$\{r\ge r_0\}$ by Proposition~\ref{prop:collar}(2), and $\rho_2\le r_0/8$ by
\eqref{eq:rho2}, so
$P=\bigcup_pB_{2\rho_2}(p)\subset\{r>3r_0/4\}$ and hence
\begin{equation}\label{eq:Zregion}
   \{r_1\le r\le r_0/2\}\ \subset\ G\cap\{r\ge r_1\}=:Z .
\end{equation}
Let $\gamma$ be a closed component of $\Gamma_t$.  Three cases exhaust the
possibilities.

\emph{Case A: $\min_\gamma r\ge r_1$.}  Then $\gamma\subset\{r\ge r_1\}$, so
$\gamma\cap G\subset Z$ and clause (1) of Corollary~\ref{cor:uniform} gives
$\mathrm{length}(\gamma\cap Z)\ge3\rho_0/8$.

\emph{Case B: $\min_\gamma r<r_1$ and $\max_\gamma r>r_0/2$.}  The function
$r$ is continuous on the circle $\gamma$ and takes a value below $r_1$ and a
value above $r_0/2$, so $\gamma$ contains two disjoint subarcs each joining
$\{r=r_1\}$ to $\{r=r_0/2\}$ inside $\{r_1\le r\le r_0/2\}$.  In the collar
chart $r$ is $1$-Lipschitz, so each has length at least
$(r_0/2-r_1)\ge r_0/4$ up to the bi-Lipschitz constant
$C_{\mathrm{met}}'$; by \eqref{eq:Zregion} both lie in $Z$, and together they
give $\mathrm{length}(\gamma\cap Z)\ge r_0/(2C_{\mathrm{met}}')$.

\emph{Case C: $\min_\gamma r<r_1$ and $\max_\gamma r\le r_0/2$.}  Then
$\gamma\subset\{0<r\le r_0/2\}\subset\mathcal{C}$, so by (ii) $\gamma$ is not
null-homotopic in the annulus $\mathcal{C}$ and its projection to $L$ is onto.
In particular $\gamma$ has points over every $s\in I$, and by (i) each of them
has $r>r_1$; by hypothesis each also has $r\le r_0/2$, so the part of $\gamma$
lying over $I$ is contained in $Z$ by \eqref{eq:Zregion}.  The projection is
$1$-Lipschitz in the collar chart, so that part has length at least the
arclength $\ell(I)$ of $I$, up to $C_{\mathrm{met}}'$.

Take $c_\star:=\min\{3\rho_0/8,\ r_0/2,\ \ell(I)\}/C_{\mathrm{met}}'$.
\end{proof}

\begin{cor}[Bounds independent of $t$]\label{cor:uniform}
With $\rho_2$, $P$, $G$, $G'$ as in the choice of $\rho_2$ above, there
are $\theta_1\in(0,\pi/2]$, $\eta_1^{\mathrm{geo}}>0$, $\kappa_1>0$ and
$\delta_2>0$ such that for every $t\in(0,\delta_2]$ and every closed component
$\gamma$ of $A_t\cap B$:
\begin{enumerate}
\item $\mathrm{length}(\gamma\cap G)\ge3\rho_0/8$;
\item the opening angle $\alpha$ along
      $\gamma\cap G'\cap\{r\ge r_1/2\}$ satisfies the
      \emph{two-sided} bound $\min\{\alpha,\pi-\alpha\}\ge\theta_1$;
\item along $\gamma\cap G'\cap\{r\ge r_1/2\}$ the functions of
      Lemma~\ref{lem:round}, read pointwise, satisfy
      $\eta_{\max}\ge\eta_1^{\mathrm{geo}}$ and $\kappa\ge\kappa_1$.
\end{enumerate}
\end{cor}

\begin{proof}
(1) By Lemma~\ref{lem:no-small} there are $a,b\in\gamma$ with
$d(a,b)\ge\rho_0/2$.  Collapse each of the $N$ components of $P$ to a point and
let $d'$ be the resulting quotient pseudometric.  Collapsing $N$ sets of
diameter at most $4\rho_2$ decreases distances by at most $4N\rho_2$, so
$d(a,b)\le d'(a,b)+4N\rho_2$; and the image of an arc of $\gamma$ from $a$ to
$b$ is a path of length at most $\mathrm{length}(\gamma\cap G)=:\Lambda$.
Hence $\rho_0/2\le\Lambda+4N\rho_2<\Lambda+\rho_0/8$, so
$\Lambda>3\rho_0/8$.  The quotient formulation is needed, not a decoration: a
curve may enter and leave the same ball an unbounded number of times, so the
count ``$N$ balls, each costing $4\rho_2$'' is not valid, whereas each
component is collapsed only once.

(2) Put $G'':=E(K)\setminus\bigcup_pB_{\rho_2/2}(p)$, so $G'\subset G''$, and
put
\[
   \cK\;:=\;A\cap B\cap G''\cap\{r\ge r_1/4\} ,
\]
with the convention that $r$ is extended by $r\equiv r_0$ outside the collar.
The set $\cK$ is compact, and $A$ meets $B$ transversally at every point of it.
Indeed, at a point of $\cK$ inside the collar the two sheets are the graphs
$\{y=f\}$ and $\{y=0\}$ of \S\ref{ssec:tangential}, and transversality along
$\{f=0\}$ is exactly the condition $df\neq0$, which holds by \eqref{eq:N4}; the
tangency points of $L$ itself lie on $\{r=0\}$ and are excluded by the
constraint $r\ge r_1/4$.  At a point of $\cK$ outside the collar
Proposition~\ref{prop:collar} gives transversality away from $\mathcal{T}$, and
$\mathcal{T}$ has been excised by $G''$.  Hence
$\min\{\alpha,\pi-\alpha\}$, a continuous positive function on $\cK$, has a
positive infimum $2\theta_1$; the upper bound is included because the two
frontiers occupy opposite sectors and Lemma~\ref{lem:round} needs $\alpha$
bounded away from $\pi$ as well as from $0$.  For small $t$ the curve
$A_t\cap B$ lies in the $\min\{\rho_2/4,r_1/4\}$-neighbourhood of $A\cap B$ and
on $G''$ the $C^1$ distance from $A_t$ to $A$ is $O(t)$, so the two-sided bound
persists with $\theta_1$ along $\gamma\cap G'\cap\{r\ge r_1/2\}$.

(3) We verify all the local data in Lemma~\ref{lem:round} on a
\emph{buffered} region; the same verification will also be used for (T1).
The compact set $\cK$ above is uniformly transverse and has positive
distance from both the ambient boundary and the tangency set.  Cover it by
finitely many coordinate balls whose doubled balls still avoid those sets.
In each doubled ball $A$ and $B$ are single smooth graph sheets, their
normal Jacobian has determinant bounded away from zero, and there are no
other portions of either surface.  The last assertion follows from
embeddedness and compactness: a smaller ball about each point of
$A\cap B$ excludes the complements of the selected graph patches.
The functions defining $A_t$ converge to those of $A$ in every smooth
norm, since the perturbation uses fixed smooth cut-offs.  The uniform
implicit function theorem therefore gives, on smaller concentric balls,
a single intersection arc of $A_t$ and $B$, with uniformly bounded
derivatives through any fixed finite order.  These balls cover every
point of
\[
 V_t:=\Gamma_t\cap G'\cap\{r\ge r_1/2\}
\]
for small $t$: otherwise a sequence of missed points would converge to
a point of $\cK$ and hence eventually lie in one of the balls.
The buffer between $G',\{r\ge r_1/2\}$ and the deleted sets
$\bigcup_pB_{\rho_2/2}(p),\{r<r_1/4\}$ is fixed and positive.

This gives a uniform \emph{globally admissible} tube radius near $V_t$,
not only a curvature bound for an individual component.  To see the
injectivity point, choose a radius smaller than the preceding coordinate
ball radii and their buffers, and cap the radii on the rest of the
intersection by this same small number.  If two normal fibres over points of
$\Gamma_t$ met in that protected tube, both base points would lie in
one of the doubled balls: any base point outside its buffer is farther
away than the sum of the two chosen radii.  Within that ball the entire
intersection, across all components, is the one graph arc just described.
The normal map of this arc is injective at a uniformly positive radius,
by its uniform derivative bounds and the inverse function theorem.
The same balls exclude unrelated portions of $A_t\cup B$.  Thus neither
a remote branch on the same component nor a different component can
enter the protected tube.  The extension to the whole intersection can be
made explicitly.  Choose a slightly larger protected region $V_t^+$
inside the same fixed buffer, and prove the preceding local estimates
there as well.  Let $R>0$ be a common small radius on this larger region,
and let $0<r_t<R/2$ be a radius on which the normal map of the entire
compact $\Gamma_t$ is injective.  For this fixed $t$, decrease it also
to exclude every portion of $A_t\cup B$ outside the chosen sheet
collars; compact embeddedness permits this additional decrease.
Choose a smooth cut-off equal to one near
$V_t$ and supported in $V_t^+$.  Interpolate the radius between $r_t$
and $R$ with that cut-off, keeping every radius at most $R$.  If neither
of two fibres has been enlarged, injectivity follows from the choice of
$r_t$.  If one has been enlarged, its base lies in $V_t^+$, and any
potentially colliding base lies within the fixed buffer, where the
single-arc argument applies.  Thus the variable-radius tube is globally
injective.  The same argument excludes unrelated surface pieces.
Subsequent straightening radii are decreased, if necessary, so that
their images remain inside this normal tube; on the protected region
a uniform decrease suffices.  Hence the resulting globally valid adapted
tube has a positive radius bound near $V_t$ independent of $t$.

Apply the fibre-preserving straightening in the proof of
Lemma~\ref{lem:round} in this tube.  The defining functions have uniform
$C^4$ bounds on the doubled balls, their normal Jacobians have a uniform
inverse bound, and the normal frames have uniformly bounded derivatives
there.  Differentiating the inverse-function identity twice gives uniform
$C^2$ bounds for the straightening and its inverse on a smaller fixed
tube.  It therefore gives a constant $K_1$ bounding both the normal
derivative of the pulled-back metric and $|m'|$ \emph{on this region}.
The same argument applies to each of the four sectors; take the maximum
of these $K_1$'s and the minimum radii.  Choose continuous data $K_t$
and $\varrho_t$ in the global tube with
\[
 K_t\le K_1,\qquad \varrho_t\ge\varrho_1>0\quad\hbox{on }V_t.
\]
Such choices use the slightly larger good region for interpolation; the
data outside it can depend on $t$.  In particular no supremum of $|m'|$
over the degenerating parts of $\Gamma_t$ occurs.  Finally
$d_t\ge d_1:=\min\{1,r_1/(4C_{\mathrm{met}}')\}>0$ on $V_t$.
It follows from \eqref{eq:etalower} that we may take
\[
 \eta_1^{\mathrm{geo}}=E(\theta_1,\varrho_1,d_1,K_1)>0,
 \qquad \kappa_1=H(\theta_1)>0.
\]
The same bounds hold for all the sector pairings used in disc exchange.
For use in comparing different positive parameters, make the adapted
coordinate construction coherently in $t>0$.  Orient the intersection
circles along their smooth intersection track; use signed-distance defining
functions for $A_t,B$ and the wedge bisector frame.  Normal exponential
maps and the normalized inverse-function construction then depend smoothly
on $t$ wherever they are defined.  Choosing smaller positive continuous
radii on the track gives one such family of charts; the buffered bounds
just proved can be retained on $V_t$.  The local metric majorants can
likewise be chosen continuous on the track.  On any compact interval of
positive parameters the whole intersection track is compact, so the
admissible widths in this family have a positive common lower bound.
\end{proof}

\begin{lem}[Smooth widths with a fixed plateau]\label{lem:plateau-width}
Let a smooth compact intersection set $\Gamma\subset M$ carry the common positive
bounds $\eta_{\max},\kappa$ of Lemma~\ref{lem:round}.  Suppose a fixed
smooth function $\zeta:M\to[0,1]$ has $|d\zeta|\le L$ and that
$\eta_{\max}\ge a>0$, $\kappa\ge b>0$ on
$\Gamma\cap\operatorname{supp}\zeta$.  Put
\[
 \eta_1=\min\{a/2,b/(2\max\{1,L\})\},\qquad
 \eta_2=\min\{\eta_1/2,\tfrac12\min_\Gamma\eta_{\max}\}.
\]
Then $\eta=\eta_2+(\eta_1-\eta_2)\zeta|_\Gamma$ is a positive smooth
admissible width, equal to $\eta_1$ wherever $\zeta=1$.
\end{lem}

\begin{proof}
Where $\zeta>0$, $\eta\le\eta_1\le a/2\le\eta_{\max}$ and
$|\eta'|\le\eta_1L\le b/2\le\kappa$.  Where $\zeta=0$,
non-negativity and smoothness imply $d\zeta=0$, so $\eta'=0$ and
$\eta=\eta_2\le\eta_{\max}/2$.  The plateau assertion is exact and
requires no smoothing of a minimum or of a distance function.
\end{proof}

\begin{proof}[Proof of Theorem~\ref{thm:exchange}, configuration \textup{(T2)}]
Fix a smooth ambient cutoff $\zeta$ with
\[
 \zeta:E(K)\to[0,1],\qquad
 \zeta=1\ \hbox{on }G\cap\{r\ge r_1\},\qquad
 \operatorname{supp}\zeta\subset G'\cap\{r>r_1/2\},
 \qquad |d\zeta|\le L_\zeta<\infty.
\]
For instance multiply cut-offs which are zero on
$\overline{B_{5\rho_2/4}(p)}$ and one outside $B_{2\rho_2}(p)$ by a
collar cut-off zero for $r\le5r_1/8$ and one for $r\ge r_1$.
Smooth cut-offs between these nested closed and open sets exist in the
fixed ambient manifold; the finitely many factors have a finite derivative
bound.  All these choices depend only on the previously fixed geometry,
not on $t$.  Apply Lemma~\ref{lem:plateau-width} with
$a=\eta_1^{\mathrm{geo}}$, $b=\kappa_1$, and put
$\eta_1:=\min\{\eta_1^{\mathrm{geo}}/2,
\kappa_1/(2\max\{1,L_\zeta\})\}$ and
\begin{equation}\label{eq:epsstar}
   \varepsilon_*:=g_0(\theta_1)\cdot\eta_1\cdot c_\star>0 ,
\end{equation}
with $c_\star$ the constant of Lemma~\ref{lem:collar-empty},
which depends on $A,B$ and on the once-chosen perturbation scheme and on
nothing later.

\emph{No disc patch, uniformly in $t$.}  Suppose cutting $A_t$ and $B$ along
$A_t\cap B$ produces a disc patch.  Every closed component $\gamma$ of
$A_t\cap B$ satisfies, by Corollary~\ref{cor:uniform},
Lemma~\ref{lem:collar-empty}(iii) and the slicewise non-negativity in
Lemma~\ref{lem:round},
\[
   \cG(\gamma)\ \ge\ g_0(\theta_1)\,\eta_1\,
      \mathrm{length}\bigl(\gamma\cap G\cap\{r\ge r_1\}\bigr)
   \ \ge\ \varepsilon_* ,
\]
the gain being computed with the width function \eqref{eq:ramp} constructed
below, which equals $\eta_1$ on $G\cap\{r\ge r_1\}$, and the rest of the curve
being discarded.  So \eqref{eq:twogains} gives
$2\varepsilon_*\le\cG(a)+\cG(b)\le e(t):=\Area(A_t)-\Area(A)$, and $e(t)\to0$
by smooth convergence.  So there is
$\delta_3\le\min\{\delta_2,\delta_\star,\delta^{(1)}\}$ with $e(t)<\varepsilon_*$,
hence no disc
patch, \emph{for every} $t\in(0,\delta_3]$.  The quantifier order matters:
$\varepsilon_*$ is pinned down before $t$, since the patch exists only after
$t$ and both the length of its frontier and the angle along it vary with $t$.

\emph{The output vertices do not depend on $t$.}  For $t$ in a compact
subinterval of $(0,\delta^{(1)}]$ the surfaces $A_t$ are transverse to $B$ with
disjoint boundaries and disjoint from $S$, so there is an ambient isotopy
$H_t$ with $H_{t_0}=\mathrm{id}$, $H_t(A_{t_0})=A_t$, $H_t(B)=B$ and
$H_t|_S=\mathrm{id}$, preserving $\bd E(K)$: take $X_t=\bd_t j_t$ for
embeddings $j_t$ with image $A_t$, correct it along the compact intersection
track by a tangential term so that it becomes tangent to $B$ there --- possible
because transversality makes the relevant map onto $T_pE(K)/T_pB$ surjective
and the metric supplies a smooth right inverse --- and extend by Hadamard's
lemma and a partition of unity to a time-dependent field tangent to $B$ and to
$\bd E(K)$ and vanishing on $S$, whose flow is the required isotopy.  Lift
$H_t$ to the infinite cyclic cover starting at the identity.  It fixes $S_0$
and $S_1$ and the fundamental domain between them, commutes with the deck
generator, and cannot interchange the two sides of a wall, since the upper side
is the one containing $S_1$; so it carries the \emph{cornered} frontiers
\eqref{eq:frontiers} at $t_0$ to those at $t$.
To compare the selected smooth frontiers, use the coherent adapted charts
constructed after Corollary~\ref{cor:uniform}(3).  On the compact parameter
interval choose a positive constant width $\bar\eta$ below every pointwise
upper bound.  Its derivative along every intersection circle is zero.
The fixed smooth profile with this width and the unmodified moving
patches form a smooth compact track in $E(K)\times[t_0,t_1]$:
they agree as subsets on an open annulus at each profile edge, and the
projection to $t$ is a submersion, also on the lateral boundary track.  Choose a
smooth horizontal vector field on this track whose $t$-component is one
and which is tangent to the track of the surface boundaries, using local product charts and a
partition of unity.  Integrating it supplies a smooth family of neat
embeddings; no agreement between the parametrizations supplied by $H_t$
and by the adapted charts is required.  At either
endpoint, clause (6) of Lemma~\ref{lem:round} joins its quantitative width
to $\bar\eta$ in that same endpoint chart.  Consequently the rounded
frontiers have the same smooth ambient isotopy classes at both endpoints,
by smooth isotopy extension.  The push-off comparison in
Lemma~\ref{lem:frontier} preserves those classes.  No assertion that
$H_t$ carries a separately optimized rounding exactly to another is needed,
and the possibly non-smooth function $\eta_2(t)$ is not used as an isotopy
parameter.  Hence the two coarse classes are independent of $t>0$.  Fix once
and for all, at one value of $t$, a compression sequence from each coarse
surface to an incompressible descendant with boundary, giving
$w_\uparrow,w_\downarrow$; transporting the coarse surfaces and the discs along
the isotopy gives the same two vertices at every $t>0$.  \emph{This is what
makes (b) hold in the order stated}: given $z$ with minimiser $Z$ disjoint from
$A\cup B$, one may choose $t_z>0$ after $Z$ is known so that
$Z\cap A_{t_z}=\emptyset$.  Clause (2) of Lemma~\ref{lem:frontier}
then decreases both the rounding width and the push-off scale at $t_z$
to avoid $Z$.  Transport the fixed compression sequences along the smooth
isotopies just described and apply Corollary~\ref{cor:robust}.  The classes,
and therefore the vertices, are the ones already fixed; $t_z$ and the two
smaller geometric parameters change only representatives.

\emph{The area estimate.}  Since $\bd A_t\cap\bd B=\emptyset$, the set
$\Gamma_t:=A_t\cap B$ is a finite union of closed curves, and it is non-empty,
for otherwise $A_t$ and $B$ would be disjoint representatives of $w$ and $u$.
Corollary~\ref{cor:uniform} and Lemma~\ref{lem:collar-empty} apply to every
component, and the
slicewise gain is non-negative everywhere, so rounding the two cornered
frontiers along $\Gamma_t$ gains at least $\varepsilon_*$ in total.  The order
of choices is
\begin{align*}
   A,B&\Rightarrow(\text{perturbation scheme, including }\varsigma)
   \Rightarrow\rho_0\Rightarrow r_1\Rightarrow\rho_2\\
   &\Rightarrow(\theta_1,\varrho_1,K_1,d_1)\Rightarrow
   (\eta_1^{\mathrm{geo}},\kappa_1,\zeta,L_\zeta,c_\star)
   \Rightarrow\eta_1\Rightarrow\varepsilon_*\\
   &\Rightarrow\delta\Rightarrow t\Rightarrow\eta_2
   \Rightarrow(\widehat B_\uparrow,\widehat B_\downarrow)
   \Rightarrow C\Rightarrow\sigma ,
\end{align*}
each step using only what precedes it; in particular $t$ appears in the chain,
after $\delta$ and before the rounded surfaces.  Choose
\begin{equation}\label{eq:delta}
\begin{gathered}
   \delta\le\min\{\delta^{(1)},\delta_1,\delta_2,\delta_\star,\delta_3\}
   \quad\text{and small enough that}\\
   \Area(A_t)<\Area(A)+\varepsilon_*/2\ \text{ for }t\in(0,\delta] ,
\end{gathered}
\end{equation}
and fix such a $t$.  Choose
\[
 \eta_2(t)=\min\{\eta_1/2,\tfrac12\min_{\Gamma_t}\eta_{\max,t}\}>0,
\]
and use the smooth width
\begin{equation}\label{eq:ramp}
 \eta_t(s)=\eta_2(t)+(\eta_1-\eta_2(t))\zeta(x(s)).
\end{equation}
Lemma~\ref{lem:plateau-width} verifies its pointwise admissibility on all
of $\Gamma_t$.  In particular it is exactly $\eta_1$ on the region where
the gain is collected, while its derivative is zero on the bad region,
however small the local allowance $\kappa_t$ there becomes.  Its smoothness
and its plateau are part of the formula, and no approximation step is
being left implicit.  Only the positive floor $\eta_2(t)$ depends on $t$;
the gain on $G\cap\{r\ge r_1\}$ is the fixed number $\varepsilon_*$.
This construction is available for every sufficiently small $t$ and hence
also justifies its earlier use in the no-disc-patch argument before the
final $t$ was chosen.

Now obtain $C$ from Lemma~\ref{lem:pushoff}, and choose
$C\sigma<\varepsilon_*/4$.  Running steps (a)--(d) of the transverse case with
$(A_t,B)$ in place of $(S_-,S_+)$ and $\varepsilon_0$ replaced by
$\varepsilon_*$ gives
\[
\begin{aligned}
   \Area(P_\uparrow)+\Area(P_\downarrow)
   &\le\Area(A_t)+\Area(B)-\varepsilon_*+C\sigma\\
   &<A(w)+A(u)-\varepsilon_*/4 ,
\end{aligned}
\]
whence (d).  The genus computation is topological and is transported by the
isotopy, so (c) is unchanged, and (a) follows as before.
\end{proof}

\subsection{Configuration \textup{(T1)}: disjoint boundaries, interior
tangencies}\label{ssec:T1}

This is the easier of the two non-transverse configurations, and it is treated
by the same device: perturb $A$ to a transverse competitor and bound the area
lost against the area gained.  What makes it easier is that the intersection
stays away from $\bd E(K)$ altogether, so that the collar apparatus of
\S\ref{ssec:T2} --- the sign $\varsigma$ and the arc $I$, the length estimate
inside the collar, the third scale of the width ramp --- is not needed and is
replaced by a single positive number.

Assume $\bd A\cap\bd B=\emptyset$ and $\mathcal{T}\ne\emptyset$.  Since
$A\cap\bd E(K)=\bd A$ and $B\cap\bd E(K)=\bd B$ we have
$A\cap B\cap\bd E(K)=\bd A\cap\bd B=\emptyset$, and $A\cap B$ is compact, so
\begin{equation}\label{eq:d0}
   d_0:=\dist\bigl(A\cap B,\ \bd E(K)\bigr)>0 .
\end{equation}

\emph{The perturbation.}  By Proposition~\ref{prop:collar}(1) each of the
finitely many $p\in\mathcal{T}$ carries a disc $D_p$, inside the graph chart of
\eqref{eq:Luint}, on which $p$ is the only critical point of $f_p$ and
$f_p(p)=0$.  Shrinking, take the $D_p$ pairwise disjoint and contained in
$\{d(\cdot,\bd E(K))>d_0/2\}$, choose concentric smaller discs
$D_p'\Subset D_p$, and take bumps $\psi_p$ supported in $D_p$ and equal to
$1$ on $D_p'$.  Let $A_t$ be obtained from $A$ by replacing, over each $D_p$,
the graph $\{y=f_p\}$ by
\begin{equation}\label{eq:perturbT1}
   \{y=f_p-t\psi_p\} .
\end{equation}
Then $\bd A_t=\bd A$, still a leaf of $J$ disjoint from $\bd B$; $A_t$ is
smooth, neatly embedded and normally isotopic to $A$, hence incompressible and
a representative of $w$; and $A_t=A$ outside a compact subset of $\Int E(K)$,
so $A_t\cap S=\emptyset$ for small $t$ and $A_t\to A$ in $C^\infty$.  As in
\eqref{eq:perturb}, the perturbation is of subtract-a-constant type on $D_p'$,
and this is what the second case of Lemma~\ref{lem:T1-bounds}(1) consumes.

Transversality holds for all small $t>0$: on $D_p'$ one has
$A_t\cap B=\{f_p=t\}$, and $t\ne0$ is a regular value because the only
critical point of $f_p$ in $D_p$ is $p$, with critical value $0$; on
$D_p\setminus D_p'$ one has $|df_p|\ge c>0$, hence
$|d(f_p-t\psi_p)|\ge c/2$ for $t$ small; elsewhere $A_t=A$, which meets $B$
transversally away from $\mathcal{T}$.  Fix $\delta^{(1)}>0$ accordingly.  For
$t\in(0,\delta^{(1)}]$ the set $\Gamma_t:=A_t\cap B$ is therefore a finite
union of disjoint smooth closed curves --- closed, because an arc would have
its endpoints in $\bd A_t\cap\bd B=\emptyset$ --- lying in
$\{d(\cdot,\bd E(K))\ge d_0/2\}$; and it is non-empty, for otherwise $A_t$ and
$B$ would be disjoint representatives of $w$ and $u$, contradicting
$\dist(u,w)=2$.  Finally
\begin{equation}\label{eq:etT1}
   e(t):=\Area(A_t)-\Area(A)\longrightarrow0\qquad(t\downarrow0),
\end{equation}
and $\Area(A)=A(w)$, $A$ being a minimiser.

\begin{lem}[Bounds independent of $t$, tangency case]\label{lem:T1-bounds}
There are $\rho_0>0$, $\rho_2>0$, $\theta_1\in(0,\pi/2]$,
$\eta_1^{\mathrm{geo}}>0$, $\kappa_1>0$ and $\delta_2>0$, depending only on
$A$, $B$ and the perturbation scheme, such that, writing
$P:=\bigcup_pB_{2\rho_2}(p)$, $G:=E(K)\setminus P$ and
$G':=E(K)\setminus\bigcup_pB_{\rho_2}(p)$, the following hold for every
$t\in(0,\delta_2]$ and every closed component $\gamma$ of $\Gamma_t$:
\begin{enumerate}
\item $\operatorname{diam}(\gamma)\ge\rho_0$ and
      $\mathrm{length}(\gamma\cap G)\ge3\rho_0/8$;
\item the opening angle $\alpha$ satisfies the two-sided bound
      $\min\{\alpha,\pi-\alpha\}\ge\theta_1$ along $\gamma\cap G'$;
\item along $\gamma\cap G'$ the functions of Lemma~\ref{lem:round}, read
      pointwise, satisfy $\eta_{\max}\ge\eta_1^{\mathrm{geo}}$ and
      $\kappa\ge\kappa_1$.
\end{enumerate}
\end{lem}

\begin{proof}
Each clause is the corresponding statement of \S\ref{ssec:T2} with the collar
deleted, and we say what the deletion is.

\emph{Diameter.}  Run the proof of Lemma~\ref{lem:no-small}.  Suppose
$t_n\to0$ and closed components $\gamma_n$ with
$\operatorname{diam}(\gamma_n)\to0$; points $x_n\in\gamma_n$ subconverge to
some $x\in A\cap B$.  Two cases now exhaust the possibilities, the third case
of that proof being vacuous because $\Gamma_t$ misses the collar by
\eqref{eq:d0}.  If $x$ is a transverse point, a fixed ball at $x$ meets $A\cap
B$ in a single arc crossing it, and $C^1$ stability keeps that true for small
$t$, so $\gamma_n$ cannot lie in the ball.  If $x=p\in\mathcal{T}$, then for
large $n$ the curve $\gamma_n$ bounds a disc $D_n\subset D_p$ on which
$f_p-t\psi_p$ vanishes on the boundary; it cannot vanish identically, since
then its differential would vanish on $D_n$ while $p$ is its only critical
point in $D_p$; so it attains a strict interior extremum at a critical point,
necessarily $p$ --- but on $D_p'$ it differs from $f_p$ by the constant $t$,
and by \eqref{eq:N5} $p$ is not a local extremum of $f_p$.  This gives
$\rho_0>0$, valid for all sufficiently small $t$.  Now choose $\rho_2>0$ with the balls
$B_{2\rho_2}(p)$ pairwise disjoint, $4N\rho_2<\rho_0/8$ and
$2\rho_2<d_0/2$; the collapsing argument of
Corollary~\ref{cor:uniform}(1), which uses nothing but $\rho_0$ and this choice
of $\rho_2$, gives $\mathrm{length}(\gamma\cap G)\ge3\rho_0/8$.

\emph{Angle.}  Put $G'':=E(K)\setminus\bigcup_pB_{\rho_2/2}(p)$ and
$\cK:=A\cap B\cap G''$, a compact set on which $A$ meets $B$ transversally,
$\mathcal{T}$ having been excised.  The restriction $\{r\ge r_1/4\}$ of
Corollary~\ref{cor:uniform}(2) is not needed and has no analogue here: there it
removed the tangencies along $L$, and here $\bd A\cap\bd B=\emptyset$.  So
$\min\{\alpha,\pi-\alpha\}$ has a positive infimum $2\theta_1$ on $\cK$, and
for $t$ small the two-sided bound persists with $\theta_1$ along
$\gamma\cap G'$.

\emph{Widths.}  Apply the buffered-tube construction in
Corollary~\ref{cor:uniform}(3) with protected set
$V_t=\Gamma_t\cap G'$ and buffer
$E(K)\setminus\bigcup_pB_{\rho_2/2}(p)$.  The collar restriction is
omitted, since \eqref{eq:d0} already separates the entire intersection
from the ambient boundary.  Compactness of $\cK$, the angle
bound, and the uniform smooth bounds for $A_t,B$ give the same single-arc
charts, exclusion of all other sheet portions, global variable-radius
normal tubes, and uniform straightening estimates on this buffer.
Thus the common local data satisfy
$\varrho_t\ge\varrho_1>0$, $K_t\le K_1$ on $V_t$; outside the buffer
they may depend on $t$.  Take
$d_1=\min\{1,d_0/4\}$ and
$\eta_1^{\mathrm{geo}}=E(\theta_1,\varrho_1,d_1,K_1)$,
$\kappa_1=H(\theta_1)$ in \eqref{eq:etalower}.  This proves (3)
without taking a supremum of the slope derivative near any tangency.
\end{proof}

\begin{proof}[Proof of Theorem~\ref{thm:exchange}, configuration \textup{(T1)}]
Choose a fixed smooth $\zeta:E(K)\to[0,1]$ which is one on $G$ and
whose support lies in $G'$, by multiplying cut-offs zero on
$\overline{B_{5\rho_2/4}(p)}$ and one outside $B_{2\rho_2}(p)$.
Let $L_\zeta=\sup|d\zeta|<\infty$, and put
$\eta_1:=\min\{\eta_1^{\mathrm{geo}}/2,
\kappa_1/(2\max\{1,L_\zeta\})\}$ and
\begin{equation}\label{eq:epsstarT1}
   \varepsilon_*:=g_0(\theta_1)\cdot\eta_1\cdot\tfrac38\rho_0>0 ,
\end{equation}
fixed before $t$, in the order
$A,B\Rightarrow$ perturbation scheme $\Rightarrow\rho_0\Rightarrow\rho_2
\Rightarrow\theta_1\Rightarrow(\eta_1^{\mathrm{geo}},\kappa_1,\zeta,L_\zeta)
\Rightarrow\eta_1\Rightarrow\varepsilon_*\Rightarrow\delta\Rightarrow t$.

\emph{The width function.}  For each sufficiently small $t>0$ put
$\eta_2(t)=\min\{\eta_1/2,\tfrac12\min_{\Gamma_t}\eta_{\max,t}\}$
and $\eta_t=\eta_2(t)+(\eta_1-\eta_2(t))\zeta|_{\Gamma_t}$.
Lemma~\ref{lem:plateau-width} gives a smooth admissible width, exactly
$\eta_1$ on $\Gamma_t\cap G$.  Hence the slicewise estimate in
Lemma~\ref{lem:round} gives, on every closed component $\gamma$,
\[
 \cG(\gamma)\ge g_0(\theta_1)\eta_1\,
                  \mathrm{length}(\gamma\cap G)\ge\varepsilon_*.
\]
The derivative of the width vanishes wherever the uniform good-region
bounds are unavailable.  The positive floor, but not $\varepsilon_*$,
may depend on $t$.

\emph{No disc patch, uniformly in $t$.}  If cutting $A_t$ and $B$ along
$\Gamma_t$ produced a disc patch, Lemma~\ref{lem:no-disc} --- applicable
because $A$ is a minimiser, $A_t$ is normally isotopic to it with $\bd A_t$ a
leaf, and $A_t\pitchfork B$ with $\bd A_t\cap\bd B=\emptyset$ --- would give
closed curves $a,b$ of $\Gamma_t$ with
$2\varepsilon_*\le\cG(a)+\cG(b)\le e(t)$.  By \eqref{eq:etT1} there is
$\delta_3\le\min\{\delta_2,\delta^{(1)}\}$ with $e(t)<\varepsilon_*$, hence no disc patch, for every
$t\in(0,\delta_3]$.

\emph{Conclusion.}  Choose
$\delta\le\min\{\delta^{(1)},\delta_2,\delta_3\}$, small enough in
addition that $\Area(A_t)<\Area(A)+\varepsilon_*/2$ for $t\in(0,\delta]$, and
fix such a $t$.  Lemma~\ref{lem:frontier} applies to the pair $(A_t,B)$ and
yields $P_\uparrow,P_\downarrow$; rounding the two cornered frontiers along
$\Gamma_t$ with the width $\eta$ gains at least $\varepsilon_*$ in total, and
Lemma~\ref{lem:pushoff} then costs $C\sigma$, which we take below
$\varepsilon_*/4$.  Exactly as in configuration \textup{(T2)}, an ambient
isotopy carrying $A_{t_0}$ to $A_t$, preserving $B$ and fixing $S$ pointwise,
and preserving
$\bd E(K)$, shows that the two coarse classes, and hence the two output
vertices $w_\uparrow,w_\downarrow$, do not depend on $t$; the construction of
that isotopy used only transversality, the disjointness of the boundaries and
the disjointness from $S$, all of which hold here.  Then
\[
\begin{aligned}
   \Area(P_\uparrow)+\Area(P_\downarrow)
   &\le\Area(A_t)+\Area(B)-\varepsilon_*+C\sigma\\
   &<A(w)+A(u)-\varepsilon_*/4 ,
\end{aligned}
\]
which is (d); the genus computation is topological and transported by the
isotopy, giving (c); and (a) follows as in the transverse case.  For (b) the
$t$-independence just established is what is needed, exactly as in
configuration \textup{(T2)}: given $z$ with minimiser $Z$ disjoint from
$A\cup B$, choose $t_z>0$ after $Z$ is known so that $Z\cap A_{t_z}=\emptyset$
and apply clause (2) of Lemma~\ref{lem:frontier}, decreasing both the
rounding width and the push-off scale at $t_z$ to avoid $Z$.  Transport
the fixed compression sequences and apply Corollary~\ref{cor:robust}.
The coarse classes and hence the vertices remain fixed; these choices
change only representatives and are separate from the area competitors.
\end{proof}

The three configurations \textup{(T0)}, \textup{(T1)}, \textup{(T2)} exhaust
the possibilities, so Theorem~\ref{thm:exchange} is proved.

\subsection{Axiom \textup{(N2)} for \texorpdfstring{$\IS(K)$}{IS(K)}}

\begin{prop}[Apices of a two-interval]\label{prop:apex}
Let $K$ be a non-trivial knot.  Let $u,w$ be vertices
of $\IS(K)$ with $\dist(u,w)=2$.  Then
\begin{enumerate}
\item $\cI(u,w)$ has an apex; indeed both outputs of
      Theorem~\ref{thm:exchange}, applied to $u,w$ and any common neighbour,
      are apices;
\item some apex $x$ of $\cI(u,w)$ satisfies $c(x)<\max\{c(u),c(w)\}$;
\item if $g(u),g(w)\le\ell$ then the apex in \textup{(2)} has $g(x)\le\ell$.
\end{enumerate}
Consequently $\bigl(\IS(K),c\bigr)$ is an exchange complex in the sense of
Definition~\ref{dfn:exchange}.
\end{prop}

\begin{proof}
(1) Let $x_0\in\cI(u,w)$, which is non-empty since $\dist(u,w)=2$, and let
$y$ be either output of Theorem~\ref{thm:exchange} applied to $(x_0;u,w)$.  By
(a), $y$ is equal or adjacent to $u$ and to $w$; it cannot equal $u$, since it
would then be adjacent to $w$, contradicting $\dist(u,w)=2$, and likewise it
cannot equal $w$; so $y\in\cI(u,w)$.  By (b), every $z$ equal or adjacent to
both $u$ and $w$ --- in particular every member of $\cI(u,w)$ --- is equal or
adjacent to $y$.  So $y$ is an apex.

\emph{Local descent.}  Suppose $x\in\cI(u,w)$ satisfies $c(x)\ge c(u)$ and
$c(x)\ge c(w)$, and let $y_\uparrow,y_\downarrow$ be the outputs of
Theorem~\ref{thm:exchange} applied to $(x;u,w)$.  Lexicographic comparison
gives $g(x)\ge g(u)$ and $g(x)\ge g(w)$, so by (c)
$g(y_\uparrow)+g(y_\downarrow)\le g(u)+g(w)\le2g(x)$ and the minimum of the
two genera is at most $g(x)$.  If it is strictly smaller, the corresponding
output $y$ has $c(y)<c(x)$ and $g(y)\le g(x)$.  If it equals $g(x)$, then both
genera equal $g(x)$, the sum is $2g(x)$, and (c) is an equality with
$g(u)=g(w)=g(x)$; then $A(x)\ge A(u)$ and $A(x)\ge A(w)$, so (d) gives
$A(y_\uparrow)+A(y_\downarrow)<A(u)+A(w)\le2A(x)$ and the output of smaller
area has $g(y)=g(x)$, $A(y)<A(x)$, hence again $c(y)<c(x)$ and $g(y)\le g(x)$.

(2) The set of complexities of apices of $\cI(u,w)$ is a non-empty subset of
$W$, so by Lemma~\ref{lem:wf}(iii) it has a least element; let $x_*$ be an apex
attaining it.  Suppose $c(x_*)\ge\max\{c(u),c(w)\}$.  An apex is a common
neighbour, so local descent applies to $(x_*;u,w)$ and produces an output $y$
with $c(y)<c(x_*)$.  By the argument in (1), $y$ is itself an apex of
$\cI(u,w)$, contradicting the minimality of $c(x_*)$.  Hence
$c(x_*)<\max\{c(u),c(w)\}$, which is (2).

(3) From $c(x_*)<\max\{c(u),c(w)\}$ and the lexicographic order,
$g(x_*)\le\max\{g(u),g(w)\}\le\ell$.

For the last sentence: $\IS(K)$ is non-empty by
Proposition~\ref{prop:nonempty}, connected by
Proposition~\ref{prop:connected}, and flag by
Proposition~\ref{prop:flag}; $W$ is well-ordered by
Lemma~\ref{lem:wf}(iii); (N1) is clause (1) and (N2) is clause (2).
\end{proof}

\begin{rem}[What is and is not chosen]\label{rem:no-selection}
Theorem~\ref{thm:exchange} produces its two outputs from minimisers chosen one
class at a time, and different common neighbours may produce different apices;
no map $(u,w)\mapsto x$ is defined anywhere above, and none is needed, because
Definition~\ref{dfn:exchange} asks only for existence.  This is exactly the
point at which the obstruction identified in \cite{PS12} is bypassed rather
than solved.  The passage in (2) to an apex of least complexity is a single
instantiation inside a proof by contradiction, legitimate in $\mathsf{ZF}$ by
Remark~\ref{rem:wf-zf}.
\end{rem}

\section{Exchange complexes}\label{sec:exchange-complexes}

From this section to the end of \S\ref{sec:contractibility} no three-dimensional
topology is used.  Nothing is quoted from \S\S\ref{sec:geometric-inputs}--%
\ref{sec:exchange-lemma} except Definition~\ref{dfn:exchange} itself.

\subsection{The axioms}

Let $X$ be a connected flag simplicial complex with vertex set $X^{(0)}$, and
write $\dist$ for the edge-path metric on $X^{(0)}$.  For $u,w$ with
$\dist(u,w)=2$ the \emph{two-interval} is
\[
   \cI(u,w)\;=\;N(u)\cap N(w),
\]
the set of common neighbours; it is non-empty by definition of the distance.  A
vertex $x\in\cI(u,w)$ is an \emph{apex} of the interval if $x$ is equal or
adjacent to every member of $\cI(u,w)$; equivalently, since $X$ is flag, the
full subcomplex on $\cI(u,w)$ is a cone with apex $x$.\footnote{The word
``apex'' is used in a different sense by Chepoi, Chalopin, Hirai and Osajda
\cite{CCHO} for a vertex of a graph adjacent to all others.  We keep it here
because it is the natural word for a cone point, but the two usages should not
be conflated.}

\begin{dfn}[Exchange structure]\label{dfn:exchange}
Let $W$ be a well-ordered set --- one in which every non-empty subset has a
least element.  An \emph{exchange structure} on $X$ is a map
$c\colon X^{(0)}\to W$ such that
\begin{itemize}
\item[\textup{(N1)}] every two-interval has an apex; and
\item[\textup{(N2)}] whenever $\dist(u,w)=2$, \emph{some} apex $x$ of
$\cI(u,w)$ satisfies $c(x)<\max\{c(u),c(w)\}$.
\end{itemize}
A non-empty connected flag complex carrying such a $c$ is an \emph{exchange
complex}.
\end{dfn}

\begin{rem}[\textup{(N1)} is implied by \textup{(N2)}]\label{rem:N1-implied}
Axiom (N2) asserts that \emph{some apex} of $\cI(u,w)$ has a certain property,
and in particular that an apex exists; so (N1) follows from it.  We keep both
because the two are used for different purposes --- (N1) alone in the
manufacture of the cone structure, (N2) in the descent --- and because in
\S\ref{sec:applications} they are verified by different arguments.
\end{rem}

Three further comments on the definition, each of which is load-bearing.

\begin{rem}[No apex is selected]\label{rem:no-choice}
The axioms assert that an apex \emph{exists}; they never assert that one can be
\emph{chosen}, and no map $(u,w)\mapsto x$ is part of the data.  This does not
solve the difficulty identified in \cite{PS12} --- that a projection cannot be
made well defined on isotopy classes --- it avoids it.  Everything below is
written so that no such choice is ever needed.
\end{rem}

\begin{rem}[Non-emptiness cannot be dropped]\label{rem:nonempty}
The induction of Theorem~\ref{thm:main-comb} is based at ``order type $1$'',
and the empty complex, of order type $0$, is not contractible.
\end{rem}

\begin{rem}[The choice-theoretic dependence, stated once]\label{rem:choice}
For a total order, ``every non-empty subset has a least element'' implies ``no
infinite strictly descending sequence'' in $\mathsf{ZF}$; the converse needs
dependent choice.  We use the well-ordering theorem, hence the axiom of choice,
in Lemma~\ref{lem:ties}, to well-order the fibres of the complexity.  That is
the only use in this paper, and we flag it here rather than repeatedly.  In
particular \S\ref{sec:contractibility} does not add one: it treats $|X|$ as a CW
complex with one open cell per simplex, and never needs the total ordering of
the vertices that a $\Delta$-complex structure would require
\cite[p.~107]{Hatcher}.  For the complexity actually used
in \S\ref{sec:applications} the descending-chain form is what
Lemma~\ref{lem:wf} proves directly, and the passage to least elements is
available there in $\mathsf{ZF}$; see the remark following that lemma.
\end{rem}

\subsection{Breaking ties}

\begin{lem}[Ties may be broken]\label{lem:ties}
If $X$ carries an exchange structure, then it carries one whose complexity is
injective and well-orders $X^{(0)}$.
\end{lem}

\begin{proof}
The set of attained values is well-ordered.  Well-order each fibre $c^{-1}(a)$
and order $X^{(0)}$ lexicographically, first by the old value and then within
the fibre.  A lexicographic sum of well-orders indexed by a well-order is a
well-order, and the new complexity is injective.  Axiom (N1) does not mention
$c$.  If the old (N2) had a witness $x$ with $c(x)<\max\{c(u),c(w)\}$, then the
value block of $x$ strictly precedes the block of the larger endpoint, so the
strict inequality survives the refinement.
\end{proof}

\emph{From here on the complexity is assumed injective.}

\subsection{The ordering theorem}\label{ssec:ordering}

An injective exchange structure well-orders the vertices, and axiom (N2) then
in particular implies that any two vertices at distance $2$ have a common neighbour
strictly below their maximum.  That hypothesis --- without the word apex,
without flagness, and without (N1) --- is already enough for the following.

\begin{cor}\label{cor:initial-isometric}
Let $(X,c)$ be an exchange complex with $c$ injective and let $D\subseteq X^{(0)}$
be non-empty and downward closed for $c$.  Then the full subcomplex $X_D$ on $D$ is
isometrically embedded in $X$, and is connected.  In particular, for any two
vertices $p,q$ there is a geodesic from $p$ to $q$ all of whose vertices $z$
satisfy $c(z)\le\max\{c(p),c(q)\}$.
\end{cor}

\begin{proof}
Let $u,w\in D$ and let $n:=\dist_X(u,w)$, finite because $X$ is connected.  If
$n\le1$ there is nothing to prove, so assume $n\ge2$.  Among all paths
$v_0=u,v_1,\dots,v_n=w$ of length $n$ in $X$ --- a non-empty set --- consider
\[
   M(v_\bullet):=\max\{c(v_i):0<i<n\}\ \in W .
\]
The order on $W$ is a well-order, so the set of attained values has a least
element; fix a path attaining it.

Suppose some interior vertex has $c(v_i)>\max\{c(u),c(w)\}$, and let $v_i$ be
the interior vertex attaining $M$, unique because $c$ is injective.  Then
$c(v_{i-1})<c(v_i)$ and
$c(v_{i+1})<c(v_i)$: each of $v_{i\pm1}$ is either an endpoint, and then its
complexity is below $c(v_i)$ by assumption, or an interior vertex, and then its
complexity is at most the maximum, with equality excluded because $c$ is
injective and $v_{i\pm1}\ne v_i$.  The vertices $v_{i-1}$ and $v_{i+1}$ are
neither equal nor adjacent, since a path of length $n$ between vertices at
distance $n$ admits no shortcut; hence $\dist(v_{i-1},v_{i+1})=2$, and (N2)
supplies $x\in\cI(v_{i-1},v_{i+1})$ with
\[
   c(x)<\max\{c(v_{i-1}),c(v_{i+1})\}<c(v_i).
\]
Replace $v_i$ by $x$.  The result is a walk of length $n$ from $u$ to $w$; it
is a path, for a repeated vertex would allow a loop to be excised and produce a
walk of length $<n$, contradicting $\dist_X(u,w)=n$.  Every interior vertex of
it other than $x$ had complexity strictly below $c(v_i)$, by injectivity, and
$c(x)<c(v_i)$; so $M$ has strictly dropped, contradicting minimality.

So every vertex of the chosen path satisfies $c(v_i)\le\max\{c(u),c(w)\}$, and
$D$ being downward closed with $u,w\in D$, the whole path lies in $X_D$.  Hence
$\dist_{X_D}(u,w)\le n=\dist_X(u,w)$, and the reverse inequality is automatic;
$X_D$ is isometrically embedded, and connected because these distances are
finite.  The last sentence of the statement is the displayed property of the
chosen path, with $D=X^{(0)}$.
\end{proof}

\begin{rem}[Where this sits in the literature]\label{rem:what-was-cut}
A considerably more general statement is \cite[Lemma 9.13, p.~150]{CCHO}: for a
well-order $\preceq$ on the vertex set of a graph such that any $u,w$ at
distance $2$ admit $v\in\cI(u,w)\setminus\{u,w\}$ with
$v\prec\max\{u,w\}$, every level set $L_v=\{u:u\preceq v\}$ induces an
isometric subgraph.  Their hypotheses are strictly weaker than ours --- no
apex, no flagness, no local finiteness --- and the argument above is theirs:
they too take a shortest path whose largest vertex is least in the well-order
and exchange at that vertex, which is what is done here with $c$ in place of
$\preceq$.  We
have nevertheless given it, so that \S\S\ref{sec:exchange-complexes}--%
\ref{sec:contractibility} are self-contained and no citation stands between
the axioms and Theorem~\ref{thm:main-comb}.
\end{rem}

\subsection{Heredity}

\begin{lem}[Heredity]\label{lem:heredity}
Let $(X,c)$ be an exchange complex with $c$ injective.  For a vertex $v$ let
$X_{<v}$ be the full subcomplex on $\{u:c(u)<c(v)\}$ and let $L(v)$ be the full
subcomplex on the \emph{descending link}
$L_c(v)=\{u\in N(v):c(u)<c(v)\}$.  Then
\begin{enumerate}
\item every non-empty full subcomplex $X_D$ on a downward closed set
      $D\subseteq X^{(0)}$ is an exchange complex; in particular this holds
      for $X_{<v}$ when it is non-empty;
\item if non-empty, $L(v)$ is an exchange complex, and it is non-empty unless
      $v$ is the least vertex.
\end{enumerate}
\end{lem}

\begin{proof}
A full subcomplex of a flag complex is flag.  For (1), connectedness of
$X_D$ is Corollary~\ref{cor:initial-isometric}.  If $u,w\in D$ have distance
$2$ in $X_D$, they are non-adjacent in $X$ and have a common neighbour there,
so also have distance $2$ in $X$.  By (N2) there is an apex $x$ of
$\cI_X(u,w)$ with $c(x)<\max\{c(u),c(w)\}$.  Downward closure gives
$x\in D$.  Every common neighbour in $X_D$ is a common neighbour in $X$,
so $x$ remains an apex in $X_D$ and satisfies (N2) there.

For (2), let $u,w\in L_c(v)$ be non-adjacent.  The vertex $v$ is a common
neighbour in $X$, so (N2) supplies an apex $x$ with
\begin{equation}\label{eq:heredity}
 c(x)<\max\{c(u),c(w)\}<c(v).
\end{equation}
Since $v\in\cI_X(u,w)$, the apex $x$ is equal or adjacent to $v$; the strict
inequality excludes equality.  Thus $x\in L_c(v)$, and it is an apex of
the smaller interval there.  This proves (N1) and (N2) in $L(v)$, and also
connectedness: any non-adjacent pair has this common neighbour.

Finally, if $v$ is not the least vertex $v_0$, apply
Corollary~\ref{cor:initial-isometric} to $v$ and $v_0$: the vertex adjacent to
$v$ on the resulting geodesic is a neighbour $z\ne v$ with $c(z)\le c(v)$, and
injectivity gives $c(z)<c(v)$, so $z\in L_c(v)$.
\end{proof}

\begin{rem}[The strict inequality is not decoration]\label{rem:strict}
The strict inequality in \eqref{eq:heredity} excludes $x=v$ in the
descending-link case.  The link therefore inherits the exchange axiom
with its vertex order strictly below $v$, as required by both the direct
induction and the Morse formulation in \S\ref{ssec:F-discussion}.
\end{rem}

\subsection{Formalisation status}

\begin{rem}\label{rem:lean}
No machine-checked proof of Theorem~\ref{thm:main-comb} is supplied, and
the theorem itself has not been formalised.  Earlier versions reported a
separate Lean~4 development for four fragments: the ordering statement,
the tie-breaking lemma, a reformulation of dismantlability, and the
seven-vertex example.  The present source distribution contains neither
those Lean sources nor a pinned toolchain, so these reports are not used
as proof certificates here.  In particular we make no reproducible claim
here about that development's line count or axiom dependencies.  The
mathematical arguments are given in the text, including the complete
finite verification in Proposition~\ref{prop:not-cone}.
\end{rem}

\section{Contractibility}\label{sec:contractibility}

\subsection{Two facts about CW complexes}

The geometric realisation $|X|$ of a simplicial complex carries the weak
topology: a set is closed exactly when its intersection with every closed
simplex is closed.  For a complex that is not locally finite this is
\emph{strictly finer} than the metric topology, and $\IS(K)$ is not locally
finite \cite[p.~231]{Kak92}.  All topological statements below refer to the
weak topology, under which $|X|$ is a CW complex with one open cell per
simplex.  We use two standard facts, both cited and neither reproved.

\begin{itemize}
\item[\textup{(A)}] A compact subspace of a CW complex is contained in a finite
      subcomplex \cite[Proposition A.1, p.~520]{Hatcher}.
\item[\textup{(B)}] If $(Y,A)$ is a CW pair and $A$ is contractible, then the
      quotient map $Y\to Y/A$ is a homotopy equivalence.  This is
      \cite[Proposition 0.17, p.~15]{Hatcher}, whose hypothesis is the homotopy
      extension property, combined with \cite[Proposition 0.16, p.~15]{Hatcher},
      which supplies that property for a CW pair.
\end{itemize}

Fact (B) is what replaces, in this version, the appeal to ``a cofibration which
is a homotopy equivalence is the inclusion of a strong deformation retract''.
That implication is true but was quoted in version~1 without proof and without
a checked reference; the argument below does not need it.

\subsection{The main combinatorial theorem}

\begin{thm}\label{thm:main-comb}
Every exchange complex is contractible.
\end{thm}

\begin{proof}
By Lemma~\ref{lem:ties} we may assume the complexity $c$ is injective, so that
it well-orders the vertex set.  For an exchange complex $(Z,c)$ with $c$
injective, the \emph{order type} of $(Z,c)$ is the order type of $(Z^{(0)},c)$;
if it is $\alpha$ we write $Z^{(0)}=\{z_\beta\}_{\beta<\alpha}$ in increasing
order and let $Z_\gamma$ be the full subcomplex on $\{z_\beta:\beta<\gamma\}$,
so that $Z_0=\emptyset$ and $Z_\alpha=Z$.  By Lemma~\ref{lem:heredity}(1),
$Z_\gamma$ with $1\le\gamma\le\alpha$ is again an exchange complex with
injective complexity, of order type $\gamma$.

We prove by transfinite induction on $\alpha$ the statement
\begin{quote}
$(\ast_\alpha)$\quad \emph{Every} exchange complex with injective complexity
and order type at most $\alpha$ is contractible.
\end{quote}
The universal form is not a convenience but a necessity: the successor step
below applies the inductive hypothesis to a descending link, which is not one
of the subcomplexes $Z_\gamma$ of the complex being treated.
Lemma~\ref{lem:heredity}(2) is what makes that application legitimate --- the
descending link of the largest vertex has all its vertices below that vertex,
hence order type at most $\alpha$, which is strictly less than $\alpha+1$.

\smallskip
\emph{Base.}  Order type $1$: a single vertex.

\smallskip
\emph{Successor.}  Let $\alpha\ge1$, assume $(\ast_\alpha)$, and let $Z$ be an
exchange complex with injective complexity of order type $\alpha+1$.  Since
$\alpha+1$ is a successor, $Z$ has a $c$-largest vertex $v$; put $Y=Z$,
$A=Z_\alpha$ (the full subcomplex on the remaining vertices) and $L=L(v)$, the
descending link of $v$.  Since $Y$ is the full subcomplex on an initial segment
containing $v$, the neighbours of $v$ in $Y$ are exactly the vertices of $L$.  A
simplex of $Y$ either omits $v$, and then has all its vertices below $v$ and so
lies in $A$; or it is $\{v\}\cup\tau$ with $\tau$ a simplex of $Y$ contained in
$L_c(v)$, hence a simplex of $L$ because $L$ is a full subcomplex.  Conversely
$L\subseteq A$, again by fullness.  Hence
\[
   Y=A\cup(v*L),\qquad A\cap(v*L)=L ,
\]
the second equality because a simplex of $v*L$ lies in $A$ exactly when it omits
$v$.  The paragraph above gives the inclusion $Y\subseteq A\cup(v*L)$; the
reverse inclusion $v*L\subseteq Y$ is where flagness is used, and it should be
said, because it is easy to read the displayed equality as pure bookkeeping.  A
simplex of $v*L$ is $\{v\}\cup\tau$ with $\tau$ a simplex of $L$, and
$\tau\subseteq L_c(v)\subseteq N(v)$, so $\{v\}\cup\tau$ is a clique of the
$1$-skeleton; that it is a \emph{simplex} of $Y$ is exactly flagness.  This
costs nothing --- flagness is part of Definition~\ref{dfn:exchange}, so
$(\ast_\alpha)$ and Theorem~\ref{thm:main-comb} carry it --- but the hypothesis
is used here and not only in Lemma~\ref{lem:heredity}.
Both $A$ and $L$ are exchange complexes with injective complexity of order type
at most $\alpha$ (Lemma~\ref{lem:heredity}, using $\alpha\ge1$ for the
non-emptiness of $L$), so both are contractible by $(\ast_\alpha)$.

Now apply fact (B) twice.  First to the CW pair $(v*L,\,L)$: the cone $v*L$ is
contractible --- the straight-line homotopy towards the cone point is continuous
for the weak topology because $|v*L|\times I$ is again a CW complex, $I$ being
locally compact --- and $L$ is contractible, so $(v*L)/L$ is contractible.
Second to the CW pair $(Y,A)$: $A$ is contractible, so $Y\to Y/A$ is a homotopy
equivalence.  Collapsing $A$ collapses $L\subseteq A$, and no cell of $v*L$
outside $L$ lies in $A$, so
\[
   Y/A\;=\;(v*L)/L ,
\]
which we have just seen to be contractible.  Therefore $Y$ is contractible.

\smallskip
\emph{Limit.}  Let $\lambda$ be a limit ordinal, assume $(\ast_\alpha)$ for
every $\alpha<\lambda$, and let $Z$ be an exchange complex with injective
complexity of order type $\lambda$.  Then
$|Z|=\bigcup_{1\le\alpha<\lambda}|Z_\alpha|$, an increasing union of non-empty
subcomplexes, each contractible by $(\ast_\alpha)$, and $Z$ is connected by
Corollary~\ref{cor:initial-isometric}.  Let $n\ge0$
and let $f\colon S^n\to|Z|$ be continuous.  Its image is compact, hence
by fact (A) is contained in a finite subcomplex; that subcomplex has finitely
many vertices in total, and each has index $<\lambda$, so there is
$\alpha<\lambda$ with $f(S^n)\subseteq|Z_\alpha|$ --- here $\lambda$ being a
limit is used to find a single $\alpha$ strictly above the finitely many
indices.  Since $Z_\alpha$ is contractible, $f$ is null-homotopic in
$|Z_\alpha|$, hence in $|Z|$.  Thus $\pi_n(|Z|)=0$ for all
$n\ge0$, and $|Z|$ is a CW complex, so it is contractible by
Whitehead's theorem \cite[Theorem 4.5, p.~346, and the remark on
p.~348]{Hatcher}.

\smallskip
Applying $(\ast_\kappa)$, where $\kappa$ is the order type of $(X,c)$, to $X$
itself gives the theorem.
\end{proof}

\begin{rem}\label{rem:no-selection-comb}
The proof needs no coherent global apex-selection map.  The path and
heredity arguments use witnesses supplied by (N2), but the induction never
requires these choices to agree or transports a selected apex from one
stage to the next.  This is the form of the exchange axiom supplied by
the geometry of $\IS(K)$.
\end{rem}

\subsection{The given exchange order need not dismantle the graph}\label{ssec:not-dismantlable}

A vertex $u$ of a graph is \emph{dominated} by $w\ne u$ if $N[u]\subseteq N[w]$, and
a finite graph is \emph{dismantlable} if repeated deletion of dominated vertices
reduces it to a single vertex \cite[Definition 7.1, p.~1503]{PS12}.  Deleting a dominated
vertex is exactly deleting a vertex whose link is a cone, and it is by this
mechanism that dismantlability yields contractibility \cite{BarmakMinian}.

The supplied exchange order need not give such a deletion sequence.  The axioms
give, through Lemma~\ref{lem:heredity}(2), a descending link of diameter at most
$2$ --- contractible, by Theorem~\ref{thm:main-comb} applied to it --- and a
complex of diameter $2$ need not be a cone.

\begin{prop}\label{prop:not-cone}
There is a finite exchange complex $X$ on seven vertices, with complexity the
identity on $\{0,1,\dots,6\}$, whose descending link at the top vertex is
contractible but is not a cone.
\end{prop}

\begin{proof}
Let $X$ be the flag complex on $\{0,\dots,6\}$ generated by the hexagon
$0\text{--}1\text{--}2\text{--}3\text{--}4\text{--}5\text{--}0$, the three chords
$0\text{--}2$, $2\text{--}4$, $4\text{--}0$, and the six edges joining $6$ to
each of $0,\dots,5$.  Take $c=\mathrm{id}$.

The descending link of $6$ is the full subcomplex $L$ on $\{0,\dots,5\}$.  Its
maximal simplices are the triangles $\{0,1,2\}$, $\{2,3,4\}$, $\{4,5,0\}$ and
$\{0,2,4\}$: a triangulated hexagonal disc, hence contractible.  It is not a
cone, because a cone point would have to be adjacent to all five other vertices
of $L$, and none is: inside $L$ the vertices $1,3,5$ have exactly two
neighbours each, while $0$ misses $3$, $2$ misses $5$ and $4$ misses $1$.

The complete list of non-adjacent pairs, together with a lower apex of
their two-interval, is
\[
\begin{array}{c|cccccc}
\{u,w\}&\{0,3\}&\{1,3\}&\{1,4\}&\{1,5\}&\{2,5\}&\{3,5\}\\
x&2&2&0&0&0&4.
\end{array}
\]
In each case the common-neighbour set is a clique and contains the indicated
$x$, which is smaller than $\max\{u,w\}$.  Hence (N2), and therefore (N1), holds.

\end{proof}

\begin{rem}[What Proposition~\ref{prop:not-cone} does and does not say]
\label{rem:dismantle-scope}
The proposition shows that deleting vertices in decreasing order of the
given complexity need not be a dismantling: the first vertex to delete,
namely $6$, is not dominated.  The graph itself is dismantlable, since $6$
is adjacent to every other vertex and dominates each of them.  Thus this
example distinguishes the supplied exchange order from a dismantling order;
it does not rule out a different dismantling order, or an application of an
infinite dismantlability criterion.  The proof of
Theorem~\ref{thm:main-comb} uses contractible descending links and does not
require them to be cones.
\end{rem}

\subsection{The relation to the descending-link criterion}\label{ssec:F-discussion}

The remaining comparison is with Bestvina--Brady Morse theory in the form given
by Zaremsky, which applies to arbitrary simplicial complexes.

\begin{dfn}[{\cite[Definition 2.1]{Zar24}}]\label{dfn:morse}
A map $h\colon K\to\R$ on a simplicial complex $K$ is a \emph{Morse function}
if the image $h(K^{(0)})$ is a closed discrete subset of $\R$ and the
restriction of $h$ to any simplex takes distinct values on the vertices of that
simplex.  For $t\in\R$ let $K^{h\le t}$ denote the full subcomplex on
$\{v:h(v)\le t\}$, and let $\lk^{\downarrow}_h(v)$ denote the full subcomplex
of $\lk(v)$ on $\{u\in\lk(v):h(u)<h(v)\}$.
\end{dfn}

Since the heights are distinct on adjacent vertices, this
$\lk^{\downarrow}_h(v)$ agrees with the descending link defined using the
descending star in \cite[Definition 2.2]{Zar24}: its simplices are precisely
the simplices of $\lk(v)$ on which every vertex has height below $h(v)$.

\begin{thm}[{\cite[Lemma 2.3]{Zar24}}]\label{thm:zaremsky}
Let $h\colon K\to\R$ be a Morse function and let $s<t$ in $\R$.  If
$\lk^{\downarrow}_h(v)$ is contractible for every vertex $v$ with
$s<h(v)\le t$, then the inclusion $K^{h\le s}\hookrightarrow K^{h\le t}$ is a
homotopy equivalence.  If $\lk^{\downarrow}_h(v)$ is contractible for every
vertex $v$ with $h(v)>t$, then the inclusion $K^{h\le t}\hookrightarrow K$ is a
homotopy equivalence.
\end{thm}

The statement is relative, and necessarily so.  It concludes that an inclusion
is a homotopy equivalence, never that $K$ is contractible, and it could not
conclude the latter: for $K$ a disjoint union of two copies of the bi-infinite
line with $h$ the vertex index, every descending link is a single point while
$K$ is not even connected.  To obtain contractibility from
Theorem~\ref{thm:zaremsky} one must exhibit a level $t$ at which $K^{h\le t}$
is already contractible --- for example, the level of a unique least vertex.

There is no local finiteness hypothesis here.  Consequently the fact that
$\IS(K)$ fails to be locally finite is by itself no obstruction to applying
Theorem~\ref{thm:zaremsky}, and we make no novelty claim on that basis.

The given complexity need not take values in a closed discrete subset of
$\R$, so it cannot in general be substituted directly into this particular
version of the Morse lemma.  The issue is the chosen height, rather than
local finiteness.  We first discuss a bijective reindexing by $\N$ under a
finite-descending-link assumption.  We then remove countability from the
resulting Morse argument, and explain why the more general version in
\cite{ZarMorse} also supplies the attachment step for every countable exchange
complex, without that assumption.

\begin{dfn}\label{dfn:F}
An exchange complex $(X,c)$ satisfies \textup{(F)} if every descending link
$L_c(v)$ is finite.
\end{dfn}

Condition (F) is implied by local finiteness and by $c$ having order type at
most $\omega$ with finite fibres, but is weaker than either.

\begin{thm}[Reindexing under \textup{(F)}]\label{thm:reindex}
Let $(X,c)$ be a countably infinite exchange complex with $c$ injective and
satisfying \textup{(F)}.  Then there is a bijection $h\colon X^{(0)}\to\N$ such
that $\lk^{\downarrow}_h(v)=L_c(v)$ for every vertex $v$.
\end{thm}

\begin{proof}
Orient each edge $\{u,w\}$ of $X$ from the $c$-smaller to the $c$-larger end and
let $P$ be the reflexive transitive closure of this relation; $P$ is a partial
order because $c$ is strictly increasing along oriented edges.  Fix a vertex
$v$ and consider the tree $T_v$ of finite descending paths
$v=u_0,u_1,\dots,u_k$ with $u_{i+1}\in L_c(u_i)$.  Each node has finitely many
children by \textup{(F)}, and $T_v$ has no infinite branch because $c$ decreases
strictly along a branch and the values are well-ordered.  By K\"onig's lemma
$T_v$ is finite, so $v$ has finitely many $P$-predecessors.

A countable partial order in which every element has finitely many strict
predecessors admits a bijective linear extension onto $\N$.  To see this,
enumerate the vertices as $w_0,w_1,\ldots$.  At stage $n$, do nothing if
$w_n$ has already been assigned.  Otherwise list the as-yet unassigned
strict $P$-predecessors of $w_n$ in a linear extension of their finite
induced order, followed by $w_n$, and assign the consecutive least unused
natural numbers to this list.  Every strict predecessor of an element in
this list is either earlier in the list or already assigned.  Thus the
partial assignment remains $P$-increasing.  The vertex $w_n$ is assigned
by the end of stage $n$, so the assignment is defined on every vertex.
Infinitely many vertices receive distinct consecutive numbers with no
gaps, so its image is all of $\N$.

For a vertex $v$ and a neighbour $u$: if $c(u)<c(v)$ then $u\,P\,v$ and
$h(u)<h(v)$; if $c(u)>c(v)$ then $v\,P\,u$ and $h(u)>h(v)$.  Hence
$\lk^{\downarrow}_h(v)=L_c(v)$.
\end{proof}

\begin{cor}\label{cor:under-F}
Let $(X,c)$ be an exchange complex with $c$ injective and satisfying
\textup{(F)}.  Then $X$ is contractible, with no restriction on the
cardinality of its vertex set.  The proof uses
Theorem~\ref{thm:zaremsky}, Lemma~\ref{lem:heredity}, and induction for finite
exchange complexes; it does not use Theorem~\ref{thm:main-comb}.
\end{cor}

\begin{proof}
First consider finite exchange complexes, by induction on their number
$n$ of vertices.  The case $n=1$ is immediate.  For $n>1$, the rank of $c$
is a Morse height, its bottom sublevel is the least vertex, and each other
descending link is a non-empty exchange complex on fewer than $n$ vertices
by Lemma~\ref{lem:heredity}(2).  These links are contractible by induction,
so Theorem~\ref{thm:zaremsky} gives the result.

Now let $X$ have arbitrary cardinality.  For each vertex $v$, the tree of
strictly $c$-descending edge paths starting at $v$ is finite: it is finitely
branching by \textup{(F)} and has no infinite branch by well-foundedness,
so K\"onig's lemma applies.  Define
\[
 h(v)=\max\{k:\ v=u_0,u_1,\ldots,u_k
          \text{ is a strictly $c$-descending edge path}\}\in\N.
\]
If $u$ and $v$ are adjacent and $c(u)<c(v)$, a longest descending path
starting at $u$ extends by the initial edge from $v$ to $u$; hence
$h(v)\ge h(u)+1$.  Thus $h$ gives precisely the same orientation on every
edge as $c$, and $\lk_h^{\downarrow}(v)=L_c(v)$.  Its image is a subset of
$\N$ and its values are distinct on each simplex, so $h$ is a Morse
function after affine extension over simplices.  It need not be injective
on the entire vertex set.

Lemma~\ref{lem:heredity}(2) implies that $h(v)=0$ exactly at the least
vertex $v_0$.  Every other descending link is a non-empty finite exchange
complex, hence contractible by the finite case.  Applying
Theorem~\ref{thm:zaremsky} at height $0$ completes the proof.
\end{proof}

The preceding argument does not require the globally injective height of
Theorem~\ref{thm:reindex}.  Condition \textup{(F)} is nevertheless necessary
for that particular bijective conclusion:

\begin{prop}\label{prop:F-necessary}
The conclusion of Theorem~\ref{thm:reindex} implies \textup{(F)}.
\end{prop}
\begin{proof}
If $h\colon X^{(0)}\to\N$ is injective and preserves descending links,
then $L_c(v)$ injects into $\{0,\ldots,h(v)-1\}$ for every $v$.
\end{proof}

Thus, for a countably infinite exchange complex, \textup{(F)} characterises
the existence of a bijection to $\N$ preserving all descending links.
It is not a necessary condition for a Morse-theoretic proof.

\begin{rem}\label{rem:residue}
An infinite descending link does not force a Morse height to be unbounded
below.  Take vertices $a,v,b_0,b_1,\ldots$, with edges $av$, $ab_i$ and
$vb_i$, and order them as $a<b_0<b_1<\cdots<v$.  For the only type of
distance-two pair, $b_i,b_j$, the two-interval is the edge $\{a,v\}$ and
$a$ is a lower apex.  This is an exchange complex and $L_c(v)$ is infinite.
Nevertheless $h(a)=0$, $h(b_i)=1$ and $h(v)=2$ is a Morse height with the
same descending links and with a single-vertex bottom sublevel.
\end{rem}

There is also a general comparison for countable complexes.  The
descending-type Morse functions of \cite[Definition 1.1]{ZarMorse} allow
non-discrete heights; in particular an affine pair $(h,0)$ is allowed when
$h$ is non-constant on edges and there are no descending edge rays.
The attachment result and its connectivity consequence are
\cite[Lemma 1.8 and Corollary 1.11]{ZarMorse}.

\begin{cor}[The countable Morse formulation]\label{cor:countable-morse}
For countable exchange complexes, Theorem~\ref{thm:main-comb} can also be
deduced from \cite[Corollary 1.11]{ZarMorse},
Lemma~\ref{lem:heredity}, and transfinite induction on the vertex order type,
without assuming \textup{(F)}.
\end{cor}

\begin{proof}
After Lemma~\ref{lem:ties}, assume $c$ injective.  We argue by transfinite
induction, universally over countable exchange complexes of the given
vertex order type $\alpha$.  The one-vertex case is immediate.  For
$\alpha>1$, choose an injection $j\colon X^{(0)}\to\N$ and set
\[
   h(v)=\sum_{c(w)\le c(v)}2^{-j(w)-2}.
\]
The positive series converges and has total at most $1/2$, so
$0<h(v)<1$.  If $c(u)<c(v)$, the summation set for $u$ is a proper subset
of that for $v$ and the latter includes the positive term indexed by $v$;
therefore $h(u)<h(v)$.  Extend $h$ affinely over simplices.  A strictly
$h$-descending edge ray would be a strictly $c$-descending sequence, which
is impossible.  Consequently $(h,0)$ is a descending-type Morse function
in the stated sense, and its descending links are exactly $L_c(v)$.

For every vertex other than the least vertex $v_0$, the link $L(v)$ is a
non-empty countable exchange complex by Lemma~\ref{lem:heredity}(2).
Its inherited vertex order type is at most the index of $v$, hence
strictly less than $\alpha$.  These links are contractible by the universal
inductive hypothesis.  Since $X^{h\le h(v_0)}=\{v_0\}$ and $X^{h\le1}=X$,
\cite[Corollary 1.11]{ZarMorse}, applied for every connectivity degree,
makes the inclusion of $v_0$ a weak homotopy equivalence.  Whitehead's
theorem for CW complexes makes it a homotopy equivalence, proving the claim.
\end{proof}

This identifies an existing Morse attachment theorem in the countable
argument; it does not assert that \cite{ZarMorse} already contains the
exchange theorem.  The exchange-specific work is the heredity lemma and
the induction establishing that its descending links are contractible.
The bounded real encoding above uses countability.  An arbitrary
uncountable well-order need not admit such an encoding; nevertheless
Corollary~\ref{cor:under-F} covers every cardinality under \textup{(F)}.
We draw no conclusion here about the applicability of other Morse or
dismantlability arguments to the remaining uncountable cases.

\begin{rem}[Where our application sits]\label{rem:residue-application}
We do not know whether $\IS(K)$ satisfies \textup{(F)}.  For its complexity
$c=(g,A)$ the fibres are finite by Lemma~\ref{lem:wf}(i), so this question
does not depend on how those fibres are tie-broken.  Equivalently, for
every vertex $v$, are there only finitely many neighbours $u$ with
$c(u)\le c(v)$?  The same-genus part is finite by Lemma~\ref{lem:wf}(i),
so the question is whether there can be infinitely many neighbours of
strictly smaller genus.  This question is not needed for the Morse
attachment comparison: $\IS(K)$ is countable by Lemma~\ref{lem:wf}(iv),
and Corollary~\ref{cor:countable-morse} applies independently of
\textup{(F)}.  The direct proof of Theorem~\ref{thm:main-comb} above remains
valid without any cardinality restriction.
\end{rem}

\clearpage
\section{The Kakimizu complex, and what the method reaches}\label{sec:applications}

\subsection{Proof of Theorem A}

\begin{proof}[Proof of Theorem~A]
Let $K$ be a non-trivial knot and choose the warped-collar metric and
longitudinal foliation of \S\ref{sec:geometric-inputs}.
Theorem~\ref{thm:fixed-smooth} and the sliding-boundary argument give
\textup{(E)} and \textup{(R1)} for these data.  We check that
$\IS(K)$, with the complexity
$c(v)=(g(v),A(v))\in\N\times\R_{>0}$ of \S\ref{sec:complexity} ordered
lexicographically, is an exchange complex in the sense of
Definition~\ref{dfn:exchange}.

$\IS(K)$ is non-empty, and it is a flag complex by
Proposition~\ref{prop:flag}: a set of classes admits simultaneously pairwise
disjoint representatives as soon as it does so pairwise.  It is connected by
Kakimizu's Theorem~A \cite[p.~226]{Kak92}, in the incompressible case of that
statement, and its edge-path metric is the distance $d$
\cite[Proposition 3.1(1), p.~231]{Kak92}.

The complexity takes values in a well-ordered set: this is
Lemma~\ref{lem:wf}, which is where \textup{(E)} and the finiteness argument of
\S\ref{sec:complexity} are consumed.  Axiom (N1) then holds because $\IS(K)$ is
connected --- two vertices at distance $2$ have a common neighbour --- together
with Proposition~\ref{prop:apex}, which produces an apex of the interval from
the flagness of $\IS(K)$ and the exchange lemma.  Axiom (N2) is
Theorem~\ref{thm:exchange}, whose proof consumes \textup{(U)}, \textup{(E)},
\textup{(R1)} --- that is, Theorem~\ref{thm:R1} --- and the collar conditions
\eqref{eq:N4}, \eqref{eq:N5} of \S\ref{ssec:tangential}.

Theorem~\ref{thm:main-comb} now applies, and $\IS(K)$ is contractible.
\end{proof}

\subsection{Truncations: proof of Theorem C}

\begin{proof}[Proof of Theorem~C]
Fix the metric and foliation chosen in Theorem A and use the smooth-area
complexity of Definition~\ref{dfn:complexity}.  Its proof gives an exchange
structure on $\IS(K)$.  Break ties as in Lemma~\ref{lem:ties}; every set
downward closed for the original complexity remains downward closed after
this refinement.  Fix $\ell\ge g(K)$ and let $D_\ell=\{v:g(v)\le\ell\}$.  Since the order on
$\N\times\R_{>0}$ is lexicographic with genus in the first coordinate,
$c(u)\le c(v)$ and $g(v)\le\ell$ force $g(u)\le\ell$; so $D_\ell$ is downward
closed.  It is non-empty because $\ell\ge g(K)$.  By
Corollary~\ref{cor:initial-isometric} the full subcomplex $\IS_\ell(K)$ on
$D_\ell$ is isometrically embedded in $\IS(K)$ and connected, and by
Lemma~\ref{lem:heredity}(1) it is an exchange complex; contractibility is
Theorem~\ref{thm:main-comb}.

For the second family, fix in addition $a_0>0$ and let
$D_{\ell,a_0}=\{v:g(v)<\ell\}\cup\{v:g(v)=\ell,\ A(v)\le a_0\}$.  This is again
downward closed for the lexicographic order, and non-empty as soon as
$\ell>g(K)$, or $\ell=g(K)$ and $a_0$ is at least the least smooth relative area in
genus $g(K)$.  For each non-empty such set, Corollary~\ref{cor:initial-isometric} and
Lemma~\ref{lem:heredity}(1) again apply, followed by
Theorem~\ref{thm:main-comb}.  There is no contractibility assertion for an
empty sublevel.
\end{proof}

Taking $\ell=g(K)$ gives another proof of the conclusion of
\cite[Theorem 1.1, p.~1490]{PS12}, proved there for
the complex of spanning surfaces of minimal Thurston norm; we do not claim
that case as new.
For the truncations with $\ell>g(K)$ and the area-refined family, the
proof adds no geometric hypothesis to Theorem~\hyperref[thm:A]{A}.
Their vertex sets are downward closed for the complexity, so heredity
and the same combinatorial theorem give contractibility.

\begin{rem}\label{rem:chen-shen}
The \emph{connectedness} of $\IS_\ell(K)$ is part of
Theorem~\hyperref[thm:C]{C} and is obtained in the proof above from
Corollary~\ref{cor:initial-isometric}, hence ultimately from the connectedness
of $\IS(K)$ itself; it also has a direct source,
\cite[Proposition 5.15, p.~9 of v4]{ChenShen}, which we do not use.
\end{rem}

\subsection{Links: proof of Theorem D}\label{ssec:links}

Kakimizu's complexes are defined for links, and it is natural to ask how much of
the above survives.  The obstacle is not the combinatorics but the definition of
a vertex.

Before the definition, one point of vocabulary has to be settled, because for
$n\ge2$ the boundary slope of a spanning surface is not the preferred
longitude.  For $n\ge2$ the intersection $\lambda_i:=\Sigma\cap\bd N(L_i)$ of a
spanning surface with the $i$-th boundary torus is the
\emph{surface-framed longitude}: since $\Sigma\subset E(L)$ misses $L_i$, its
boundary bounds in the complement of $L_i$, so $\lk(\bd\Sigma,L_i)=0$ and
therefore
\begin{equation}\label{eq:sfl}
   \lk(\lambda_i,L_i)\;=\;-\sum_{j\ne i}\lk(L_j,L_i),
\end{equation}
a slope depending only on $L$.  The foliation $J_i$ of $T_i$ used in
\S\ref{sec:geometric-inputs} is taken by curves of this slope.  The total
linking number defines a homomorphism $\varphi\colon H_1(E(L))\to\Z$ sending
every meridian to $1$, and, $[c]=\sum_j\lk(c,L_j)\,[m_j]$ for a curve $c$ in
$E(L)$,
\[
   \varphi(\lambda_i)=\lk(\lambda_i,L_i)+\sum_{j\ne i}\lk(L_i,L_j)=0
\]
by \eqref{eq:sfl}.  So each $\lambda_i$ lifts to a closed curve in the infinite
cyclic cover determined by $\varphi$, and the lifting argument of
Lemma~\ref{lem:frontier} applies verbatim.

\begin{dfn}\label{dfn:link-indecomp}
A link $L=L_1\cup\dots\cup L_n$ in $S^3$ is \emph{linking-indecomposable} if for
every partition $\{1,\dots,n\}=I\sqcup I^{c}$ into two non-empty parts there is
an index $m$ with $\lk(L_m,L_J)\ne0$, where $J$ is the part not containing $m$
and $L_J=\bigcup_{j\in J}L_j$.  For $n=1$ the condition is vacuous.
\end{dfn}

\begin{lem}\label{lem:link-connected}
Let $L$ be linking-indecomposable.  Then every spanning surface for $L$ in
$E(L)$ --- that is, every properly embedded orientable surface with no closed
component whose boundary is the union of the $n$ longitudes --- is connected, and has exactly $n$
boundary curves.  Moreover the property is inherited by the surfaces produced by
the frontier construction of \S\ref{sec:exchange-lemma}.
\end{lem}

\begin{proof}
Suppose $F$ is a spanning surface with a decomposition $F=F'\sqcup F''$ into
non-empty unions of components.  A spanning surface has no closed component, so
every component of $F$ carries at least one longitude; let $I$ be the set of
indices whose longitude lies in $\bd F'$ and $I^{c}$ the set for $F''$.  These
are non-empty and partition $\{1,\dots,n\}$, the components of $F$ being
disjoint.

Write $L_I=\bigcup_{i\in I}L_i$.  The chain $F'$ lies in $E(L)$, hence misses
$L$, and $\bd F'$ is the union of the longitudes $\lambda_i$, $i\in I$, each of
which is isotopic to $L_i$ inside the solid torus $N(L_i)$ by an isotopy
disjoint from every other component.  Therefore
\[
   [L_I]=[\bd F']=0 \quad\text{in } H_1(S^3\setminus L_{I^{c}}).
\]
For $j\in I^{c}$ the linking number $\lk(L_I,L_j)$ is the image of $[L_I]$
under $H_1(S^3\setminus L_j)\cong\Z$, and the inclusion
$S^3\setminus L_{I^{c}}\hookrightarrow S^3\setminus L_j$ carries the vanishing
class to it; so $\lk(L_j,L_I)=\lk(L_I,L_j)=0$ for every $j\in I^{c}$.  The same
argument applied to $F''$ gives $\lk(L_m,L_{I^{c}})=0$ for every $m\in I$.  As
$I$ and $I^{c}$ are the two parts of a partition into non-empty sets, this says
$\lk(L_m,L_J)=0$ for \emph{every} index $m$, with $J$ the part not containing
$m$, contradicting linking-indecomposability.

Two things about this deserve emphasis, since each is a place where a shorter
argument would fail.  First, the conclusion needs \emph{both} halves: $F'$
alone yields $\lk(L_j,L_I)=0$ for $j\in I^{c}$ and nothing about the indices
$m\in I$, and it is the second application, to $F''$, that supplies those.
Only with the two together does one have $\lk(L_m,L_J)=0$ for every $m$, which
is the negation of linking-indecomposability.  Second, note which factor
carries the surface: the bounding chain $F'$ controls the
class of the \emph{union} $L_I$, and it is paired with a \emph{single}
component on the other side; the pairing in the opposite direction is not
available, since a component of $F'$ may carry several longitudes and then
bounds only their sum.

So $F$ is connected, its boundary consists of the $n$ longitudes, and there are
exactly $n$ of them.  The frontier construction outputs surfaces which are
again spanning surfaces for $L$, so the same argument applies to them.
\end{proof}

The compression step of Lemma~\ref{lem:euler} needs a companion in the link
setting, for the same reason: ``keep the component with boundary'' has to be
unambiguous.  We number it beside Lemma~\ref{lem:link-connected}, which it
uses.

\begin{lemlinkprime}[Compression of a spanning surface for a link]
Let $L$ be linking-indecomposable and let $F$ be a connected spanning surface
for $L$, with $n$ boundary curves.  Compressing $F$ along an essential circle
that separates $F$ produces two disjoint surfaces $F_1,F_2$.  Exactly one of
them carries boundary, and it carries all $n$ boundary curves; so ``keep the
component with boundary'' is unambiguous, and the genus of that component
strictly decreases.  Compression along a non-separating circle keeps the
surface connected and again lowers the genus.
\end{lemlinkprime}

\begin{proof}
If both $F_1$ and $F_2$ carried boundary curves, then $F_1\sqcup F_2$ would be
a disconnected spanning surface for $L$ --- it is properly embedded,
orientable, and its boundary is the union of the $n$ longitudes, compression
having altered $F$ only in the interior --- contradicting
Lemma~\ref{lem:link-connected}.  So exactly one side, say $F_1$, carries
boundary, and it carries all $n$ curves; $F_2$ is closed and is discarded.  The
compressing circle being essential, $F_2$ is not a sphere, so $\chi(F_2)\le0$
and $\chi(F_1)=\chi(F)+2-\chi(F_2)\ge\chi(F)+2$; with $n$ fixed and
$\chi=2-2g-n$ this makes the genus drop by at least one.  The non-separating
case is that of Lemma~\ref{lem:euler} verbatim.
\end{proof}

Lemma~\ref{lem:link-connected} is what makes the link case tractable, and it
does so by removing a discrepancy rather than by proving something about
surfaces.  Kakimizu's vertices are isotopy classes of spanning surfaces that are
allowed to be disconnected and are required only to be incompressible
componentwise; ours are connected.  On a linking-indecomposable link the two
notions coincide, and with them the two edge conventions: the convention of
\cite[(1.3)(b)]{Kak92}, that disjointness is equivalent to distance $1$, is
false for links in general precisely because of disconnected surfaces, and here
that failure cannot occur.

\begin{proof}[Proof of Theorem~D]
Let $L$ be non-split, linking-indecomposable and not the unknot.  Choose
a flat metric on each boundary torus, its foliation by surface-framed
longitudes, and the exact warped collar \eqref{eq:collar}; extend the metric
smoothly over $E(L)$.
By
Lemma~\ref{lem:link-connected} the vertex set of $\IS(L)$ is the set of isotopy
classes of connected incompressible spanning surfaces.  Connectedness of
$\IS(L)$ is Kakimizu's Theorem~A for a non-split link \cite[p.~226]{Kak92} ---
this is where non-splitness is used, through the irreducibility of $E(L)$ ---
while flagness is Proposition~\ref{prop:flag}, whose proof consumes only
\textup{(U)} and \textup{(E)} and runs unchanged once every spanning surface is
known to be connected.  The two are separate inputs and neither is a corollary
of the other.

The complexity is again $(g,A)$, now with $g$ the genus of the connected
spanning surface and $A$ the infimum over its smooth neat ambient-isotopy
representatives with one boundary leaf on each torus.
The geometric inputs are available in $n$-component form under the
hypotheses just stated.  Universal disjointness \textup{(U)} is supplied by
\cite[Theorem~4 and Theorem~A.5]{Schultens}.  Fixed-boundary attainment is
Theorem~\ref{thm:fixed-smooth}, applied to the entire prescribed tuple of
boundary leaves.  The $n$-component sliding-boundary argument following
Theorem~\ref{thm:R1} gives \textup{(E)}(ii) and \textup{(R1)}.
Kapovich's compactness theorem requires connected surfaces meeting each
boundary torus in one curve \cite[p.~897]{Schultens}.  These conditions
hold by Lemma~\ref{lem:link-connected}; the foliations are taken in the
surface-framed slopes described above.  Lemma~\ref{lem:R2} is local to each
boundary torus and gives transversality.  Thus no new existence theorem is
being inferred merely by replacing a knot with a link.
The argument of \S\ref{sec:exchange-lemma} then runs with three modifications:
\begin{enumerate}
\item \emph{The exchange lemma on $n$ boundary tori.}  Everything in
\S\S\ref{ssec:tangential}--\ref{ssec:T1} is local to a single boundary torus
apart from finitely many global constants, and the bookkeeping is as follows.

On each torus $T_i$ the curves $\bd_iA$ and $\bd_iB$ are leaves of $J_i$, hence
equal or disjoint; let $E\subseteq\{1,\dots,n\}$ be the set of indices at which
they are equal.  The trichotomy of \S\ref{ssec:tangential} becomes: $E=\emptyset$
and $A\pitchfork B$, which is \textup{(T0)}; $E=\emptyset$ with interior
tangencies, which is \textup{(T1)} and needs no change whatever, the proof
there never mentioning the boundary except through \eqref{eq:d0}; and
$E\ne\emptyset$, which is \textup{(T2)}.

In the last case Proposition~\ref{prop:collar}(1) is unchanged, its proof being
local at an interior tangency point.  Clause (2) is applied on each $T_i$ with
$i\in E$ separately, in the collar chart $(s,r_i,y)$ of that torus, the collars
being taken disjoint; it yields $r_0^{(i)}>0$ and
$\mathcal{T}\subseteq\bigcap_{i\in E}\{r_i\ge r_0^{(i)}\}$, and clause (3) then
gives $\mathcal{T}$ finite as before.  Lemma~\ref{lem:bdry-sign} gives on each
such torus a sign $\varsigma_i$, an arc $I_i\subset L_i$ and constants
$a_0^{(i)},r_1^{(i)}$.  The perturbation \eqref{eq:perturb} becomes
\[
   A_t=\Bigl\{\,y=f+t\sum_{i\in E}\varsigma_i\,\chi_i(r_i)
        \;-\;t\sum_{p\in\mathcal{T}}\psi_p\,\Bigr\},
\]
with $\chi_i$ supported in the collar of $T_i$; on a torus with $i\notin E$
nothing is done, $\bd_iA_t=\bd_iA$ being already disjoint from $\bd_iB$.  The
sign of the interior bumps is immaterial, \eqref{eq:N5} being a statement about
both signs at once, and that is what lets the boundary signs be chosen
independently torus by torus.

Lemma~\ref{lem:no-small} and Lemma~\ref{lem:collar-empty}(i),(ii) are proved
one collar at a time.  In Lemma~\ref{lem:collar-empty}(iii) the one genuinely
new point is that a closed component $\gamma$ of $\Gamma_t$ may visit several
collars, so the three cases must be read per torus.  If $\gamma$ meets no set
$\{r_i<r_1^{(i)}\}$, Case A applies and gives
$\mathrm{length}(\gamma\cap G)\ge3\rho_0/8$.  Otherwise fix an $i$ with
$\gamma\cap\{r_i<r_1^{(i)}\}\ne\emptyset$.  If $\gamma$ reaches
a level $r_i>r_0^{(i)}/2$ before leaving that collar, Case B applies
for that torus; in particular this includes every curve which leaves
the collar.  Otherwise $\gamma$ is contained in
$\{r_i\le r_0^{(i)}/2\}$ and Case C applies.  This also classifies a
curve which stays in the collar but rises above the half-depth level.  Both cases use only the behaviour of $\gamma$ inside that one collar,
so the per-torus statement suffices, with $c_\star$ the minimum of $3\rho_0/8$,
of the numbers $r_0^{(i)}/2$ and of the arclengths $\ell(I_i)$, $i\in E$,
divided by the largest of the $n$ bi-Lipschitz constants of the collar charts.
For Corollary~\ref{cor:uniform}, use the compact transverse set outside
the interior tangency balls and the smaller collar bands
$\{r_i<r_1^{(i)}/4\}$, $i\in E$.  On the remaining boundary tori the
two boundary curves are disjoint, so compactness supplies a positive
boundary clearance for their nearby intersections.  The finitely many
buffered charts give common lower bounds for angle, tube radius and
boundary clearance, and a common upper bound for the local constant $K$
of Lemma~\ref{lem:round}.  Thus its width bounds are uniform on the good
region, without using a derivative supremum near any excluded tangency.

For the smooth width construction, multiply the interior-ball cut-offs by
one collar cut-off for each $i\in E$, zero when
$r_i\le5r_1^{(i)}/8$ and one when $r_i\ge r_1^{(i)}$, extended by one
outside that collar.  Their product $\zeta$ is smooth, equals one on the
region where the preceding length estimate is collected, and has support
in the buffered good region.  The finite product has a fixed derivative
bound $L_\zeta$.  Lemma~\ref{lem:plateau-width} now gives
$\eta_t=\eta_2(t)+(\eta_1-\eta_2(t))\zeta|_{\Gamma_t}$, with $\eta_1>0$
fixed before $t$ and with $\eta_2(t)>0$ chosen afterwards.  On a band in
one collar all the other collar factors equal one, since the collars are
disjoint; hence each of the three length cases above supplies gain on
this same plateau.  Equation~\eqref{eq:epsstar} and the no-disc and area
arguments therefore apply with one fixed positive gain.  The
positive-width isotopies in Lemma~\ref{lem:frontier}(2) also work
simultaneously at the finitely many boundary tori, so the output vertices
remain fixed when both rounding and push-off scales are reduced for a
subsequently specified neighbour.

Lemma~\ref{lem:no-disc} needs no change: of $A$ and $B$ its proof uses only
that each is connected, which is Lemma~\ref{lem:link-connected}, that each has
non-empty boundary on $\bd E(L)$, and that $E(L)$ is irreducible.  In
Lemma~\ref{lem:frontier} the boundary count is made torus by torus: over $T_i$
the lifted boundary is an annulus containing the two disjoint circles
$\bd_iS^+_0$ and $\bd_iS^-_0$, one above the other, and the argument given
there sends them to opposite frontiers.  Which frontier receives $\bd_iS^+_0$
may depend on $i$, and nothing requires it not to; either way each frontier
meets each $T_i$ in exactly one curve, hence is a spanning surface for $L$
once its closed components are discarded, hence is connected by
Lemma~\ref{lem:link-connected} and carries all $n$ boundary curves.
\item Lemma~\ref{lem:link-connected} is consumed at four places, and it is
worth listing them rather than gesturing at them:
\begin{enumerate}
\item[\textup{(i)}] in Proposition~\ref{prop:disjoint}, where the alternative
``coincide'' is turned into $x=u$ --- a step that needs the surface to be
connected, since otherwise two distinct vertices may share a component;
\item[\textup{(ii)}] in Lemma~\ref{lem:frontier}, where each frontier output
must have exactly one component with boundary and that component must carry all
$n$ boundary curves;
\item[\textup{(iii)}] in Lemma~8.3$'$, where a separating compression is not
allowed to separate the boundary curves, so that ``keep the component with
boundary'' names a single surface;
\item[\textup{(iv)}] in aligning our vertex set and edge convention with
Kakimizu's, as explained after Lemma~\ref{lem:link-connected}.
\end{enumerate}
\item The knot argument invokes, through Lemma~\ref{lem:frontier}, the
assertion that a Seifert surface represents a primitive generator of
$H_2(E(K),\bd E(K))$.  For a link $H_2(E(L),\bd E(L))\cong\Z^{n}$ and the
assertion is \emph{false} as stated.  It is replaced by the following, which is
what the argument actually needs: the total linking number defines a
homomorphism $\pi_1(E(L))\to\Z$, and a spanning surface for $L$ is Poincar\'e
dual to a primitive class, namely to the image of the generator under the
induced map on $H^1$.  Primitivity holds because the homomorphism is onto, each
meridian mapping to $1$.
\end{enumerate}
Axiom (N1) and axiom (N2) are then obtained exactly as in the proof of
Theorem~\hyperref[thm:A]{A}, and Theorem~\ref{thm:main-comb} gives the result.
\end{proof}

\begin{rem}[The hypothesis is load-bearing]\label{rem:indecomp-needed}
Linking-indecomposability is not a convenience.  It is used in
Lemma~\ref{lem:link-connected}, and through it at the four points listed in the
proof of Theorem~\hyperref[thm:D]{D}; removing it breaks each of them
separately.  The last of the four is what aligns the two vertex definitions and
the two edge conventions.  For
connected sums, the structure of $\IS$ is governed by
\cite{Banks} rather than by anything proved here.
\end{rem}

\subsection{Where the method stops}\label{ssec:general-links}

\emph{General links.}  For a link which is not linking-indecomposable the
failing step is not the connectedness of $\IS(L)$, which Kakimizu proves for
every non-split link, nor the combinatorics of
\S\S\ref{sec:exchange-complexes}--\ref{sec:contractibility}.  It is the
hypothesis $S\cap(S_+\cup S_-)=\emptyset$ of Lemma~\ref{lem:frontier}, supplied
for knots by \textup{(U)} through Proposition~\ref{prop:disjoint}.  For a link,
two distinct vertices may share the isotopy class of one component of the
surface; the minimisers then \emph{coincide} on that component rather than
being disjoint from it.  These inputs therefore no longer justify the
opening sentence of the proof of Theorem~\ref{thm:exchange}.  This is an obstruction to \emph{this
proof} only: we have constructed no example realising the failure, and we have
not checked whether the appendix of \cite{Schultens} yields the stronger
dichotomy ``disjoint or identical'', which would repair the step.

\emph{Equivariance.}\label{ssec:equivariant}  Let $G$ be a finite group
of symmetries of $K$ acting on $\IS(K)$.  The analogous question concerns
the geometric fixed-point set $|\IS(K)|^G$; \cite[\S8]{PS12} treats
symmetries in the minimal-genus setting.  Our argument does not establish
contractibility of this fixed-point set.  First, the chosen metric and
longitudinal foliation need not be $G$-invariant.  An equivariant argument
would require invariant geometric data within the flat-torus warped-collar
scope of Theorem~\ref{thm:fixed-smooth}.

Even if such invariant geometric data are supplied,
an additional argument is needed.  For two $G$-fixed vertices at distance
$2$, the least-complexity apices form a finite $G$-invariant clique
(finiteness follows from Lemma~\ref{lem:wf}(i)).  Flagness makes it a
simplex.  Its vertices need not be fixed individually, but its barycentre
\emph{is} a fixed point in the geometric realisation.  Consequently the
absence of a fixed vertex is not an obstruction to the existence of a
geometric fixed point.  A proof for the fixed-point set, or for the
corresponding subcomplex after barycentric subdivision, would have to
establish the required connectivity and exchange properties there; this
paper supplies no such proof.

Contractibility of the ambient complex alone is insufficient: Floyd and
Richardson \cite[\S3, pp.~74--75]{FR59} construct a simplicial $A_5$-action
without fixed points on a finite complex which is acyclic and simply
connected, hence contractible by Hurewicz and Whitehead; see also
\cite{Oliver}.  This general example does not rule out a special theorem
for $\IS(K)$, and we claim no impossibility of an equivariant extension.

\emph{Haken manifolds.}\label{ssec:haken}  For a Haken manifold $M$ with a
boundary pattern one may define $\IS(M,\gamma,\alpha)$ in a similar way.
The abstract combinatorial theorem applies if its exchange axioms and
connectedness are verified.  This paper does not supply that verification:
in particular, it does not establish the required connectedness.
A general boundary pattern need not satisfy the flat-torus collar and
one-boundary-curve-per-torus hypotheses of
Theorem~\ref{thm:fixed-smooth}.  Even when these geometric hypotheses hold,
the connectedness and exchange inputs still require verification.
Kakimizu's connectedness proof used here is for links in $S^3$.  We make no
extension claim without those inputs and the hypotheses of the geometric
lemmas.


\begin{thebibliography}{99}

\bibitem{ArmstrongCollars}
M.~A. Armstrong,
\emph{Collars and concordances of topological manifolds},
Comment. Math. Helv. \textbf{45} (1970), 119--128.
Theorem~2, pp.~123--124, is the collar uniqueness theorem used here.

\bibitem{Banks}
J.~E. Banks,
\emph{The Kakimizu complex of a connected sum of links},
Trans. Amer. Math. Soc. \textbf{365} (2013), no.~11, 6017--6036;
\texttt{arXiv:1109.0965}.

\bibitem{BarmakMinian}
J.~A. Barmak and E.~G. Minian,
\emph{Strong homotopy types, nerves and collapses},
Discrete Comput. Geom. \textbf{47} (2012), no.~2, 301--328.

\bibitem{CavaliereTransirico}
P.~Cavaliere and M.~Transirico,
\emph{The Dirichlet problem for elliptic equations in the plane},
Comment. Math. Univ. Carolin. \textbf{46} (2005), no.~4, 751--758.
Lemma~3.4, pp.~756--757, with the hypotheses on pp.~753 and~755,
supplies the planar nondivergence Dirichlet $W^{2,p}$ theorem used here.

\bibitem{CCHO}
J.~Chalopin, V.~Chepoi, H.~Hirai and D.~Osajda,
\emph{Weakly modular graphs and nonpositive curvature},
Mem. Amer. Math. Soc. \textbf{268} (2020), no.~1309.
Lemma~9.13 is cited by its number and page in the Memoir; it carries the same
number in \texttt{arXiv:1409.3892v4}.

\bibitem{ChenShen}
X.~Chen and W.~Shen,
\emph{A linear bound on the diameter of the Kakimizu complex for hyperbolic
knots},
preprint, \texttt{arXiv:2508.03353v4}.
The version matters: v1 does not contain $\IS_\ell$.

\bibitem{CerfEmbeddings}
J.~Cerf,
\emph{Topologie de certains espaces de plongements},
Bull. Soc. Math. France \textbf{89} (1961), 227--380.
The comparison in III, \S3.2.1, Theorem~8 and Corollary~3,
pp.~365--367, is conditional on the Smale conjecture; every use here includes
Hatcher's theorem \cite{HatcherSmale}.  The interface in that comparison is
the separating interface of a regular decomposition.

\bibitem{FarbMargalit}
B.~Farb and D.~Margalit,
\emph{A Primer on Mapping Class Groups},
Princeton Mathematical Series \textbf{49},
Princeton University Press, Princeton, NJ, 2012.
Theorem~1.13, p.~42, concerns replacing a surface homeomorphism by a
diffeomorphism in its isotopy class.

\bibitem{FR59}
E.~E. Floyd and R.~W. Richardson,
\emph{An action of a finite group on an $n$-cell without stationary points},
Bull. Amer. Math. Soc. \textbf{65} (1959), 73--76.

\bibitem{FHS}
M.~Freedman, J.~Hass and P.~Scott,
\emph{Least area incompressible surfaces in $3$-manifolds},
Invent. Math. \textbf{71} (1983), no.~3, 609--642.

\bibitem{GT}
D.~Gilbarg and N.~S. Trudinger,
\emph{Elliptic Partial Differential Equations of Second Order},
2nd ed., Grundlehren der mathematischen Wissenschaften \textbf{224},
Springer, Berlin, 1983.
All numbers and pages cited here refer to this edition; Chapters~8 and~9 differ
from the 1977 first edition.

\bibitem{HassScott}
J.~Hass and P.~Scott,
\emph{The existence of least area surfaces in $3$-manifolds},
Trans. Amer. Math. Soc. \textbf{310} (1988), no.~1, 87--114.

\bibitem{HamstromHomeotopy}
M.-E.~Hamstrom,
\emph{Homotopy in homeomorphism spaces, TOP and PL},
Bull. Amer. Math. Soc. \textbf{80} (1974), no.~2, 207--230.
Theorem~1.5.2, p.~217, gives the homotopy groups of the identity component
of the homeomorphism group of the torus; only its fundamental group is used.

\bibitem{Hatcher}
A.~Hatcher,
\emph{Algebraic Topology},
Cambridge University Press, Cambridge, 2002.

\bibitem{Hatcher3M}
A.~Hatcher,
\emph{Notes on Basic $3$-Manifold Topology}, online notes,
\url{https://pi.math.cornell.edu/~hatcher/3M/3M.pdf}.
Corollary~3.3, p.~48, is the Loop Theorem criterion for two-sided surfaces.

\bibitem{HatcherSmale}
A.~E. Hatcher,
\emph{A proof of the Smale conjecture, $\operatorname{Diff}(S^3)\simeq O(4)$},
Ann. of Math. (2) \textbf{117} (1983), no.~3, 553--607.

\bibitem{HvdM}
S.~Hildebrandt and H.~von der Mosel,
\emph{Conformal representation of surfaces, and Plateau's problem for Cartan
functionals},
Riv. Mat. Univ. Parma (7) \textbf{4}$^{*}$ (2005), 1--43.
Theorem~4.1 is on p.~18.

\bibitem{Lickorish}
W.~B.~R. Lickorish,
\emph{An Introduction to Knot Theory},
Graduate Texts in Mathematics, vol.~175,
Springer-Verlag, New York, 1997.

\bibitem{Kak92}
O.~Kakimizu,
\emph{Finding disjoint incompressible spanning surfaces for a link},
Hiroshima Math. J. \textbf{22} (1992), no.~2, 225--236.



\bibitem{Oliver}
R.~Oliver,
\emph{Fixed-point sets of group actions on finite acyclic complexes},
Comment. Math. Helv. \textbf{50} (1975), 155--177.

\bibitem{PS12}
P.~Przytycki and J.~Schultens,
\emph{Contractibility of the Kakimizu complex and symmetric Seifert surfaces},
Trans. Amer. Math. Soc. \textbf{364} (2012), no.~3, 1489--1508.

\bibitem{Schultens}
J.~Schultens,
\emph{The Kakimizu complex is simply connected},
with an appendix by M.~Kapovich,
J. Topol. \textbf{3} (2010), no.~4, 883--900.

\bibitem{Vekua}
I.~N. Vekua,
\emph{Generalized Analytic Functions},
Pergamon Press, Oxford, 1962.

\bibitem{Wald68}
F.~Waldhausen,
\emph{On irreducible $3$-manifolds which are sufficiently large},
Ann. of Math. (2) \textbf{87} (1968), 56--88.
Cited only in Remark~\ref{rem:alt-route} and in
\S\ref{ssec:definitions-used}, for the statement that homotopic two-sided
incompressible and $\bd$-incompressible properly embedded surfaces in an
irreducible, $\bd$-irreducible $3$-manifold are ambient isotopic.  No internal
number is quoted: we have not checked one against the source, and nothing in
this paper rests on the statement.

\bibitem{Wall2016}
C.~T.~C. Wall,
\emph{Differential Topology},
Cambridge University Press, Cambridge, 2016.

\bibitem{ZarMorse}
M.~C.~B.~Zaremsky,
\emph{Bestvina--Brady discrete Morse theory and Vietoris--Rips complexes},
Amer. J. Math. \textbf{144} (2022), no.~5, 1177--1200.
\texttt{doi:10.1353/ajm.2022.0026}.

\bibitem{Zar24}
M.~C.~B. Zaremsky,
\emph{Contractible Vietoris--Rips complexes of $\mathbb{Z}^n$},
Proc. Amer. Math. Soc. \textbf{154} (2026), no.~2, 503--508,
\texttt{doi:10.1090/proc/17468};
\texttt{arXiv:2410.11993}.
Definition~2.1, Definition~2.2 and Lemma~2.3 carry the same numbers in the
arXiv version and in the published one.

\end{thebibliography}
\end{document}